\documentclass[12pt]{amsart}

\newtheorem{theorem}{Theorem}[section]
\newtheorem{lemma}{Lemma}[section]
\newtheorem{proposition}{Proposition}[section]
\newtheorem{corollary}{Corollary}[section]

\theoremstyle{definition}
\newtheorem{definition}{Definition}[section]

\theoremstyle{remark}
\newtheorem{remark}{Remark}[section]
\newtheorem{example}{Example}[section]

\numberwithin{equation}{section}

\usepackage[margin=2.4cm]{geometry}

\newcommand{\R}{\mathbf{R}}

\newcommand{\bS}{\mathbf{S}}
\newcommand{\bH}{\mathbf{H}}

\newcommand{\cH}{\mathcal{H}}
\newcommand{\cL}{\mathcal{L}}

\newcommand{\eps}{\varepsilon}
\newcommand{\loc}{\textrm{loc}}
\newcommand{\uloc}{\textrm{uloc}}

\DeclareMathOperator{\supp}{\operatorname{supp}}

\DeclareMathOperator{\inj}{\operatorname{inj}}

\DeclareMathOperator{\Ric}{\operatorname{Ric}}
\DeclareMathOperator{\Vol}{\operatorname{Vol}}
\DeclareMathOperator{\Area}{\operatorname{Area}}
\DeclareMathOperator{\Width}{\operatorname{Width}}
\DeclareMathOperator{\Dis}{\operatorname{Dis}}
\DeclareMathOperator{\Pac}{\operatorname{Pac}}

\def\polhk#1{\setbox0=\hbox{#1}{\ooalign{\hidewidth
    \lower1.5ex\hbox{`}\hidewidth\crcr\unhbox0}}}

\begin{document}

\title[Solvability of a semilinear heat equation]
{Solvability of a semilinear heat equation\\ on Riemannian manifolds}

\author[J. Takahashi]{Jin Takahashi}
\address{Department of Mathematical and Computing Science, 
Tokyo Institute of Technology, 2-12-1 Ookayama, Meguro-ku, Tokyo 152-8552, Japan}
\email{takahashi@c.titech.ac.jp}
\thanks{The first author was supported in part 
by JSPS KAKENHI Grant Numbers  JP19K14567 and JP22H01131. 
The second author was supported in part by 
JSPS KAKENHI Grant Numbers JP18K13415 and JP22K13909.}

\author[H. Yamamoto]{Hikaru Yamamoto}
\address{Department of Mathematics, Faculty of Pure and Applied Science, 
University of Tsukuba, 1-1-1 Tennodai, Tsukuba, Ibaraki 305-8571, Japan}
\email{hyamamoto@math.tsukuba.ac.jp}



\begin{abstract}
We study the solvability of the initial value problem 
for the semilinear heat equation $u_t-\Delta u=u^p$ 
in a Riemannian manifold $M$ 
with a nonnegative Radon measure $\mu$ on $M$ as initial data. 
We give sharp conditions on the local-in-time solvability of the problem 
for complete and connected $M$ with positive injectivity radius 
and bounded sectional curvature. 
\end{abstract}

\maketitle

\tableofcontents

\section{Introduction}
The study of nonlinear parabolic equations on Riemannian manifolds 
is important, since it has many applications in geometry and other areas. 
The existence, nonexistence and behavior of the solution usually reflect the geometric nature of the underlying Riemannian manifold. 
As the simplest possible model case, 
we study the following semilinear heat equation: 
\begin{equation}\label{eq:main}
	\left\{
	\begin{aligned}
	&u_t-\Delta u=u^p  &&\mbox{ in }M\times(0,T), \\
	&u(\,\cdot\,,0)=\mu  &&\mbox{ in }M.  \\
	\end{aligned}
	\right.
\end{equation}
Here, $p>1$, $T>0$, $(M,g)$ is a 
possibly non-compact Riemannian $N$-manifold 
with bounded sectional curvature 
and $\mu$ is a nonnegative Radon measure on $M$. 
Our purpose in this paper is to give sharp conditions 
on the local-in-time solvability of \eqref{eq:main}. 

We first recall known results for the necessary conditions 
and the sufficient conditions on $\mu$ 
concerning the existence of solutions 
of \eqref{eq:main} on $M=\R^N$. 
For necessary conditions, Baras and Pierre \cite{BP85} proved that 
if $u\geq0$ satisfies $u_t-\Delta u=u^p$ in $\R^N\times(0,T)$, 
then there exists an initial trace $\mu$ uniquely determined by $u$. 
Their result particularly shows that 
the initial data of nonnegative solutions
are nonnegative Radon measures. 
They \cite{BP85}  also gave necessary conditions on $\mu$ 
for the existence of solutions of \eqref{eq:main} 
and their conditions were refined by 
Hisa and Ishige \cite[Theorem 1.1]{HI18} as follows. 
We write 
\[
	p_F:= \frac{N+2}{N}
\] 
and $B(x,\rho):=\{y\in\R^N; |x-y|<\rho\}$ 
for $x\in\R^N$ and $\rho>0$.
There exists a constant $C>0$ depending only on $N$ and $p$ 
such that $\mu$ must satisfy the following (i) and (ii): 
\begin{itemize}
\item[(i)]
If $p<p_F$, then 
$\displaystyle \sup_{z\in \R^N} \mu( B(z,\sqrt{T}) )\leq C T^{\frac{N}{2}-\frac{1}{p-1}}$. 
\item[(ii)]
If $p\geq p_F$, then for any $0<\rho< \sqrt{T}$, 
\[
	\sup_{z\in \R^N} \mu ( B(z,\rho) ) 
	\leq \left\{ 
	\begin{aligned}
	&C ( \log (e+\sqrt{T}\rho^{-1}) )^{-\frac{N}{2}} &&\mbox{ if }p= p_F, \\
	&C \rho^{N-\frac{2}{p-1}} &&\mbox{ if }p>p_F. 
	\end{aligned}
	\right.
\]
\end{itemize}
Moreover, by \cite[Theorem 1.3]{HI18}, 
(i) with $C$ small is also a sufficient condition. 
By \cite[Corollary 1.2]{HI18}, (ii) is optimal 
in the sense that there exist measures $\mu$ with the same growth rate 
as in the right-hand side 
such that \eqref{eq:main} admits local-in-time solutions.

For sufficient conditions on $M=\R^N$, 
if $\mu$ is an $L^\infty$-function, 
it is easy to show the existence of 
a local-in-time solution of \eqref{eq:main}. 
In the case where $\mu$ is an $L^q$-function, 
the study of sufficient conditions dates back to Weissler \cite{We79,We80}. 
For the progress with functions as initial data, 
we refer recent papers \cite{FI18,GM20pre,HI18,IKO20,LRSV16,Mi20}, 
a book \cite[Section 15]{QS19b} and the references therein. 
For sufficient conditions with measures as initial data, 
we refer \cite{AQ04,An97,AD91,BP85,Kk89,KY94,Niw86,SL10,Um21}. 
The first author \cite[Theorem 1]{Ta16} proved that 
the above (ii) is not a sufficient condition. 
More precisely, 
\eqref{eq:main} does not admit any local-in-time nonnegative solutions
for some $\mu$ satisfying 
\[
	\sup_{z\in\R^N}\mu (B(z,\rho) )
	\leq C\rho^{N-\frac{2}{p-1}}(\log(e+\rho^{-1}))^{-\frac{1}{p-1}}
	\quad \mbox{ if }p\geq p_F
\]
for any $\rho>0$ with a constant $C>0$. 
The first author \cite[Theorem 2]{Ta16} also proved that 
\eqref{eq:main} admits a local-in-time solution if 
\[
	\sup_{z\in\R^N}\mu (B(z,\rho))
	\leq C\rho^{N-\frac{2}{p-1}}(\log(e+\rho^{-1}))^{-\frac{1}{p-1}-\varepsilon}
	\quad \mbox{ if }p\geq p_F
\]
for any $\rho>0$ with constants $C>0$ and $0<\varepsilon<1/(p-1)$.

We next consider the solvability of \eqref{eq:main} on a Riemannian manifold $M$. 
In \cite[Theorems 4.2, 4.4]{Pu11}, Punzo generalized 
the result of Weissler \cite{We80} to bounded open subsets of $\bH^N$. 
Punzo \cite{Pu12o,Pu14} also gave analogous results of \cite{Pu11} 
for open subsets of $\bS^N$ 
and manifolds with negative sectional curvature, respectively. 
For related results, see 
\cite{BPT11,GGPpre,GSX20,MMP17,PS19,Pu12b,Pu19,Zh99}. 
It seems to the authors that, even for $\bS^N$ and $\bH^N$, 
the existence and estimates of initial traces and 
the conditions on solvability 
are not as well studied as the Euclidean case.

Intuitively, the local-in-time solvability of semilinear heat equations 
is determined by the local structure of the ambient space and of the initial data, 
and so one expects that the above results with measures on $\R^N$ 
as initial data also hold for the problem on $M$. 
In this paper, our main results show that the expectation is true
under appropriate assumptions on $M$ 
involving $\R^N$, $\bS^N$, $\bH^N$ and 
more general manifolds with 
positive injectivity radius and 
bounded sectional curvature. 

We summarize our main results. 
In what follows, we denote by $d=d(x,y)$ the distance function on $M$. 
For $x\in M$ and $\rho>0$, we write $B(x,\rho):= \{y\in M  ; d(x,y)<\rho\}$. 
In Theorem \ref{th:nec}, 
we show the existence and uniqueness of an initial trace $\mu$. 
We also show that if \eqref{eq:main} admits a solution, then $\mu$ must satisfy 
\begin{align}
	&\sup_{z\in M} \mu( B(z,\rho_T) )\leq C \rho_T^{N-\frac{2}{p-1}}
	\quad \mbox{ if }p<p_F, \label{eq:subcrinec} \\
	& \label{eq:crisupnec}
	\sup_{z\in M} \mu ( B(z,\rho) ) \leq
	\left\{ 
	\begin{aligned}
	&C ( \log (e+\rho_T\rho^{-1}) )^{-\frac{N}{2}}
	&& \mbox{ if }p=p_F, \\
	&C \rho^{N-\frac{2}{p-1}}
	&&\mbox{ if }p>p_F,
	\end{aligned}
	\right.
\end{align}
for any $0<\rho< \rho_T$. 
Here $\rho_T$ is an explicit constant determined by $M$ and $T$, 
see \eqref{eq:rho0def} for the definition.

Theorem \ref{th:necsha} shows that 
if $\mu$ satisfies \eqref{eq:subcrinec} with $C$ small, 
then \eqref{eq:main} admits a solution. 
We note that we can also construct solutions
if $\mu$ satisfies \eqref{eq:subcrinec} with $C$ large and 
$T>0$ small. 
Thus \eqref{eq:subcrinec} is a necessary and sufficient condition for existence.

In Theorem \ref{th:necshacri}, we prove that 
there exists $\mu$ satisfying 
\[
\left\{ 
\begin{aligned}
	&C_1 \rho^{N-\frac{2}{p-1}}
	\leq \sup_{z\in M} \mu ( B(z,\rho) ) \leq
	C_2 \rho^{N-\frac{2}{p-1}} 
	&&\mbox{ if } p>p_F, \\
	& C_1 ( \log (e+\rho_T\rho^{-1}) )^{- \frac{N}{2}}
	\leq \sup_{z\in M} \mu ( B(z,\rho) ) \leq
	C_2 ( \log (e+\rho_T\rho^{-1}) )^{- \frac{N}{2}} 
		&&\mbox{ if } p=p_F, \\
\end{aligned}
\right. 
\]
for any $0<\rho<\rho_T$ with small constants $C_1,C_2>0$ 
such that the problem \eqref{eq:main} admits a solution. 
Hence the condition \eqref{eq:crisupnec} is sharp. 
We remark that $\mu$ in Theorem \ref{th:necshacri} 
can be written as $d\mu=u_0dx$ with a suitable function $u_0$.

As proved in Theorem \ref{th:necsha}, 
the necessary condition \eqref{eq:subcrinec} for $p<p_F$
is also a sufficient condition for the existence of local-in-time solutions. 
However, for $p\geq p_F$, 
Theorem \ref{th:nex} shows that
the necessary condition \eqref{eq:crisupnec}, 
with $C$ replaced by a small constant, is not a sufficient condition.  
In the case $p\geq p_F$, for any constant $C>0$, 
there exists a nonnegative Radon measure $\mu$ satisfying 
\[
	\sup_{z\in M} \mu( B(z,\rho) ) \leq 
	C \rho^{N-\frac{2}{p-1}} 
	( \log(e+\rho^{-1}) )^{-\frac{1}{p-1}}
\]
for any $0<\rho< \rho_{\infty}$ such that 
the problem \eqref{eq:main} does not admit any solutions in $M\times[0,T)$ 
for all $T>0$. 
Here $\rho_\infty$ is an explicit constant determined by $M$, 
see \eqref{eq:rhopr0def} for the definition.

For existence in the case $p\geq p_F$, 
Theorem \ref{th:suf} shows that if $\mu$ satisfies 
\[
	\sup_{z\in M} \mu( B(z,\rho) )\leq C \rho^{N-\frac{2}{p-1}} 
	( \log (e+\rho_T \rho^{-1}) )^{-\frac{1}{p-1} -\eps}
\]
for any $0<\rho< \rho_T$ with small constants $C$ and $\eps$, 
then the problem \eqref{eq:main} admits a solution.

As corollaries, we consider the solvability 
of \eqref{eq:main} with $\mu$ replaced by a function $u_0$. 
Corollaries \ref{cor:ex} and \ref{cor:nex} give 
solvability results with $u_0$ in the uniformly local Lebesgue spaces. 
Fix $z_0 \in M$. 
Let $u_0$ satisfy 
\begin{equation}\label{eq:u0pw}
	u_0(x)= \left\{ 
	\begin{aligned}
	& C d(z_0,x)^{-\frac{2}{p-1}} 
	&&\mbox{ if }p>p_F, \\
	& C d(z_0,x)^{-N} 
	( \log (e + d(z_0,x)^{-1} ))^{-\frac{N}{2}-1} 
	&&\mbox{ if }p=p_F,
	\end{aligned}
	\right.
\end{equation}
for any $x\in B(z_0,\rho_\infty)$ with some constant $C>0$ 
and $u_0\in L^\infty(M\setminus B(z_0,\rho_\infty))$ 
if $M\setminus B(z_0,\rho_\infty)\neq \emptyset$. 
Then Corollaries \ref{cor:singnex} and \ref{cor:singex} 
imply that there exists a constant $C_*>0$ satisfying the following: 
If $C< C_*$, then a local-in-time solution starting from 
$u_0$ with \eqref{eq:u0pw} exists, 
and if $C> C_*$, then local-in-time solutions do not exist. 
We note that, up to the coefficients, 
the strongest singularity of $u_0$ for existence 
is $d(z_0,\,\cdot\,)^{-2/(p-1)}$ if $p>p_F$ and 
$d(z_0,\,\cdot\,)^{-N} ( \log (e + d(z_0,\,\cdot\,)^{-1} ))^{-N/2-1}$ if $p=p_F$. 
We also note that the authors do not know the explicit representation of 
$C_*$ and that specifying the representation is 
an open problem even for $M=\R^N$, 
see Souplet and Weissler \cite[Theorem 1]{SW03} and 
Hisa and Ishige \cite[Corollary 1.2]{HI18}.

This paper is organized as follows. 
In Section \ref{sec:main}, we state assumptions on $M$ 
and give the exact statements of main theorems and their corollaries. 
In Section \ref{sec:nec}, we study the necessary conditions on $\mu$ 
for the solvability of \eqref{eq:main} and prove Theorem \ref{th:nec}. 
In Section \ref{sec:sharp}, to see the sharpness of the necessary conditions, 
we construct solutions and prove Theorems \ref{th:necsha} and \ref{th:necshacri}. 
We also prove Theorem \ref{th:suf} in this section. 
In Section \ref{sec:nonex}, we consider the nonexistence of solutions 
and prove Theorem \ref{th:nex}. 
In Section \ref{sec:poc}, we show corollaries stated in Section \ref{sec:main}. 
In Appendix \ref{sec:app-comp}, we give some comparison theorems adjusted 
to the use in this paper. 
In Appendix \ref{sec:appa}, 
we give covering theorems used in Section \ref{sec:nec}.

\section{Main results}\label{sec:main}
We list assumptions and notation in Subsection \ref{sec:nota}. 
In Subsection \ref{sec:state}, we state main theorems rigorously.  
In Subsection \ref{sec:cor}, we summarize corollaries of the theorems.

\subsection{Notation}\label{sec:nota}
Let $N\geq1$. 
We always assume the following conditions: 
\[
\left\{ 
\begin{aligned}
&\mbox{$(M,g)$ is a connected and complete Riemannian $N$-manifold without boundary. } \\
&\mbox{The injectivity radius satisfies $0<\inj(M)\leq \infty$}. \\
&\mbox{The sectional curvature satisfies $|\sec(M)| \leq \kappa$
for some $0\leq \kappa<\infty$  when $N\geq 2$. }
\end{aligned}
\right.
\]
In the context of sufficient conditions for the existence of solutions, 
we may handle the problem \eqref{eq:main} in $\Omega$ with 
$u=0$ on $\partial \Omega$, where $\Omega$ is a domain in $M$. 
For more details, see Remark \ref{rem:Di}. 
For $0<T\leq \infty$, we set 
\begin{equation}\label{eq:rho0def}
	\rho_T:= 
	\min\left\{\sqrt{T}, \frac{1}{4}\inj(M), 
	\frac{\pi}{4\sqrt{\kappa}}\right\}, \\
\end{equation}
where we interpret $\pi/(4\sqrt{\kappa})=\infty$ 
if $\kappa=0$ or $N=1$. 
We remark that $M=\R^N$ if and only if $\rho_\infty=\infty$, where 
we also interpret 
\begin{equation}\label{eq:rhopr0def}
	\rho_\infty= 
	\min\left\{\frac{1}{4}\inj(M), 
	\frac{\pi}{4\sqrt{\kappa}}\right\}. 
\end{equation}

\begin{example}
We list the possible values of $\rho_{\infty}$ for some typical Riemannian manifolds. 
\begin{itemize}
\item $\rho_{\infty}=\infty$ for $M=\R^N$ with the standard Euclidean metric. 
\item $0<\rho_{\infty}\leq \pi/4$ for $M=\bS^N$, the sphere with the standard metric. 
\item $0<\rho_{\infty}\leq \pi/4$ for $M=\R^{N-\ell}\times\bS^{\ell}$, 
the multi-cylinder with the standard metric ($\ell\geq1$). 
\item $0<\rho_{\infty}\leq \pi/4$ for $M=\bH^N$, 
the hyperbolic space with the standard metric. 
\item $0<\rho_{\infty}<\infty$ for any compact $M$. 
\item $0<\rho_{\infty}<\infty$ for any non-compact $M$ with bounded geometry, 
that is, the Riemannian curvature tensor 
and its derivatives are bounded and the injectivity radius is positive. 
For example, asymptotically locally euclidean manifolds 
(the so-called ALE spaces) are in this class. 
\end{itemize}
\end{example}

We denote by $d=d(x,y)$ the distance function on $(M,g)$. 
For $x\in M$ and $\rho>0$, we write 
$B(x,\rho):= \{y\in M  ; d(x,y)<\rho\}$. 
$C_0(M)$ denotes the space of compactly supported continuous functions on $M$. 
We denote by $C^{2,1}$ the space of functions 
which are twice continuously differentiable in the space variable  
and once in the time variable. 
$dV_g$ denotes the Riemannian volume form of $g$. 
Unless otherwise stated, we use the term ``solution'' in the following sense: 

\begin{definition}\label{def:sol}
Let $T>0$. 
A function $u$ is called a solution of \eqref{eq:main} 
in $M\times[0,T)$ 
if $u\geq0$, $u\in C^{2,1}(M\times(0,T))$ and $u$ satisfies 
\begin{align}
	\label{eq:fujita}
	&u_t-\Delta u=u^p  \quad \mbox{ in }M\times(0,T), \\
	\label{eq:ini}
	&\lim_{t\to 0} \int_M \psi  u(\,\cdot\,,t) dV_g
	= \int_M \psi d\mu
	\quad 
	\mbox{ for any }\psi\in C_0(M). 
\end{align}
\end{definition}

\subsection{Statements of main results}\label{sec:state}
Our first result gives necessary conditions on the solvability of \eqref{eq:main}. 

\begin{theorem}\label{th:nec}
Let $N\geq1$, $p>1$ and $0<T<\infty$. 
Let $u\in C^{2,1}(M\times(0,T))$ be a nonnegative function satisfying \eqref{eq:fujita}. 
Then there exists a nonnegative Radon measure $\mu$ 
uniquely determined by $u$ such that 
$u$ is a solution of \eqref{eq:main} in $M\times [0,T)$. 
Moreover, there exists a constant $C>0$ depending only on $N$ and $p$ 
such that the following (i), (ii) and (iii) hold: 
\begin{itemize}
\item[(i)]
If $p<p_F$, then 
$\displaystyle \sup_{z\in M} \mu( B(z,\rho_T) )\leq C \rho_T^{N-\frac{2}{p-1}}$. 
\item[(ii)]
If $p=p_F$, then 
$\displaystyle \sup_{z\in M} \mu ( B(z,\rho) ) \leq
	C ( \log (e+\rho_T\rho^{-1}) )^{-\frac{N}{2}}$
for any $0<\rho< \rho_T$. 
\item[(iii)]
If $p>p_F$, then 
$\displaystyle \sup_{z\in M} \mu( B(z,\rho) ) \leq C \rho^{N-\frac{2}{p-1}}$
for any $0<\rho< \rho_T$. 
\end{itemize}
\end{theorem}

Theorem \ref{th:nec} is sharp 
in view of the following two theorems:

\begin{theorem}\label{th:necsha}
Let $N\geq1$, $1<p<p_F$ and $0<T<\infty$. 
Then there exists a constant $c>0$ depending only on $N$ and $p$ 
such that the problem \eqref{eq:main} admits a solution in $M\times[0,\tilde \rho^2)$ 
for any $\tilde \rho\in(0,\rho_T]$ 
and nonnegative Radon measure $\mu$ satisfying
\begin{equation}\label{eq:sha1}
	\sup_{z\in M} \mu( B(z,\tilde \rho) )
	\leq c \tilde \rho^{N-\frac{2}{p-1}}. 
\end{equation}

\end{theorem}

\begin{remark}
By taking $T$ small, we can construct solutions
for $\mu$ satisfying \eqref{eq:sha1} with $c$ large. 
Thus \eqref{eq:sha1} is a necessary and sufficient condition 
for the existence of local-in-time solutions. 
\end{remark}

\begin{theorem}\label{th:necshacri}
Let $N\geq1$, $p\geq p_F$ and $0<T<\infty$. 
Then there exist constants $C>0$ and $\tilde C>1$ 
such that the following holds: 
For any constant $0<c<C$, 
there exists a nonnegative Radon measure $\mu$ satisfying 
\[
\left\{ 
\begin{aligned}
	&\tilde C^{-1} c \rho^{N-\frac{2}{p-1}}
	\leq \sup_{z\in M} \mu ( B(z,\rho) ) \leq
	\tilde C c \rho^{N-\frac{2}{p-1}} 
	&&\mbox{ if } p>p_F, \\
	&\tilde C^{-1} c ( \log (e+\rho_T\rho^{-1}) )^{- \frac{N}{2}}
	\leq \sup_{z\in M} \mu ( B(z,\rho) ) \leq
	\tilde C c ( \log (e+\rho_T\rho^{-1}) )^{- \frac{N}{2}} 
		&&\mbox{ if } p=p_F, \\
\end{aligned}
\right. 
\]
for any $0<\rho<\rho_T$ such that 
the problem \eqref{eq:main} admits a solution in $M\times[0,\rho_T^2)$. 
Here $C$ depends only on $N$ and $p$ if $p>p_F$ 
and only on $N$ and $T$ if $p=p_F$. 
On the other hand, $\tilde C$ depends only on $N$ and $p$ in each of the cases. 
\end{theorem}

\begin{remark}
Each $\mu$ in Theorem \ref{th:necshacri} 
can be written by $d\mu=u_0dx$ with a suitable function $u_0$. 
For the precise admissible singularity of $u_0$, see Corollary \ref{cor:singex}. 
\end{remark}

We show that the necessary conditions in Theorem \ref{th:nec} (ii) and (iii)
are not always sufficient conditions even if 
$C$ is replaced by a small constant. 

\begin{theorem}\label{th:nex}
Let $N\geq1$ and $p\geq p_F$. Then, for any constant $c>0$, 
there exists a nonnegative Radon measure $\mu$ satisfying 
\[
	\sup_{z\in M} \mu( B(z,\rho) ) \leq 
	c \rho^{N-\frac{2}{p-1}} 
	( \log(e+\rho^{-1}) )^{-\frac{1}{p-1}}
	\quad 
	\mbox{ for any }0<\rho< \rho_{\infty}
\]
such that 
the problem \eqref{eq:main} does not admit any solutions in $M\times[0,T)$ for any $T>0$. 
\end{theorem}

If we impose a slightly stronger condition than that of Theorem \ref{th:nex}, 
we obtain a general sufficient condition 
for the solvability of \eqref{eq:main} in the case $p\geq p_F$.

\begin{theorem}\label{th:suf}
Let $N\geq1$, $p\geq p_F$ and $0<T<\infty$. 
Then there exists a constant $c>0$ depending only on $N$ and $p$ 
such that the following holds: 
For any $\eps> 0$, $\tilde \rho\in (0,\rho_T]$ 
and nonnegative Radon measure $\mu$ satisfying
\begin{equation}\label{eq:suf}
	\sup_{z\in M} \mu( B(z,\rho) )\leq c \rho^{N-\frac{2}{p-1}} 
	( \log (e+\tilde \rho \rho^{-1}) )^{-\frac{1}{p-1} -\eps}
	\quad \mbox{ for any }0<\rho< \tilde \rho, 
\end{equation}
the problem \eqref{eq:main} admits a solution 
in $M\times[0,\tilde \rho^2)$. 
\end{theorem}

\subsection{Corollaries}\label{sec:cor}
We summarize corollaries of our theorems 
for the following problem: 
\begin{equation}\label{eq:fun}
\left\{ 
\begin{aligned}
	&u_t-\Delta u=u^p &&\mbox{ in }M\times(0,T), \\
	&u(\,\cdot\,,0)=u_0 &&\mbox{ in }M, 
\end{aligned}
\right.
\end{equation}
where $u_0$ is a nonnegative function. 
A function $u$ is called a solution of \eqref{eq:fun} in $M\times[0,T)$ 
if $u$ is a solution of \eqref{eq:main} 
in the sense of Definition \ref{def:sol} 
with \eqref{eq:ini} replaced by 
\[
	\lim_{t\to 0} \int_M \psi  u(\,\cdot\,,t) dV_g
	= \int_M \psi u_0 dV_g
	\quad \mbox{ for any } \psi\in C_0(M). 
\]

Let us first consider the existence of solutions of \eqref{eq:fun} 
with functions in 
the uniformly local Lebesgue space $L^q_{\uloc,\rho}$ as initial data. 
Here, we define $L^q_{\uloc,\rho}(M)$ by 
\[
\begin{aligned}
	&L^q_{\uloc,\rho}(M):=
	\{ f\in L^q_\loc(M); \| f\|_{L^q_{\uloc,\rho}(M)} <\infty\}, \\
	&\| f\|_{L^q_{\uloc,\rho}(M)}:= 
	\sup_{z\in M} \left( \int_{B(z,\rho)} |f|^q dV_g \right)^\frac{1}{q},
\end{aligned}
\]
for $\rho>0$ and $q>1$ 
(see \cite{Ka75,GV97,MT06} for $M=\R^N$). 
The following corollary generalizes 
the results by Weissler \cite{We79,We80} 
which handle the case where $u_0\in L^q(\Omega)$ with $\Omega\subset \R^N$. 

\begin{corollary}\label{cor:ex}
Let $N\geq1$ and $0<T<\infty$. 
Assume one of the following: 
\begin{itemize}
\item[(i)]
$p>p_F$ and $q\geq N(p-1)/2$. 
\item[(ii)]
$p=p_F$ and $q>1$. 
\item[(iii)]
$1<p<p_F$ and $q\geq1$. 
\end{itemize}
Then there exists a constant $c>0$ such that the following holds: 
For any $\tilde \rho\in(0,\rho_T]$ and nonnegative function 
$u_0\in L^q_{\uloc,\tilde \rho}(M)$ 
with $\|u_0\|_{L^q_{\uloc,\tilde \rho}(M)}\leq c$, 
the problem \eqref{eq:fun} admits a solution in $M\times[0,\tilde \rho^2)$. 
Here, $c$ depends only on $N$, $p$, $q$ 
and $\tilde \rho^{2-N(p-1)/q}$ if (i) holds 
and on $N$, $p$, $q$ and $\tilde \rho^{N(1-1/q)}$ if (ii) or (iii) holds. 
\end{corollary}

We next examine the nonexistence of solutions. 
Theorem \ref{th:nec} implies the following:  

\begin{corollary}\label{cor:singnex}
Let $N\geq1$, $p\geq p_F$ and $z_0\in M$. 
Then there exists a constant $c>0$ depending only on $N$ and $p$ such that the following holds: 
For any nonnegative function $u_0$ satisfying 
\begin{equation}\label{eq:u0la}
	u_0(x)\geq \left\{ 
	\begin{aligned}
	& c d(z_0,x)^{-\frac{2}{p-1}} 
	&&\mbox{ if }p>p_F, \\
	& c d(z_0,x)^{-N} 
	( \log (e + d(z_0,x)^{-1} ))^{-\frac{N}{2}-1}
	&&\mbox{ if }p=p_F, 
	\end{aligned}
	\right.
\end{equation}
for any $x\in B(z_0,\rho_\infty)$ and $u_0\in L^\infty(M\setminus B(z_0,\rho_\infty))$ 
if $M\setminus B(z_0,\rho_\infty)\neq \emptyset$, 
the problem \eqref{eq:fun} does not admit any local-in-time solutions. 
\end{corollary}

By Corollary \ref{cor:singnex}, 
we also obtain the nonexistence of solutions 
with functions in $L_{\uloc,\rho_\infty}^q$ as initial data. 
The following corollary generalizes the results by 
Weissler \cite{We80} and Brezis and Cazenave \cite{BC96} 
for $u_0\in L^q(\Omega)$ with $\Omega\subset \R^N$.

\begin{corollary}\label{cor:nex}
Let $N\geq1$. Assume one of the following:  
\begin{itemize}
\item[(i)]
$p>p_F$ and $1\leq q<N(p-1)/2$. 
\item[(ii)]
$p=p_F$ and $q=1$. 
\end{itemize}
Then there exists a nonnegative function 
$u_0\in L_{\uloc,\rho_\infty}^q(M)$ such that 
the problem \eqref{eq:fun} does not admit any solutions 
in $M\times[0,T)$ for any $T>0$. 
\end{corollary}

Finally, we give a pointwise condition for existence. 
The proof of Theorem \ref{th:necshacri} shows the following: 

\begin{corollary}\label{cor:singex}
Let $N\geq1$, $p\geq p_F$ and $z_0 \in M$. 
Then there exists a constant $c>0$ depending only on $N$ and $p$ such that the following holds: 
For any nonnegative function $u_0$ satisfying 
\[
	u_0(x)\leq \left\{ 
	\begin{aligned}
	& c d(z_0,x)^{-\frac{2}{p-1}} 
	&&\mbox{ if }p>p_F, \\
	& c d(z_0,x)^{-N} 
	( \log (e + d(z_0,x)^{-1} ))^{-\frac{N}{2}-1} 
	&&\mbox{ if }p=p_F, 
	\end{aligned}
	\right.
\]
for any $x\in B(z_0,\rho_\infty)$ and $u_0\in L^\infty(M\setminus B(z_0,\rho_\infty))$ 
if $M\setminus B(z_0,\rho_\infty)\neq \emptyset$, 
the problem \eqref{eq:fun} admits a local-in-time solution. 
\end{corollary}

\begin{remark}\label{rem:Di}
Analogous results to Theorems \ref{th:necsha}, \ref{th:necshacri} and \ref{th:suf} 
and Corollaries \ref{cor:ex} and \ref{cor:singex} 
hold for the problem \eqref{eq:main} in a smooth domain $\Omega\subset M$ 
with $u=0$ on $\partial \Omega$, 
since the heat kernel on $M$ dominates the Dirichlet heat kernel on $\Omega$ 
(see \cite[Page 188, Theorem 4]{Chbook}). 
We note that the analog of Corollary \ref{cor:ex} for $\Omega$ 
is stronger than the results of \cite{Pu11,Pu12o,Pu14}.
\end{remark}

\section{Necessary conditions}\label{sec:nec}

Fix $0<T<\infty$. 
In Lemma \ref{lem:exini}, 
we show the existence and uniqueness of the initial trace $\mu$ of $u$ 
based on the argument of \cite[Lemma 2.3]{HI18} 
due to the weak compactness of Radon measures and 
some limiting argument. 
In Lemma \ref{lem:young}, to estimate $\mu$, 
we substitute appropriate test functions into 
the following weak form of \eqref{eq:main}: 
\[
	-\int_\tau^T \int_{M} u(\varphi_t+\Delta \varphi) dV_{g} dt
	=
	\int_\tau^T \int_{M} u^p \varphi dV_{g}  dt
	+ \int_{M} u(\,\cdot\,,\tau) \varphi(\,\cdot\,,\tau) dV_{g}, 
\]
where $\varphi\in C_0^{2,1}(M\times[0,T))$ and $0<\tau<T$. 
Then by the comparison theorems in Riemannian geometry 
(see Appendix \ref{sec:app-comp}), 
we obtain estimates of $\mu ( B(z,c \rho) )$ with some $0<c<1$. 
By showing a generalization of the Besicovitch type covering theorem 
(see Appendix \ref{sec:appa}), 
we also obtain the desired estimates of $\mu ( B(z,\rho) )$. 
We note that the covering theorems used in this paper 
seem to be of independent interest.

\begin{lemma}\label{lem:exini}
Let $u\in C^{2,1}(M\times(0,T))$ be a nonnegative function satisfying \eqref{eq:fujita}. 
Then there exists a nonnegative Radon measure $\mu$ uniquely determined by $u$ 
such that $u$ satisfies \eqref{eq:ini} 
and that $u$ is a solution of \eqref{eq:main} 
in the sense of Definition \ref{def:sol}. 
\end{lemma}

\begin{proof}
For $0<\tau<T$, 
we have $u(\,\cdot\,,t)\to u(\,\cdot\,,\tau)$ and 
$\int_M K(\,\cdot\,,y,t-\tau) u(y,\tau) dV_g(y) \to u(\,\cdot\,,\tau)$ 
in $L^2_\loc(M)$ as $t\to\tau$. 
These together with $u\geq0$, $u_t\geq \Delta u$ and 
the minimality of $K$ in \cite[Theorem 8.1]{Grbook} 
show that 
\begin{equation}\label{eq:uKta}
	u(x,t)
	\geq 
	\int_M  K(x,y,t-\tau) u(y,\tau) dV_g(y)
\end{equation}
for $x\in M$ and $0<\tau<t<T$. 
Fix $z_0\in M$ and $(2/3)T<t_0<T$. 
Let $M_0$ be a compact subset of $M$. 
Since $M$ is connected, by \cite[Corollary 8.12]{Grbook}, 
we have $K(x,y,t)>0$ for $x,y\in M$ and $t>0$. 
Thus, 
\[
\begin{aligned}
	u(z_0,t_0)
	&\geq 
	\int_{M_0}  K(z_0,y,t_0-t) u(y,t) dV_g(y) \\
	&\geq 
	\inf_{y\in M_0, \frac{1}{3}T<s<T} K(z_0,y,s) 
	\int_{M_0} u(\,\cdot\,,t) dV_g
\end{aligned}
\]
for $0<t<T/3$, and so 
\begin{equation}\label{eq:M0usup}
	\int_{M_0} u(\,\cdot\,,t) dV_g \leq C
	\quad \mbox{ for } 0<t<\frac{T}{3},  
\end{equation}
where $C>0$ is a constant depending on $z_0$, $t_0$, $M_0$ and $T$. 
Define a family of nonnegative Radon measures $\{\mu_t\}_{t\in(0,T/3)}$ on $M$ 
by $\mu_t(A):= \int_A u(\,\cdot\,,t) dV_g$. 
Let $\{t_k\}_{k=1}^\infty\subset(0,T/3)$ satisfy $t_k\to 0$. 
Then, we obtain $\sup_{k\geq1} \mu_{t_k}(M_0) <\infty$ 
for any compact subset $M_0\subset M$. 
By the weak compactness of Radon measures \cite[THEOREM 4.4]{Sibook}, 
there exist a subsequence $\{\mu_{t_k'}\}_{k=1}^\infty$
and a nonnegative Radon measure $\mu$ on $M$ such that 
\[
	\lim_{k\to\infty} \int_M \psi  d\mu_{t_k'}
	= \int_M \psi d\mu
	\quad 
	\mbox{ for any }\psi\in C_0(M). 
\]

We claim that $\mu_{t_k}\to\mu$ without taking a subsequence. 
To prove this, take a subsequence 
$\{\mu_{t_k''}\}_{k=1}^\infty$ of $\{\mu_{t_k}\}_{k=1}^{\infty}$ arbitrary. 
Then, by the same reason as above, 
there exist a subsequence $\{\mu_{t_k'''}\}_{k=1}^\infty$ 
of $\{\mu_{t_k''}\}_{k=1}^\infty$ and 
a nonnegative Radon measure $\nu$ on $M$ such that  
$\lim_{k\to \infty} \int_M \psi  d\mu_{t_k'''}
= \int_M \psi d\nu$
for any $\psi\in C_0(M)$. 
If $\mu=\nu$, then the limit 
$\mu_{t_k}\to\mu$ immediately follows. 
Let $\phi\in C_0(M)$ with $\phi\geq0$. 
We take a subsequence again if necessary 
and suppose $t_k'>t_k''$ for $k\geq1$. 
From \eqref{eq:uKta} with $\tau=t_k''$ and $t=t_k'$, 
it follows that 
\[
\begin{aligned}
	\int_M \phi d\mu_{t_k'} & = 
	\int_M \phi(x) u(x,t_k')  dV_g(x)  \\
	&\geq 
	\int_M \phi(x) \int_M K(x,y,t_k'-t_k'') u(y,t_k'')  dV_g(y) dV_g(x) \\
	&= 
	\int_M \int_M K(y,x,t_k'-t_k'') \phi(x)  dV_g(x) d\mu_{t_k''}(y) \\
	&\geq 
	\int_{\supp \phi} \int_M K(y,x,t_k'-t_k'') \phi(x)  dV_g(x)  d\mu_{t_k''}(y)
\end{aligned}
\]
and that 
\[
\begin{aligned}
	\int_M \phi d\mu_{t_k'} 
	&\geq 
	\int_{\supp \phi} \phi d\mu_{t_k''} 
	- \int_{\supp \phi} 
	\left| \int_M K(y,x,t_k'-t_k'') \phi(x)  dV_g(x) - \phi(y) \right| d\mu_{t_k''}(y) \\
	&\geq
	\int_{\supp \phi} \phi d\mu_{t_k''} 
	- \mu_{t_k''}(\supp \phi)
	\sup_{y\in M}\left| \int_M K(y,x,t_k'-t_k'') \phi(x)  dV_g(x) - \phi(y) \right|. 
\end{aligned}
\]
By the uniform continuity of $\phi$ on $M$ 
and \cite[Theorem 7.13]{Grbook}, 
we have 
\[
	\lim_{k\to\infty} \int_M K(y,x,t_k'-t_k'') \phi(x)  dV_g(x)= \phi(y)
\]
uniformly for $y\in M$. 
By the compactness of $\supp \phi$ and \eqref{eq:M0usup}, 
we have $\sup_{k\geq1} \mu_{t_k''}(\supp \phi) <\infty$. 
Therefore, letting $k\to\infty$ yields 
$\int_M \phi d\mu \geq \int_{\supp \phi} \phi d\nu = \int_M \phi d\nu$. 
Interchanging $t_k'$ and $t_k''$ gives 
$\int_M \phi d\mu = \int_M \phi d\nu$ 
for any $\phi\in C_0(M)$ with $\phi\geq0$. Thus, 
$\int_M \psi d\mu = \int_M \psi d\nu$  
for any $\psi\in C_0(M)$. 
This implies $\mu=\nu$, 
and so $\mu_{t_k}\to\mu$ without taking a subsequence. 
Then \eqref{eq:ini} holds, and so $u$ is a solution of \eqref{eq:main} 
in the sense of Definition \ref{def:sol}. 
The uniqueness of $\mu$ also follows from the above argument. 
\end{proof}

We prepare an estimate for proving Theorem \ref{th:nec}. 
In what follows, 
$f^+:=\max\{f,0\}$ and $f^-:=\max\{-f,0\}$ denote the positive part 
and the negative part of $f$, respectively.

\begin{lemma}\label{lem:young}
Let $p>1$ and let $u$ be a solution of \eqref{eq:main}. 
Then there exists a constant $C>0$ depending only on $p$ such that 
the initial data $\mu$ of $u$ satisfies 
\begin{equation}\label{eq:phi}
\begin{aligned}
	\int_{M} \phi(\,\cdot\,,0)^\frac{2p}{p-1} d\mu
	\leq
	C \int_0^T \int_{M}  
	\phi^\frac{p}{p-1}  
	\left(  \left(\phi_t + \Delta \phi\right)^{-} \right)^\frac{p}{p-1} dV_{g} dt
\end{aligned}
\end{equation}
for any $\phi \in C^\infty_0(M \times [0,T))$ with $\phi\geq0$. 
\end{lemma}

\begin{proof}
Let $\varphi\in C_0^{2,1}(M\times[0,T))$ and $0<\tau<T$. 
Multiplying \eqref{eq:fujita} by $\varphi$ and 
using integration by parts, we have 
\[
	-\int_\tau^T \int_{M} u(\varphi_t+\Delta \varphi) dV_{g} dt
	=
	\int_\tau^T \int_{M} u^p \varphi dV_{g}  dt
	+ \int_{M} u(\,\cdot\,,\tau) \varphi(\,\cdot\,,\tau) dV_{g}.
\]
Let $\phi \in C^\infty_0(M \times [0,T))$ with $\phi\geq0$. 
We take $\varphi=\phi^{2p/(p-1)}$. 
Then, 
\[
\begin{aligned}
	&\varphi_t=\frac{2p}{p-1} \phi^\frac{p+1}{p-1} \phi_t, \\
	&\Delta \varphi =
	\frac{2p(p+1)}{(p-1)^2} \phi^{\frac{2}{p-1}} |\nabla \phi|^2
	+ \frac{2p}{p-1} \phi^\frac{p+1}{p-1} \Delta \phi
	\geq \frac{2p}{p-1} \phi^\frac{p+1}{p-1} \Delta \phi, 
\end{aligned}
\]
and so 
\[
\begin{aligned}
	\int_\tau^T \int_{M} u^p \phi^\frac{2p}{p-1} dV_{g} dt
	+ \int_{M} u(\,\cdot\,,\tau) \phi(\,\cdot\,,\tau)^\frac{2p}{p-1} dV_{g} 
	&\leq
	- \int_\tau^T \int_{M} u\phi^\frac{p+1}{p-1}
	\frac{2p}{p-1} ( \phi_t + \Delta \phi ) dV_{g} dt \\
	&\leq
	\int_\tau^T \int_{M} u\phi^\frac{p+1}{p-1}
	\frac{2p}{p-1} (\phi_t + \Delta \phi )^{-} dV_{g} dt. 
\end{aligned}
\]
By Young's inequality, we have 
\[
	u\phi^\frac{p+1}{p-1} \frac{2p}{p-1} (\phi_t + \Delta \phi )^{-}
	\leq
	u^p \phi^\frac{2p}{p-1}
	+ C \phi^\frac{p}{p-1} 
	\left(  \left(\phi_t + \Delta \phi\right)^{-} \right)^\frac{p}{p-1}. 
\]
Thus, 
\[
	\int_{M} u(\,\cdot\,,\tau) \phi(\,\cdot\,,\tau)^\frac{2p}{p-1} dV_{g}
	\leq
	C \int_\tau^T \int_{M} \phi^\frac{p}{p-1}  
	\left(  \left(\phi_t + \Delta \phi\right)^{-} \right)^\frac{p}{p-1} dV_{g} dt. 
\]
By using \eqref{eq:ini} and letting $\tau\to0$, 
there exists a constant $C>0$ depending only on $p$ such that 
\eqref{eq:phi} holds. 
\end{proof}

In the rest of this section, we prove Theorem \ref{th:nec} 
by choosing appropriate $\phi$ in \eqref{eq:phi}.

\begin{proof}[Proof of Theorem \ref{th:nec} (i) and (iii)]
We take $\eta\in C^\infty(\R)$ such that  
$0\leq \eta \leq 1$, $\eta'\geq 0$, 
$\eta(\tau)=0$ for $\tau \leq 0$ and 
$\eta(\tau)=1$ for $\tau \geq 1$. 
Fix $z\in M$. For $0<\rho< \rho_T$, 
we set 
\[
	\phi(x,t):=\eta(\Phi(x,t)), \qquad 
	\Phi(x,t):= \frac{\rho^2}{d(x,z)^2+t} -1,
\]
for $(x,t)\in M\times [0,\infty)$. 
Since $\rho_T\leq \sqrt{T}$, we have 
$\Phi(x,T)\leq \rho_T^2T^{-1}-1\leq0$. 
Then $\phi(\,\cdot\,,T)=0$, and so $\phi\in C^\infty_0(M\times[0,T))$. 
Let 
\[
	D:= \{ (x,t)\in M\times[0,\infty); \rho^2 /2< d(x,z)^2 +t < \rho^2\}. 
\]
Note that $d(\,\cdot\,,z)$ is smooth on $B(z,\inj(M))\setminus\{z\}$ and that
\begin{equation}\label{eq:inc}
\begin{aligned}
	&(M\times [0,\infty)) \setminus D 
	\subset \{ (x,t)\in M\times[0,\infty); \eta'(\Phi(x,t))= 0\}, \\
	&\{(x,t)\in M\times[0,\infty); \eta'(\Phi(x,t))\neq 0\} 
	\subset D  \subset B(z,\inj(M)) \times [0,\infty). 
\end{aligned}
\end{equation}

From $\Phi_t = -\rho^2 (d(x,z)^2+t)^{-2}$, it follows that 
\begin{equation}\label{202009050046-1}
	\phi_t = (\eta'\circ \Phi)\Phi_t \geq
	-C \rho^2 (d(x,z)^2+t)^{-2} 
	\quad \mbox{ for } (x,t)\in D. 
\end{equation}
Direct computations yield 
\[
\begin{aligned}
	\nabla \Phi &= -2\rho^2 (d(x,z)^2+t)^{-2} d(x,z) \nabla d(x,z), \\
	\Delta \Phi &= 8\rho^2 (d(x,z)^2+t)^{-3} d(x,z)^2 |\nabla d(x,z)|^2 \\
	&\quad -2\rho^2 (d(x,z)^2+t)^{-2} |\nabla d(x,z)|^2
	- 2 \rho^2 (d(x,z)^2+t)^{-2} d(x,z) \Delta d(x,z). 
\end{aligned}
\]
Then by $|\nabla d(\,\cdot\,,z)|=1$, we have 
\[
\begin{aligned}
	\Delta \phi &=
	(\eta''\circ \Phi) |\nabla \Phi|^2
	+ (\eta'\circ \Phi) \Delta \Phi   \\
	&\geq
	-C |\nabla \Phi|^2 -C (\Delta \Phi)^{-} \\
	&\geq
	-C\rho^2 (d(x,z)^2+t)^{-2} 
	-C \rho^2 (d(x,z)^2+t)^{-2} d(x,z) (\Delta d(x,z))^{+} 
	\quad \mbox{ for }(x,t)\in D. 
\end{aligned}
\]

We claim that 
\begin{equation}\label{202009050046-2}
	\Delta \phi \geq -C \rho^2 (d(x,z)^2+t)^{-2}.
\end{equation}
To show this, it suffices to prove 
\begin{equation}\label{eq:deld}
	\Delta d(x,z) \leq Cd(x,z)^{-1}
	\quad \mbox{ for }x\in B(z,\rho)\setminus\{z\}. 
\end{equation}
Let $r(x):=d(x,z)$. 
Remark that $0<r<\rho<\rho_T$. 
By the assumption on $M$, 
we can apply \eqref{prep1} and we have 
$\Delta r(x)\leq 2(N-1)r^{-1}(x)$ 
for $x\in B(z,\rho)\setminus\{z\}$, 
where we used $\rho_T\leq \pi/(4\sqrt{\kappa})$. 
Hence \eqref{202009050046-2} follows.

By \eqref{202009050046-1} and \eqref{202009050046-2}, we obtain 
$(\phi_t + \Delta \phi)^{-}\leq C \rho^2 (d(x,z)^2+t)^{-2}$. 
This together with Lemma \ref{lem:young} and $\phi\leq1$ shows that 
\begin{equation}\label{eq:youngmd}
\begin{aligned}
	\int_{M} \phi(\,\cdot\,,0)^\frac{2 p}{p-1} d\mu
	\leq
	C \rho^\frac{2p}{p-1}
	\iint_D (d(x,z)^2+t)^{-\frac{2p}{p-1}} dV_{g}(x) dt, 
\end{aligned}
\end{equation}
where $C>0$ is a constant depending only on $N$ and $p$. 
Since $\phi(x,0) = 1$ for $x\in B(z,\rho/2)$, 
the left-hand side of \eqref{eq:youngmd} can be estimated as 
\[
	\int_{M} \phi(\,\cdot\,,0)^\frac{2p}{p-1} d\mu \geq \mu(B(z,\rho/2)).
\]
We estimate the right-hand side of \eqref{eq:youngmd} 
by using Riemannian normal coordinates 
$(x_1,\ldots,x_N)$ centered at $z$, 
where we identify $x\in B(z,\inj(M))$ 
with $(x_1,\ldots,x_N)$. 
In these coordinates, we have 
\[
\begin{aligned}
	\iint_D (d(x,z)^2+t)^{-\frac{2p}{p-1}} dV_{g}(x) dt 
	=&\iint_{\tilde D}
	(|x|^2+t)^{-\frac{2p}{p-1}} \sqrt{\det(g_{ij}(x))}dx dt
\end{aligned}
\]
with $\tilde D:=\{(x,t)\in \R^{N}\times[0,\infty); 
\rho^2/2 \leq |x|^2 +t \leq \rho^2\}$. 
By the assumption on $M$, 
we can apply the volume comparison theorem \eqref{prep2} and have 
$\sqrt{\det(g_{ij}(x))}\leq 2^{N-1}$ 
for $x\in B(z,\rho)\setminus\{z\}$, 
where we used $\rho_T\leq \pi/(4\sqrt{\kappa})$. 
Thus, 
\[
\begin{aligned}
	\iint_{\tilde D}
	(|x|^2+t)^{-\frac{2p}{p-1}} \sqrt{\det(g_{ij})}dx dt 
	&\leq 
	C \iint_{\tilde D}
	(|x|^2+t)^{-\frac{2p}{p-1}}  dx dt \\
	&\leq 
	2C \iint_{ \{ |x|^2 +s^2 \leq \rho^2\} }
	(|x|^2+s^2)^{-\frac{2p}{p-1}} s dx ds \\
	&\leq 
	C \iint_{ \{ |x|^2 +s^2 \leq \rho^2\} }
	(\sqrt{|x|^2+s^2} )^{-\frac{4p}{p-1}+1} dx ds 
	\leq 
	C \rho^{N+2-\frac{4p}{p-1}}. 
\end{aligned}
\]
Hence, for all $p>1$, we obtain 
\begin{equation}\label{eq:must}
	\mu(B(z,\rho/2)) \leq C \rho^{N-\frac{2}{p-1}}
	\quad 
	\mbox{ for any }0<\rho<\rho_T, 
\end{equation}
where $C>0$ is a constant depending only on $N$ and $p$. 
Fix $z'\in M$. 
Since $\rho_T\leq R'$ (see \eqref{sp_rad} for the definition of $R'$), 
we can use Theorem \ref{half_ball} and say that 
there exists a number $k_{1}$ depending only on $N$ such that $B(z',\rho)$ can be covered 
by at most $k_{1}$ balls with radii $\rho/2$. 
This together with \eqref{eq:must} implies 
\[
	\mu(B(z',\rho)) \leq k_{1}C \rho^{N-\frac{2}{p-1}}
	\quad 
	\mbox{ for any }0<\rho<\rho_T. 
\]
This shows (i) and (iii). 
\end{proof}

We improve  \eqref{eq:must} in the case $p=p_F$. 

\begin{proof}[Proof of Theorem \ref{th:nec} (ii)]
We take $\eta\in C^\infty(\R)$ 
as in the proof of Theorem \ref{th:nec} (i) and (iii). 
Fix $z\in M$. 
For $0<\rho<\rho_T$, 
we set 
\[
	\phi(x,t):=\eta(\Phi(x,t)), \qquad 
	\Phi(x,t) :=  2 \left( \log\left( e+\frac{16\rho_T^2}{\rho^2} \right) \right)^{-1}
	\log\left( e+\frac{\rho_T^2}{d(x,z)^2+ t} \right)  - 1,
\]
for $(x,t)\in M\times [0,\infty)$. 
Remark that $\phi(\,\cdot\,,T)=0$. Indeed, this follows from 
\[
\begin{aligned}
	\Phi(x,T)\leq 
	2 ( \log( e+16 ) )^{-1}
	\log( e+ 1)  - 1 \leq 0. 
\end{aligned}
\]
Let 
\[
	D' := \left\{ (x,t)\in M\times[0,\infty); 
	\rho^2/16 < d(x,z)^2 +t 
	< \rho_T^2( (e+16\rho_T^2\rho^{-2})^\frac{1}{2} -e  )^{-1} \right\}. 
\]
By $(e+ 16 \rho_T^2\rho^{-2})^{1/2} -e>1$,  
we have $\rho_T ( (e+ 16 \rho_T^2\rho^{-2})^{1/2} -e  )^{-1/2}<\rho_T\leq \inj(M)$. 
Therefore, \eqref{eq:inc} holds with $D$ replaced by $D'$. 
Note that $d(\,\cdot\,,z)$ is smooth on $B(z,\inj(M))\setminus\{z\}$. 
By 
\[
	\Phi_t = -2\rho_T^2(\log(e+16\rho_T^2\rho^{-2}))^{-1} 
	(e(d^2+t)+\rho_T^2)^{-1} (d^2+t)^{-1}, 
\]
we have 
\[
	\phi_t \geq -C|\Phi_t| \geq 
	- C ( \log( e+ 16\rho_T^2 \rho^{-2}))^{-1}  (d(x,z)^2+t)^{-1}
\]
for some constant $C>0$ independent of $T$. 
Straightforward computations yield
\[
\begin{aligned}
	\nabla \Phi &= 
	-4\rho_T^2 ( \log( e+ 16 \rho_T^2 \rho^{-2}))^{-1} 
	(e(d^2+t)+\rho_T^2)^{-1} (d^2+t)^{-1} d\nabla d, \\
	\Delta \Phi &= 
	4\rho_T^2 ( \log( e+ 16 \rho_T^2 \rho^{-2}))^{-1} 
	(e(d^2+t)+\rho_T^2)^{-1} (d^2+t)^{-1} \\
	&\quad \times \left( 
	2(d^2+t)^{-1} d^2 |\nabla d|^2 
	+2e (e(d^2+t)+\rho_T^2)^{-1} d^2|\nabla d|^2 
	-|\nabla d|^2 - d\Delta d 
	\right). 
\end{aligned}
\]
Then, 
\[
\begin{aligned}
	\Delta \phi &=
	(\eta''\circ \Phi) |\nabla \Phi|^2
	+ (\eta'\circ \Phi) \Delta \Phi   \\
	&\geq
	-C|\nabla \Phi|^2 -C(\Delta \Phi)^{-} \\
	&\geq
	-C \rho_T^4 ( \log( e+ 16\rho_T^2 \rho^{-2}))^{-2} 
	(e(d^2+t)+\rho_T^2)^{-2} (d^2+t)^{-2} d^2 |\nabla d|^2 \\
	&\quad - C\rho_T^2 ( \log( e+ 16\rho_T^2 \rho^{-2}))^{-1} 
	(e(d^2+t)+\rho_T^2)^{-1} 
	(d^2+t)^{-1}  ( |\nabla d|^2 + d (\Delta d)^+ ). 
\end{aligned}
\]
From $|\nabla d|=1$ and \eqref{eq:deld}, it follows that 
\[
\begin{aligned}
	\Delta \phi \geq 
	-C ( \log( e + 16 \rho_T^2 \rho^{-2}))^{-1} (d(x,z)^2+t)^{-1}, 
\end{aligned}
\]
where $C>0$ is a constant depending only on $N$ and $p_{F}$. 

By the above computations, Lemma \ref{lem:young} and $p=p_F$, we obtain 
\begin{equation}\label{eq:Lemcri}
\begin{aligned}
	\int_{M} \phi(\,\cdot\,,0)^{N+2} d\mu
	\leq
	C ( \log( e+\rho_T^2 \rho^{-2}))^{-\frac{N}{2}-1}
	\iint_{D'} (d(x,z)^2+t)^{-\frac{N}{2}-1} dV_{g}(x) dt. 
\end{aligned}
\end{equation}
Since $\phi(x,0) = 1$ for $x\in B(z,\rho/4)$, 
the left-hand side of \eqref{eq:Lemcri} can be estimated as 
\[
	\int_{M} \phi(\,\cdot\,,0)^{N+2} d\mu \geq \mu(B(z,\rho/4)).
\]
We estimate the right-hand side of \eqref{eq:Lemcri}. 
In Riemannian normal coordinates $(x_1,\ldots,x_N)$ centered at $z$, 
we have 
\[
\begin{aligned}
	& \iint_{D'} (d(x,z)^2+t)^{-\frac{N}{2}-1} dV_{g}(x) dt
	= 
	\iint_{\tilde D'} (|x|^2+t)^{-\frac{N}{2}-1} \sqrt{\det(g_{ij})}dx dt,  \\
	& \tilde D' := \left\{ (x,t)\in M\times[0,\infty); 
	\theta_1^2 < |x|^2 +t < \theta_2^2 \right\}, \\
	&\theta_1:=\rho/4, \qquad
	\theta_2:= \rho_T ( (e+ 16 \rho_T^2\rho^{-2})^\frac{1}{2} -e  )^{-\frac{1}{2}}. 
\end{aligned}
\]
Then by the volume comparison theorem \eqref{prep2} 
and $\theta_2<\rho_T$, we obtain
\[
\begin{aligned}
	\iint_{\tilde D'} (|x|^2+t)^{-\frac{N}{2}-1} \sqrt{\det(g_{ij})}dx dt 
	&\leq 
	C\iint_{\tilde D'} (|x|^2+t)^{-\frac{N}{2}-1} dx dt \\
	&=
	2C\iint_{\{\theta_1^2 < |x|^2 +t < \theta_2^2\}} 
	(|x|^2+s^2)^{-\frac{N}{2}-1} s dx ds \\
	&\leq 
	2C\iint_{\{\theta_1^2 < |x|^2 +s^2 < \theta_2^2\}} 
	(\sqrt{|x|^2+s^2})^{-N-1} dx ds \\
	&\leq 
	 C \log(\theta_2/\theta_1) 
	 \leq 
	 C \log(e+ \rho_T \rho^{-1}), 
\end{aligned}
\]
where $C>0$ is a constant independent of $\rho_T$. 
Hence we deduce that 
\[
	\mu(B(z,\rho/4)) \leq C ( \log( e+\rho_T\rho^{-1}))^{-\frac{N}{2}}
\]
for any $0<\rho<\rho_T$, 
where $C>0$ is a constant depending only on $N$ and $p_{F}$. 
Fix $z'\in M$. 
Recall that $\rho_T\leq R'$, where $R'$ is given by \eqref{sp_rad}. 
Then, applying Theorem \ref{half_ball} twice, we obtain 
\[
	\mu(B(z',\rho)) \leq  k_{1}^2C ( \log( e+\rho_T\rho^{-1}))^{-\frac{N}{2}}
\]
for any $0<\rho<\rho_T$. 
This is equivalent to (ii), and the proof is complete. 
\end{proof}

\section{Sharpness of necessary conditions}\label{sec:sharp}
Let $K=K(x,y,t)$ be the heat kernel on $(M,g)$. 
To construct solutions of \eqref{eq:main}, we construct integral solutions 
in the following sense:

\begin{definition}
Let $0<T\leq \infty$. 
\begin{itemize}
\item[(i)]
A function $u$ is called an integral solution of \eqref{eq:main} in $M\times[0,T)$
if $u$ is a measurable function on $M \times (0,T)$ with 
$0\leq u<\infty$ a.e. $M\times (0,T)$ 
satisfying 
\begin{align}
	&u(x,t) = \Psi[u](x,t) \quad 
	\mbox{ for a.e. }(x,t)\in M \times (0,T), \label{eq:integ} \\
	&\Psi[u](x,t) :=
	\int_{M} K(x,y,t) d\mu(y) 
	+ \int_0^t \int_{M} K(x,y,t-s) u(y,s)^p dV_{g}(y) ds. \notag
\end{align}
\item[(ii)]
A function $\overline{u}$ is called a supersolution of \eqref{eq:integ} 
in $M\times[0,T)$ 
if $\overline{u}$ is a measurable function on $M \times (0,T)$ satisfying 
$0\leq \overline{u}<\infty$ 
and $\overline{u} \geq \Psi[\overline{u}]$ 
a.e. in $M\times (0,T)$. 
\end{itemize}
\end{definition}

By applying a similar argument of \cite{HI18} together with 
the Li-Yau Harnack inequalities \cite{LY86} and the volume comparison,  
we give a sharp pointwise estimate of a solution $\int_{M} K(x,y,t) d\mu(y)$ of the linear heat equation, 
where $\mu$ is any nonnegative Radon measure on $M$. 
Then we give candidates of supersolutions of the integral equation 
in the same spirit of \cite[Section 4]{RS13} 
and \cite[Section 4]{HI18}. 
By using the estimate of $\int_{M} K(x,y,t) d\mu(y)$, 
we prove that the candidates are in fact appropriate supersolutions. 
Then the following lemma guarantees 
the existence of solutions:

\begin{lemma}\label{lem:ite}
Assume that there exists 
a supersolution $\overline{u}$ of \eqref{eq:integ} in $M\times[0,T)$. 
Then there exists an integral solution of \eqref{eq:main} in $M\times[0,T)$. 
\end{lemma}

\begin{proof}
Since $M$ is connected, $K$ is positive. 
Then this lemma follows from the same monotone iteration scheme as in 
\cite[Theorem 1]{RS13}. See also \cite[Lemma 2.2]{HI18}. 
\end{proof}

In what follows, we first give estimates 
on the linear part in Subsection \ref{subsec:elin}. 
Next, we prove Theorems \ref{th:necsha}, \ref{th:necshacri} and \ref{th:suf}
in Subsections \ref{subsec:subcri}, \ref{subsec:cri} and \ref{subsec:suff}, 
respectively.

\subsection{Estimates of the linear part}\label{subsec:elin}
We estimate solutions of the linear heat equation 
with measures as initial data. 
The following lemma is a generalization of \cite[Lemma 2.1]{HI18}: 

\begin{lemma}\label{lem:linheat}
Let $\mu$ be a nonnegative Radon measure on $M$. 
Then there exists a constant $C>0$ depending only on $N$ 
such that 
\begin{equation}\label{eq:lines}
	\int_{M} K(x,y,t) d\mu(y)  
	\leq
	C t^{-\frac{N}{2}} 
	\sup_{z\in M} \mu( B(z,t^\frac{1}{2}) ) 
\end{equation}
for $(x,t)\in M\times(0,16 \rho_\infty^2)$. 
\end{lemma}

\begin{remark}
If $M=\R^N$, then \eqref{eq:lines} holds for 
$(x,t)\in M\times (0,\infty)$. 
\end{remark}

\begin{proof}
Assume $\rho_\infty<\infty$. 
The case $\rho_\infty=\infty$ will be handled 
at the end of this proof. 
Let $0<t< 16 \rho_\infty^2$. 
Define $\mathcal{F}:= \bigcup_{x\in M} \overline{B(x,(t/4)^{1/2})}$. 
Since $(t/4)^{1/2}<2\rho_\infty<\inj(M)/2$, 
Corollary \ref{besi} with $\xi:=2 \rho_\infty$ implies the following:  
There exist points $x_{k,i}\in M$ for $k=1,\ldots,k_0$ and $i\geq1$ 
such that 
$\overline{B(x_{k,i}, ( t/4)^{1/2})} 
\cap \overline{B(x_{k,j}, ( t/4)^{1/2})}=\emptyset$ for $i\neq j$ and 
$M = \bigcup_{k=1}^{k_0} \bigcup_{i=1}^\infty 
\overline{B(x_{k,i},( t/4)^{1/2})}$. 
Here, the constant $k_{0}$ depends only on $N$.

Fix $x\in M$. Then, 
\[
\begin{aligned}
	\int_{M} K(x,y,t) d\mu(y) 
	&\leq 
	\sum_{k=1}^{k_0} \sum_{i=1}^\infty 
	\int_{\overline{B(x_{k,i},( t/4)^{1/2}) }} K(x,y,t) d\mu(y) \\
	&\leq 
	\sum_{k=1}^{k_0} \sum_{i=1}^\infty 
	\mu( \overline{B(x_{k,i},( t/4)^{1/2})} ) 
	\sup_{ y\in \overline{B(x_{k,i},( t/4)^{1/2} )}} K(x,y,t) \\
	&\leq 
	\sup_{z\in M} \mu( B(z, t^{1/2} ))
	\sum_{k=1}^{k_0} \sum_{i=1}^\infty 
	\sup_{ y\in \overline{B(x_{k,i},( t/4)^{1/2} )}} K(x,y,t). 
\end{aligned}
\]

To estimate $\sup_{y} K(x,y,t)$, 
we take $y,z\in \overline{B(x_{k,i}, ( t/4)^{1/2})}$. 
Then, by $|\sec(M)|\leq \kappa$, 
$d(y,z)\leq \sqrt{t}$ and $t<\pi^2/\kappa$, 
the estimate of the heat kernel \eqref{prep8} implies 
the existence of a constant $C>0$ depending only on $N$ such that 
\[
K(x,y,t) \leq C K(x,z,2t). 
\]
Then there exists a constant $C>0$ depending only on $N$ such that 
\[
\begin{aligned}
	\sup_{ y\in \overline{B(x_{k,i}, ( t/4)^{1/2} )}} K(x,y,t)
	&\leq 
	C
	\inf_{ z\in \overline{B(x_{k,i}, ( t/4)^{1/2} )}} K(x,z,2t)  \\
	&\leq 
	\frac{ C }
	{ \Vol( \overline{ B(x_{k,i}, ( t/4)^{1/2} )} ) } 
	\int_{ \overline{B(x_{k,i}, ( t/4)^{1/2} )} } K(x,z,2t) dV_g(z). 
\end{aligned}
\]
Since $(t/4)^{1/2}<2\rho_\infty$, 
the volume comparison theorem \eqref{prep5} shows 
the existence of a constant $C>0$ depending only on $N$ such that 
\[
\Vol( \overline{ B(x_{k,i}, ( t/4)^{1/2} )} )\geq C t^\frac{N}{2}. 
\]

From the above computations, it follows that 
\[
\begin{aligned}
	\int_{M} K(x,y,t) d\mu(y) 
	&\leq 
	C
	 t^{-\frac{N}{2}}  \sup_{z\in M} 
	\mu( B(z, t^\frac{1}{2}) ) 
	\sum_{k=1}^{k_0} \sum_{i=1}^\infty 
	\int_{ \overline{B(x_{k,i}, ( t/4)^{1/2} )} } K(x,z,2t) dV_g(z) \\
	&\leq 
	C
	 t^{-\frac{N}{2}}  \sup_{z\in M} 
	\mu( B(z, t^\frac{1}{2}) ) 
	\sum_{k=1}^{k_0} \int_M K(x,z,2t) dV_g(z) \\
	&\leq 
	k_{0}C t^{-\frac{N}{2}}  
	\sup_{z\in M} \mu( B(z, t^\frac{1}{2}) )
\end{aligned}
\]
for $(x,t)\in M\times(0, 16 \rho_\infty^2)$, where $C>0$ depends only on $N$. 
Hence we obtain \eqref{eq:lines} in the case $\rho_\infty<\infty$. 
Since $C$ is independent of $\rho_\infty$, \eqref{eq:lines} also holds 
in the case $\rho_\infty=\infty$.  
\end{proof}

\subsection{The case $p<p_F$}\label{subsec:subcri}
We prove Theorem \ref{th:necsha}. 

\begin{proof}[Proof of Theorem \ref{th:necsha}]
Let $\mu$ and $\tilde \rho\in(0,\rho_T]$ satisfy \eqref{eq:sha1}, that is, 
\[\sup_{z\in M} \mu( B(z,\tilde \rho) )\leq c \tilde \rho^{N-\frac{2}{p-1}}, \]
where $c>0$ is a constant chosen later. Set 
\begin{equation}\label{eq:linsuper}
	\overline{u}(x,t) := 2U(x,t), \qquad 
	U(x,t):= \int_{M} K(x,y,t) d\mu(y). 
\end{equation}
We check that $\overline{u}$ is a supersolution of \eqref{eq:integ} 
in $M\times[0,\tilde \rho^2)$. 
From Fubini's theorem and the semigroup property of $K$, it follows that 
\begin{equation}\label{eq:UFubi}
\begin{aligned}
	\Psi[\overline{u}](x,t) &=  
	U(x,t) + 2^p \int_0^t \int_{M} K(x,y,t-s) U(y,s)^{p-1}U(y,s) dV_{g}(y) ds \\
	&\leq 
	U(x,t) + 2^p \int_0^t \|U(\,\cdot\,,s)\|_{L^\infty}^{p-1} 
	\int_{M} K(x,y,t-s) \int_{M} K(y,z,s) d\mu(z) dV_{g}(y) ds \\
	&=
	U(x,t) + 2^p U(x,t) \int_0^t \|U(\,\cdot\,,s)\|_{L^\infty}^{p-1} ds 
\end{aligned}
\end{equation}
for $(x,t)\in M\times (0,\tilde \rho^2)$. 
By Lemma \ref{lem:linheat} with $\tilde \rho$ 
($\leq 4 \rho_\infty$), we have 
\[
	U(x,t)
	\leq
	C t^{-\frac{N}{2}} 
	\sup_{z\in M} \mu( B(z,t^\frac{1}{2}) ) 
	\leq 
	C t^{-\frac{N}{2}} 
	\sup_{z\in M} \mu( B(z,\tilde \rho) ) 
\]
for $(x,t)\in M\times(0,\tilde \rho^2)$. 
This together with \eqref{eq:sha1} and $1-N(p-1)/2>0$ gives 
\[
	\int_0^t \|U(\,\cdot\,,s)\|_{L^\infty}^{p-1} ds
	\leq 
	C c^{p-1} t^{1-\frac{N(p-1)}{2}} \tilde \rho^{N(p-1)-2}
	\leq 
	C c^{p-1}
\]
for $(x,t)\in M\times(0,\tilde \rho^2)$, 
where $C>0$ is a constant depending only on $N$ and $p$. 
Therefore, by choosing $c$ satisfying $2^p C c^{p-1} \leq 1$, we see that $\overline{u}$ is a supersolution of \eqref{eq:integ} 
in $M\times[0,\tilde \rho^2)$.

By Lemma \ref{lem:ite}, we obtain an integral solution $u$ 
of \eqref{eq:main} in $M\times[0,\tilde \rho^2)$ satisfying 
$0\leq u\leq \overline{u}$. 
In particular, $u\in L^\infty_\loc((0,\tilde \rho^2); L^\infty(M) )$. 
By $u=\Psi[u]$, Fubini's theorem and the semigroup property of $K$, 
we can see that $u$ satisfies 
\begin{equation}\label{eq:uint}
	u(x,t) = \int_{M} K(x,y,t-\tau) u(y,\tau) dV_g(y) 
	+ \int_\tau^t \int_{M} K(x,y,t-s) u(y,s)^p dV_{g}(y) ds
\end{equation}
for a.e. $x\in M$ and $0<\tau<t<\tilde \rho^2$. 
Hence we easily see that $u\in C^{2,1}(M\times(0,\tilde \rho^2))$.

It remains to prove that $u$ satisfies \eqref{eq:ini}. 
To see this, it suffices to check \eqref{eq:ini} 
for any nonnegative function $\psi\in C_0(M)$. 
We modify the argument of \cite[Lemma 2.4]{HI18}. 
Since $u$ satisfies \eqref{eq:integ}, we have 
\begin{equation}\label{eq:upsiV}
\begin{aligned}
	\int_M u(\,\cdot\,,t) \psi dV_g 
	&= 
	\int_M \int_M K(x,y,t) \psi(x) dV_g(x) d\mu(y) \\
	&\quad + 
	\int_M \int_0^t  \int_M 
	K(x,y,t-s) u(y,s)^p dV_g(y) ds \psi(x) dV_g(x) \\
\end{aligned}
\end{equation}
for $x\in M$ and $0<t<\tilde \rho^2$. 
We claim that the second term in the right-hand side of \eqref{eq:upsiV} 
converges to $0$ as $t\to0$. 
Write $M_0:= \supp\psi$. 
From \eqref{eq:uint} and Fubini's theorem, it follows that 
\[
\begin{aligned}
	&\int_M u(\,\cdot\,,t) \psi dV_g \\
	&= 
	\int_M \int_M K(y,x,t-\tau) \psi(x) dV_g(x) u(y,\tau) dV_g(y) \\
	&\quad + 
	\int_M \int_\tau^t  \int_M 
	K(x,y,t-s) u(y,s)^p dV_g(y) ds \psi(x) dV_g(x) \\
	&\geq 
	\int_{M_0} u(\,\cdot\,,\tau) \psi dV_g 
	- 
	\int_{M_0} u(\,\cdot\,,\tau) dV_g 
	\left\| \int_M K(\,\cdot\,,x,t-\tau) \psi(x) dV_g(x) -\psi\right\|_{L^\infty(M_0)} \\
	&\quad + 
	\int_M \int_\tau^t  \int_M 
	K(x,y,t-s) u(y,s)^p dV_g(y) ds \psi(x) dV_g(x)
\end{aligned}
\]
for $x\in M$ and $0<\tau<t<\tilde \rho^2/3$. 
Since $u$ satisfies \eqref{eq:fujita} in $M\times (0,\tilde \rho^2)$, 
by Lemma \ref{lem:exini}, 
there exists a nonnegative Radon measure $\nu$ on $M$ 
such that $u$ satisfies \eqref{eq:ini} with $\mu$ replaced by $\nu$. 
Thus, using \eqref{eq:M0usup} and letting $\tau\to0$ yield 
\[
\begin{aligned}
	\int_M \psi d\nu &\geq 
	\int_{M_0} \psi d\nu 
	- 
	\sup_{0<\tau<\tilde \rho^2/3}\int_{M_0} u(\,\cdot\,,\tau) dV_g
	\left\| \int_M K(\,\cdot\,,x,t) \psi(x) dV_g(x) -\psi\right\|_{L^\infty(M_0)}  \\
	&\quad + 
	\int_M \int_0^t  \int_M 
	K(x,y,t-s) u(y,s)^p dV_g(y) ds \psi(x) dV_g(x)
\end{aligned}
\]
for $x\in M$ and $0<t<\tilde \rho^2/3$. 
Since $M_0= \supp\psi$, we see that 
\[
\begin{aligned}
	&\int_M \int_0^t  \int_M 
	K(x,y,t-s) u(y,s)^p dV_g(y) ds \psi(x) dV_g(x) \\
	&\leq 
	\sup_{0<\tau<\tilde \rho^2/3}\int_{M_0} u(\,\cdot\,,\tau) dV_g
	\left\| \int_M K(\,\cdot\,,x,t) \psi(x) dV_g(x) -\psi\right\|_{L^\infty(M_0)}. 
\end{aligned}
\]
By \cite[Theorem 7.13]{Grbook}, the right-hand side converges to $0$ 
as $t\to0$. 
Hence by letting $t\to0$, we can see that 
the second term in the right-hand side of \eqref{eq:upsiV} 
converges to $0$ as $t\to0$. Then the claim follows.

Finally, let us consider 
the first term in the right-hand side of \eqref{eq:upsiV}. 
Fix $z_0\in M$ and $R\geq \sqrt{3}$ so that $M_{0}=\supp \psi\subset B(z_0,R)$. 
Since $\mu$ is a Radon measure, 
we have $\mu(B(z_0,R))<\infty$. Then,  
\[
	\int_{M}\psi(y)d\mu(y)\leq\|\psi\|_{L^\infty(M_0)}\mu(B(z_0,R))<\infty. 
\]
Thus, $\psi\in L^1_\mu(M)$, 
that is, the nonnegative function $\psi$ is $\mu$-integrable. 
We check the $\mu$-integrability of 
$y\mapsto \int_{M}K(x,y,t)\psi(x)dV_{g}(x)$. 
In the case $y\in B(z_0,2R)$, we see that 
\[
	\int_{M}K(x,y,t)\psi(x)dV_{g}(x) \leq \|\psi\|_{L^\infty(M_0)}<\infty. 
\]
On the other hand,  
in the case $y\notin B(z_0,2R)$, 
the upper bound of the heat kernel \eqref{prep9} yields 
\[
\begin{aligned}
\int_{M}K(x,y,t)\psi(x)dV_{g}(x)&= \int_{B(z_0,R)}K(x,y,t)\psi(x)dV_{g}(x)\\
&\leq \|\psi\|_{L^\infty(M_0)} 
\int_{B(z_0,R)} Ct^{-N/2}\exp\left(-\frac{d(x,y)^2}{(4+(1/2))t}\right)dV_{g}(x), 
\end{aligned}
\]
where $C>0$ depends only on $N$, $\kappa$ and $\inj(M)$. 
Since $x\in B(z_0,R)$ and $y\notin B(z_0,2R)$ with $R\geq \sqrt{3}$, 
we have $d(x,y)^2\geq (3/4)\{d(x,y)^2+1\}$. 
Thus, 
\[
\begin{aligned}
	t^{-\frac{N}{2}}\exp\left(-\frac{d(x,y)^2}{(4+(1/2))t}\right)
	&\leq t^{-\frac{N}{2}}\exp\left(\frac{-1}{6t}\right) 
	\exp\left(-\frac{d(x,y)^2}{6t}\right)
	\leq C\exp\left(-\frac{d(x,y)^2}{6t}\right) \\
	&\leq C\exp\left(-\frac{d(x,y)^2}{2\tilde \rho^2}\right)
	=C (\tilde \rho^2)^{\frac{N}{2}} (\tilde \rho^2)^{-\frac{N}{2}}
	\exp\left(-\frac{d(x,y)^2}{2 \tilde \rho^2}\right)
\end{aligned}
\]
for $0<t<\tilde \rho^2/3$. 
Combining the above inequalities, we have 
\[
	\int_{M}K(x,y,t)\psi(x)dV_{g}(x)
	\leq C\|\psi\|_{L^\infty(M_0)}  
	\int_{B(z_0,R)}(\tilde \rho^2)^{-\frac{N}{2}}
	\exp\left(-\frac{d(x,y)^2}{2\tilde \rho^2}\right)dV_{g}(x)
\]
for some $C>0$ depending only on $N$, $\kappa$, $\inj(M)$ 
and $\tilde \rho$. 
Then, by the lower bound of the heat kernel \eqref{prep10}, we have 
\begin{equation}\label{intKsupK}
	\int_{M}K(x,y,t)\psi(x)dV_{g}(x)
	\leq C\|\psi\|_{L^\infty(M_0)}  \int_{B(z_0,R)}K(x,y,\tilde \rho^2)dV_{g}(x). 
\end{equation}
Thus, dividing the integral with respect to $\mu$ into $B(z_0,2R)$ and $B(z_0,2R)^{c}$, we see that 
\[
\begin{aligned}
\int_{M}\int_{M}K(x,y,t)\psi(x)dV_{g}(x)d\mu(y)
&=\int_{y\in B(z_0,2R)}\int_{x\in B(z_0,R)}K(x,y,t)\psi(x)dV_{g}(x)d\mu(y)\\
&+\int_{y\notin B(z_0,2R)}\int_{x\in B(z_0,R)}K(x,y,t)\psi(x)dV_{g}(x)d\mu(y). 
\end{aligned}
\]
The first term is easily estimated from above as 
\[
	\int_{y\in B(z_0,2R)}\int_{x\in B(z_0,R)}K(x,y,t)\psi(x)dV_{g}(x)d\mu(y)
	\leq \|\psi\|_{L^\infty(M_0)}\mu(B(z_0,2R))<\infty. 
\]
On the second term, we have 
\[
\begin{aligned}
	&\int_{y\notin B(z_0,2R)}\int_{x\in B(z_0,R)}K(x,y,t)\psi(x)dV_{g}(x)d\mu(y)\\
	&\leq C\|\psi\|_{L^\infty(M_0)} 
	\int_{x\in B(z_0,R)}\int_{y\in M}K(x,y,\tilde \rho^2)d\mu(y)dV_{g}(x)\\
	&\leq   C \|\psi\|_{L^\infty(M_0)} 
	\int_{x\in B(z_0,R)} (\tilde \rho^2)^{-\frac{N}{2}} 
	\sup_{z\in M} \mu( B(z,\tilde \rho) )dV_{g}(x), 
\end{aligned}
\]
where the first inequality follows from \eqref{intKsupK} 
and the second one follows from Lemma \ref{lem:linheat}. 
Thus, combining the assumption \eqref{eq:sha1} for $\mu$, we have 
\[
	\int_{x\in B(z_0,R)} (\tilde \rho^2)^{-\frac{N}{2}} 
	\sup_{z\in M} \mu( B(z,\tilde \rho) )dV_{g}(x)
	\leq c \tilde \rho^{-\frac{2}{p-1}}\Vol(B(z_0,R))<\infty.
\]
Thus, we see that $\int_{M}K(x,\,\cdot\,,t)\psi(x)dV_{g}(x)\in L^{1}_{\mu}(M)$. 
Since the pointwise convergence of $\int_{M}K(x,y,t)\psi(x)dV_{g}(x)$ to $\psi(y)$ as $t\to 0$ is established 
(see  \cite[Theorem 7.13]{Grbook} for instance), by Lebesgue's dominated convergence theorem, we can say that 
\[
	\int_M \left| \int_M K(y,x,t) \psi(x) dV_g(x) -\psi(y)  \right| d\mu(y)
	\to0 
	\quad \mbox{ as }t\to0. 
\]
This immediately implies that 
\[
\begin{aligned}
	\int_M \int_M K(x,y,t) \psi(x) dV_g(x) d\mu(y)
	&= \int_M \psi d\mu
	+ \int_M \left( \int_M K(y,x,t) \psi(x) dV_g(x) -\psi(y) \right)d\mu(y)  \\
	&\to \int_M \psi d\mu 
	\quad \mbox{ as }t\to0. 
\end{aligned}
\]
Hence by letting $t\to0$ in \eqref{eq:upsiV}, we obtain \eqref{eq:ini}. 
The proof is complete. 
\end{proof}

\subsection{The case $p\geq p_F$}\label{subsec:cri}
In this subsection, 
let $N$, $p$ and $T$ be as in Theorem \ref{th:necshacri}. 
Let $\eta\in C^\infty(\R)$ satisfy 
$0\leq \eta \leq 1$, $\eta'\leq 0$, 
$\eta(z)=0$ for $z\geq 1$,  
$\eta(z)=1$ for $z\leq 1/2$ and 
$\sup\{z\in\mathbb{R}; \eta(z)>0\}=1$. 
Fix $z_0\in M$ and define $f \in L^1_\loc(M)$ by $f(x):=\tilde{f}(r(x))$, where $r(x):=d(z_0,x)$ and 
\[
	\tilde{f}(r):= 
	\left\{ 
	\begin{aligned}
	& r^{-\frac{2}{p-1}}\eta( \rho_T^{-1} r ) 
	&&\mbox{ if }p>p_F, \\
	& (\rho_T r^{-1})^N 
	( \log (e^2 + \rho_T r^{-1} ))^{-\frac{N}{2}-1}
	\eta( \rho_T^{-1} r)
	&&\mbox{ if }p=p_F. 
	\end{aligned}
	\right.
\]
We remark that $\tilde{f}(r)$ is strictly decreasing for 
$r\in(0,\rho_T)$. 
We set a nonnegative Radon measure $\mu_f$ 
by $\mu_f(A):=\int_A f dV_g$. 
We prepare estimates of $\mu_f$.

\begin{lemma}\label{lem:lumu}
There exists a constant $C>1$ depending only on $N$ and $p$ such that 
the following (i) and (ii) hold: 
\begin{itemize}
\item[(i)]
If $p>p_F$, then 
$\displaystyle C^{-1}  \rho^{N-\frac{2}{p-1}} 
\leq \sup_{z\in M} \mu_f( B(z,\rho) ) 
\leq C \rho^{N-\frac{2}{p-1}}$ 
for any $0<\rho<\rho_T$. 
\item[(ii)]
If $p=p_F$, then 
$\displaystyle C^{-1} \rho_T^N ( \log (e+\rho_T\rho^{-1}) )^{-\frac{N}{2}}
\leq \sup_{z\in M} \mu_f ( B(z,\rho) ) 
\leq C \rho_T^N ( \log (e+\rho_T\rho^{-1}) )^{-\frac{N}{2}}$ 
for any $0<\rho< \rho_T$. 
\end{itemize}
\end{lemma}

\begin{proof}
We prepare some estimates which hold for both (i) and (ii). 
Let $0<\rho<\rho_T$. 
We estimate $\sup_{z\in M} \mu_f( B(z,\rho) )$ from above. 
Fix $z\in M$ and $0<\rho<\rho_T$. 
By the coarea formula, we have 
\[
\mu_{f}(B(z,\rho))=\int_{B(z,\rho)}f(x)dV_{g}(x)=\int_{0}^{\infty}\Vol\left(f^{-1}((\lambda,\infty])\cap B(z,\rho)\right)d\lambda. 
\]
It is clear that 
\[
\begin{aligned}
&\int_{0}^{\infty}\Vol\left(f^{-1}((\lambda,\infty])\cap B(z,\rho)\right)d\lambda\\
&=
\int_{0}^{\tilde{f}(\rho)}\Vol\left(f^{-1}((\lambda,\infty])\cap B(z,\rho)\right)d\lambda
+
\int_{\tilde{f}(\rho)}^{\infty}\Vol\left(f^{-1}((\lambda,\infty])\cap B(z,\rho)\right)d\lambda. 
\end{aligned}
\]
The first term is $0$ if $\rho\geq \rho_T$. 
In the case $\rho<\rho_T$, the first term is estimated as 
\[
\begin{aligned}
\int_{0}^{\tilde{f}(\rho)}\Vol\left(f^{-1}((\lambda,\infty])\cap B(z,\rho)\right)d\lambda
\leq \int_{0}^{\tilde{f}(\rho)}\Vol\left(B(z,\rho)\right)d\lambda
\leq C(N)\rho^{N}\tilde{f}(\rho), 
\end{aligned}
\]
where the second inequality follows from the volume comparison theorem 
\eqref{prep3} with $C(N):=2^{N-1}N^{-1}\Area(\bS^{N-1})$. 
To estimate the second term, we note that 
$f^{-1}((\lambda,\infty])=B(z_0,\tilde{f}^{-1}(\lambda))$ 
for $\lambda>0$. 
Then, similarly to the estimate of the first term, we see that 
\[
	\int_{\tilde{f}(\rho)}^{\infty}
	\Vol\left(f^{-1}((\lambda,\infty])\cap B(z,\rho) \right)d\lambda
	\leq C(N)\int_{\tilde{f}(\rho)}^{\infty}(\tilde{f}^{-1}(\lambda))^{N}d\lambda 
	=C(N)\int_{\rho}^{0}r^{N}\tilde{f}'(r)dr, 
\]
where the last equality follows from the change of variable 
with $\lambda:=\tilde{f}(r)$. 
We have 
\[
\begin{aligned}
	C(N)\int_{\rho}^{0}r^{N}\tilde{f}'(r)dr
	&=C(N)\left. \left(r^N\tilde{f}(r)\right)
	\right|_{r=\rho}^{r\to 0}-C(N)N\int_{\rho}^{0}r^{N-1}\tilde{f}(r)dr\\
	&=-C(N)\rho^N\tilde{f}(\rho)+C(N)N\int_{0}^{\rho}r^{N-1}\tilde{f}(r)dr. 
\end{aligned}
\]
Thus, for $p\geq p_F$, we have 
\begin{equation}\label{ballcheck}
	\mu_{f}(B(z,\rho))\leq C(N)N\int_{0}^{\rho}r^{N-1}\tilde{f}(r)dr. 
\end{equation}

(i) In the case $p>p_{F}$, we have
\[
	C(N)N\int_{0}^{\rho}r^{N-1}\tilde{f}(r)dr
	\leq C(N)N\int_{0}^{\rho}r^{N-1-\frac{2}{p-1}}dr
	=\frac{C(N)N}{N-(2/(p-1))}\rho^{N-\frac{2}{p-1}}. 
\]
This together with \eqref{ballcheck} shows 
the desired estimate from above. 
For the estimate from below, 
in Riemannian normal coordinates $(x_1,\ldots,x_N)$ centered at $z_0$, 
we see that 
\[
	\sup_{z\in M} \mu_f( B(z,\rho) ) 
	\geq 
	\int_{B(z_0,\rho/2)} d(z_0,x)^{-\frac{2}{p-1}} dV_g(x)  
	= 
	\int_{B(0,\rho/2)} |x|^{-\frac{2}{p-1}} \sqrt{\det(g_{ij}(x))} dx. 
\]
By the volume comparison theorem \eqref{prep4}, 
we have $\sup_{z\in M} \mu_f( B(z,\rho) ) 
\geq C^{-1} \rho^{N-2/(p-1)}$. 
Hence (i) follows.

(ii) In the case $p=p_{F}$, 
from the volume comparison theorem \eqref{prep4}, it follows that 
\[
\begin{aligned}
	\sup_{z\in M} \mu_f( B(z,\rho) ) 
	&\geq 
	\rho_T^N \int_{B(z_0,\rho/2)} 
	d(z_0,x)^{-N} ( \log (e^2 + \rho_T d(z_0,x)^{-1} )^{-\frac{N}{2}-1} dV_g(x)  \\
	&\geq 
	C^{-1} \rho_T^N \int_{B(0,\rho/2)} 
	\frac{N}{2} \left( \frac{\rho_T}{e^2|x| +\rho_T} \right) |x|^{-N} 
	( \log (e^2 + \rho_T |x|^{-1} )^{-\frac{N}{2}-1} dx \\
	&= 
	C^{-1} \Area(\bS^{N-1})
	 \rho_T^N ( \log (e^2+2\rho_T\rho^{-1}) )^{-\frac{N}{2}} 
	 \geq 
	C^{-1} \rho_T^N  ( \log (e+\rho_T\rho^{-1}) )^{-\frac{N}{2}}, 
\end{aligned}
\]
where $C>1$ depends only on $N$. 
On the other hand, by \eqref{ballcheck}, we have 
\[
\begin{aligned}
	\sup_{z\in M} \mu_f( B(z,\rho) ) 
	&\leq 
	C \int_{0}^{\rho}r^{N-1}\tilde{f}(r) dr 
	\leq C \int_{0}^{\rho}r^{N-1}(\rho_T r^{-1})^{N}
	( \log (e^2 + \rho_T r^{-1} )^{-\frac{N}{2}-1}dr \\
	&=\frac{2C}{N} \rho_T^{N}\int_{0}^{\rho} 
	\left( \frac{e^2 r+\rho_T}{\rho_T}\right)
	\frac{N}{2} \left( \frac{\rho_T r^{-1}}{e^2 r +\rho_T} \right)
	( \log (e^2 + \rho_T r^{-1} )^{-\frac{N}{2}-1} dr  \\
	&\leq 
	C \rho_T^N 
	( \log (e^2 +\rho_T \rho^{-1}) )^{-\frac{N}{2}} 
	\leq 
	C \rho_T^N ( \log (e+\rho_T\rho^{-1}) )^{-\frac{N}{2}}. 
\end{aligned}
\]
Hence (ii) follows. 
\end{proof}

We prepare auxiliary functions for proving Theorem \ref{th:necshacri}. 
For $1<\alpha<p$ and $0<\beta<N/2$, 
we define a convex and strictly increasing function $h$ by 
\begin{equation}\label{eq:hdefff}
	h(z):= 
	\left\{ 
	\begin{aligned}
	&z^\alpha && \mbox{ if }p>p_F, \\
	&z(\log(A+z))^\beta && \mbox{ if }p=p_F. 
	\end{aligned}
	\right.
\end{equation}
In the case $p=p_F$, 
we fix $A>e^{N}$ such that 
$z\mapsto z^p/h(z)= z^{p-1} (\log(A + z))^{-\beta}$ is strictly increasing. 
Then, $z\mapsto z^p/h(z)$ and $z\mapsto h(z)/z$ are strictly increasing 
in each case. 
Note that the inverse function $h^{-1}$ of $h$ exists and is also strictly increasing. 
For $c>0$, define 
\begin{equation}\label{eq:super}
	\overline{u}(x,t):= 
	2 c U(x,t), \qquad 
	U(x,t) := h^{-1} \left( \int_M K(x,y,t) h( f(y) )  dV_g(y) \right). 
\end{equation}
Let us estimate $U$.

\begin{lemma}\label{lem:linex}
There exists a constant $C>0$ depending only on $N$ and $p$ such that 
\[
\begin{aligned}
	U(x,t) &\leq 
	\left\{
	\begin{aligned}
	& C t^{-\frac{1}{p-1}} && \mbox{ if }p>p_F,   \\
	& C \rho_T^N t^{-\frac{N}{2}} 
	(\log(e+\rho_T t^{-\frac{1}{2}}))^{-\frac{N}{2}} 
	&& \mbox{ if }p=p_F,   \\
	\end{aligned}
	\right.
\end{aligned}
\]
for $(x,t)\in M\times (0,\rho_T^2)$. 
\end{lemma}

\begin{proof}
Let $0<t< \rho_T^2$. 
We first estimate $h(U)$. 
Since $h(U(x,t)) = \int_M K(x,y,t) h( f(y) )  dV_g(y)$, 
Lemma \ref{lem:linheat} gives 
\[
\begin{aligned}
	h(U(x,t))
	\leq C  t^{-\frac{N}{2}} 
	\sup_{z\in M} \int_{B(z,t^\frac{1}{2})} h( f(y)) dV_g(y). 
\end{aligned}
\]
Since $h(f(y))=h(\tilde{f}(r(y)))$ and $h(\tilde{f}(r))$ is strictly decreasing for $r\in(0,\rho_T)$, 
we can obtain \eqref{ballcheck} with replacing $\tilde{f}$ with $h(\tilde{f})$ and have 
\[
\begin{aligned}
	&\int_{B(z,t^\frac{1}{2})} h( f(y)) dV_g(y) \leq C\int_{0}^{t^{\frac{1}{2}}}r^{N-1}h(\tilde{f}(r))dr\\
	&\leq 
	\left\{ 
	\begin{aligned}
	&C \int_0^{t^\frac{1}{2}} r^{N-\frac{2\alpha}{p-1}-1} dr 
	&&\mbox{ if }p>p_F,   \\
	&C\rho_T^N \left( \frac{et^{1/2}}{\rho_T} +1 \right) 
	\int_0^{t^\frac{1}{2}}
	\frac{\rho_T}{er+\rho_T} r^{-1}
	(\log(e+\rho_T r^{-1}))^{-\frac{N}{2}-1+\beta} dr 
	&&\mbox{ if }p=p_F.
	\end{aligned}
	\right.
	\end{aligned}
\]
Then by $t< \rho_T^2$, we see that 
\[
\begin{aligned}
	\int_{B(z,t^\frac{1}{2})} h( f(y)) dV_g(y) 
	\leq 
	\left\{ 
	\begin{aligned}
	&C t^{\frac{N}{2}-\frac{\alpha}{p-1}} 
	&&\mbox{ if }p>p_F,  \\
	&C \rho_T^N (\log(e+\rho_T t^{-\frac{1}{2}}))^{-\frac{N}{2}+\beta} 
	&&\mbox{ if }p=p_F,  
	\end{aligned}
	\right.  \\
\end{aligned}
\]
where $C>0$ is a constant depending only on $N$ and $p$. 
Thus, 
\begin{equation}\label{eq:hfygcom}
	h(U(x,t))
	\leq 
	\left\{ 
	\begin{aligned}
	&C t^{-\frac{\alpha}{p-1}} 
	&&\mbox{ if }p>p_F,   \\
	&C (\rho_T t^{-\frac{1}{2}})^N (\log(e+\rho_T t^{-\frac{1}{2}}))^{-\frac{N}{2}+\beta} 
	&&\mbox{ if }p=p_F. 
	\end{aligned}
	\right. 
\end{equation}
Fundamental computations show that 
\begin{equation}\label{eq:hinves}
	h^{-1}(z)
	\left\{ 
	\begin{aligned}
	&=z^{1/\alpha}  &&\mbox{ if }p>p_F, \\
	&\leq C z (\log(A + z))^{-\beta}  &&\mbox{ if }p=p_F.  
	\end{aligned}
	\right. 
\end{equation}
From this, it follows that the desired inequality holds. 
\end{proof}

We are now in a position to prove Theorem \ref{th:necshacri}. 

\begin{proof}[Proof of Theorem \ref{th:necshacri}]
Fix $c>0$. Define $\mu(A):=\int_A c f dV_g$ and $\overline{u}$ by \eqref{eq:super}. 
Since $h$ is convex, Jensen's inequality gives 
\[
	h\left( \frac{ \int_{M} K(x,y,t)  f(y) dV_{g}(y) }
	{\int_{M} K(x,y,t) dV(y)} \right)
	\leq
	\frac{ \int_{M} K(x,y,t) h( f(y)) dV_{g}(y) }
	{\int_{M} K(x,y,t) dV_{g}(y)}. 
\]
This together with $\int_{M} K(x,y,t) dV_{g}(y)=1$ and $d\mu = c f dV_g$ yields 
\[
	\int_{M} K(x,y,t) d\mu(y)
	\leq
	c h^{-1} \left( \int_{M} K(x,y,t) h( f(y) ) dV_{g}(y) \right)
	= c U(x,t), 
\]
and so 
\begin{equation}\label{eq:ouesti}
	\Psi[\overline{u}]
	\leq 
	c U + 2^p c^p \int_0^t \int_{M} K(x,y,t-s) U(y,s)^p dV_{g}(y) ds
	=: c U + 2^p c^p I. 
\end{equation}

Let us estimate $I$. Write $\|\cdot\|_\infty:=\|\cdot\|_{L^\infty(M)}$. 
From $h(U(y,s)) = \int_M K(y,z,s) h( f(z) )  dV_g(z)$, 
Fubini's theorem and the semigroup property of $K$, it follows that 
\[
\begin{aligned}
	I &= 
	\int_0^t \int_{M} K(x,y,t-s) \frac{U(y,s)^p}{h(U(y,s))} 
	h(U(y,s)) 
	dV_{g}(y) ds \\
	&\leq 
	\int_0^t \left\| \frac{U(\,\cdot\,,s)^p}{h(U(\,\cdot\,,s))} \right\|_\infty
	\int_{M} K(x,y,t-s)  
	\int_{M} K(y,z,s) h( f(z) ) dV_{g}(z) dV_{g}(y) ds  \\
	&= 
	\int_0^t \left\| \frac{U(\,\cdot\,,s)^p}{h(U(\,\cdot\,,s))} \right\|_\infty ds 
	\int_{M} K(x,z,t) h( f(z) ) dV_{g}(z) 
	= h(U(x,t)) \int_0^t \left\| \frac{U(\,\cdot\,,s)^p}{h(U(\,\cdot\,,s))} \right\|_\infty ds. 
\end{aligned}
\]
Therefore, 
\begin{equation}\label{eq:ouestij}
\begin{aligned}
	I &\leq 
	U(x,t) \left\| \frac{h(U(\,\cdot\,,t))}{U(\,\cdot\,,t)} \right\|_\infty 
	\int_0^t \left\| \frac{U(\,\cdot\,,s)^p}{h(U(\,\cdot\,,s))} \right\|_\infty ds  
	=: U(x,t) J(t). 
\end{aligned}
\end{equation}

We estimate $J$ for $0<t<\rho_T^2$. Since 
$z\mapsto z^p/h(z)$ and $z\mapsto h(z)/z$ are increasing, 
Lemma \ref{lem:linex} and direct computations yield 
\begin{equation}\label{eq:Jtesti}
\begin{aligned}
	J(t) 
	&\leq 
	\left\{ 
	\begin{aligned}
	& \|U(\,\cdot\,,t)\|_\infty^{\alpha-1}  
	\int_0^t \|U(\,\cdot\,,s)\|_\infty^{p-\alpha} ds 
	&& (p>p_F) \\
	& (\log(A + \|U(\,\cdot\,,t)\|_\infty ))^\beta
	\int_0^t \|U(\,\cdot\,,s)\|_\infty^{p-1} 
	(\log(A + \|U(\,\cdot\,,s)\|_\infty))^{-\beta} ds 
	&& (p=p_F)
	\end{aligned}
	\right.\\
	&\leq 
	\left\{ 
	\begin{aligned}
	& C t^{-\frac{\alpha-1}{p-1}} 
	\int_0^t s^{-\frac{p-\alpha}{p-1}} ds 
	&& (p>p_F) \\
	& 
	C \rho_T^2 (\log(e + \rho_T t^{-\frac{1}{2}} ))^\beta
	\int_0^t  s^{-1} 
	(\log(e+ \rho_T s^{-\frac{1}{2}}))^{-1-\beta} ds 
	&& (p=p_F) 
	\end{aligned}
	\right. \\
	&\leq 
	\left\{ 
	\begin{aligned}
	& C  && (p>p_F) \\
	& C \rho_T^2 && (p=p_F)
	\end{aligned}
	\right. 
\end{aligned}
\end{equation}
for $0<t<\rho_T^2$, where $C>0$ is a constant depending only on $N$ and $p$.

The above estimates together with $\rho_T\leq T^{1/2}$ show that 
\[
	\Psi[\overline{u}] 
	\leq 
	\left\{ 
	\begin{aligned}
	& c U + 2^p C c^p U && \mbox{ if }p>p_F,  \\
	& c U + 2^p C T c^p  U && \mbox{ if }p=p_F, \\
	\end{aligned}
	\right.
\]
in $M\times (0,\rho_T^2)$. 
Hence we deduce that $\overline{u}$ is a supersolution 
of \eqref{eq:integ} in $M\times[0,\rho_T^2)$ provided that 
$2^p C c^{p-1}\leq 1$ if $p>p_F$ and  
$2^p C Tc^{2/N}\leq 1$ if $p=p_F$. 
Lemma \ref{lem:ite} gives an integral solution $u$ 
of \eqref{eq:main} in $M\times[0,\rho_T^2)$. 
Moreover, the estimates of $\mu$ in the statement (i) and (ii) 
follow from Lemma \ref{lem:lumu}. 
The rest of the proof to see that $u$ is a solution in $M\times[0,\rho_{T}^2)$ is the same as that of Theorem \ref{th:necsha}. 
\end{proof}

\subsection{Sufficient conditions}\label{subsec:suff}
In the previous subsections, we proved the sharpness of the necessary conditions. 
Even though Theorem \ref{th:suf} does not say the sharpness directly, 
we prove the theorem in this subsection, since all estimates for the proof have been prepared. 
The proof of Theorem \ref{th:suf} is similar 
to the proof of Theorems \ref{th:necsha}. 

\begin{proof}[Proof of Theorem \ref{th:suf}]
Let $\mu$ and $\tilde \rho \in(0,\rho_T]$ satisfy \eqref{eq:suf}, 
where $c$ is chosen later. 
Set $\overline{u}$ and $U$ be as in \eqref{eq:linsuper}. 
By \eqref{eq:UFubi}, we have 
\[
\begin{aligned}
	\Psi[\overline{u}](x,t) &=  
	U(x,t) + 2^p \int_0^t \int_{M} K(x,y,t-s) U(y,s)^p dV_{g}(y) ds  \\
	&\leq 
	U(x,t) + 2^p \left( \int_0^t \|U(\,\cdot\,,s)\|_\infty^{p-1} ds \right) U(x,t). 
\end{aligned}
\]
Lemma \ref{lem:linheat} and \eqref{eq:suf} 
with $\tilde \rho$ ($\leq \rho_T\leq 4\rho_\infty$) yield  
\[
	U(x,t)  \leq 
	C c t^{-\frac{1}{p-1}} 
	(\log(e+ \tilde \rho t^{-\frac{1}{2}})^{-\frac{1}{p-1}-\eps} 
	\quad \mbox{ for }(x,t)\in M\times (0,\tilde \rho^2), 
\]
and so 
\[
\begin{aligned}
	\int_0^t \|U(\,\cdot\,,s)\|_\infty^{p-1} ds  
	&\leq 
	C c^{p-1} 
	\int_0^t s^{-1} (\log(e+\tilde \rho s^{-\frac{1}{2}}))^{-1-(p-1)\eps} ds \\
	&\leq 
	C c^{p-1}  (\log(e+\tilde \rho t^{-\frac{1}{2}}))^{-(p-1)\eps}
	\leq 
	C c^{p-1} 
\end{aligned}
\]
for $(x,t)\in M\times(0,\tilde \rho^2)$, 
where $C>0$ is a constant depending only on $N$ and $p$. 
If $c$ satisfies $2^{p}C c^{p-1}<1$, 
then $\overline{u}$ is a supersolution of \eqref{eq:integ} 
in $M\times[0,\tilde \rho^2)$. 
This together with Lemma \ref{lem:ite} shows 
the existence of an integral solution $u$ 
of \eqref{eq:main} in $M\times[0,\tilde \rho^2)$ . 
The rest of the proof to see that $u$ is a solution in $M\times[0,\tilde{\rho}^2)$ is the same as that of Theorem \ref{th:necsha}. 
\end{proof}

\section{Nonexistence of solutions}\label{sec:nonex}
To prove nonexistence, we use a Cantor type (fractal type) set and 
show the unboundedness of some fractional maximal operator. 
In the Euclidean case, 
such sets were introduced in Kan and the first author \cite{KT17} for one dimension 
and were generalized by the first author \cite{Ta16} for higher dimension. 
In this paper, we construct such Cantor type sets on Riemannian manifolds.

Throughout this section, we assume $p\geq p_F$ and write 
\begin{equation}\label{eq:phisec5def}
	\phi(\rho):= \log\left( e+\frac{1}{\rho} \right)^{-\frac{1}{p-1}} 
	\quad \mbox{ for }\rho>0. 
\end{equation}
We frequently use the following 
monotonicity properties for $\rho>0$: 
\begin{itemize}
\item
$\phi(\rho)$ is strictly increasing. 
\item
$\rho^{N-2/(p-1)} \phi(\rho)$ 
is strictly increasing. 
\item
There exists $0<\alpha<2/(p-1)$ such that 
$\rho^{-\alpha}\phi(\rho)$ 
is strictly decreasing. 
\item
If $p>p_F$, 
there exists $0<\beta<N-2/(p-1)$ such that 
$\rho^\beta \phi(\rho)$ is strictly increasing. 
\end{itemize}

\subsection{A Cantor type set}
In this subsection, we repeat the construction of 
a Cantor type set in \cite{KT17,Ta16}. 
For $0<\sigma<2^{-N}$, define a continuous function 
\[
	F(R):= R^{N-\frac{2}{p-1}} \phi(R) - \phi(2^{-1}) \sigma 
	\quad \mbox{ for }R>0. 
\]
By $\lim_{R\to 0}F(R)=-\phi(2^{-1})\sigma<0$ and 
\[
	F(2^{-1})= ( (2^{-1})^{N-\frac{2}{p-1}} - \sigma ) \phi(2^{-1}) 
	> ( 2^{-N+\frac{2}{p-1}} - 2^{-N} ) \phi(1/2) >0
\]
together with the monotonicity of $\rho^{N-2/(p-1)} \phi(\rho)$, 
the equation $F(R)=0$ has a unique solution $0<R(\sigma)<1/2$ 
for each $\sigma$. 
Set 
\[
	R_0:=1, \qquad 
	R_n:=R(2^{-Nn}) \quad \mbox{ for }n=1,2,\ldots. 
\]
We observe that if $\sigma$ becomes small, then 
$F$ becomes large and $R(\sigma)$ becomes small. 
Therefore, we see that 
\[
	\frac{1}{2} > R_1 > R_2 >\cdots > R_n \to0
	\quad \mbox{ as }n\to\infty. 
\]
In addition, since $F(R(\sigma))=0$, we have 
\[
	2^{Nn} R_n^{N-\frac{2}{p-1}} \phi(R_n) = \phi(2^{-1}) 
	\quad \mbox{ for }n=1,2,\ldots. 
\]

We denote the ratio of $R_n$ to $R_{n-1}$ by 
\[
	r_n:=\frac{R_n}{R_{n-1}} 
	\quad \mbox{ for } n=1,2,\ldots. 
\]
Let us estimate $r_n$. 
For $n\geq 2$, we have 
\[
	\phi(2^{-1}) = 
	2^{Nn} R_n^{N-\frac{2}{p-1}} \phi(R_n) 
	= 2^{N(n-1)} R_{n-1}^{N-\frac{2}{p-1}} \phi(R_{n-1}). 
\]
The monotonicity $R_n<R_{n-1}$ gives 
\[
\begin{aligned}
	2^{Nn} R_n^{N-\frac{2}{p-1}} \phi(R_n) 
	&= 
	2^{N(n-1)} R_{n-1}^{N-\frac{2}{p-1}+\alpha} R_{n-1}^{-\alpha} \phi(R_{n-1})  \\
	&\leq 
	2^{N(n-1)} R_{n-1}^{N-\frac{2}{p-1}+\alpha} R_n^{-\alpha} \phi(R_n) 
\end{aligned}
\]
for $n\geq 2$, and so 
\[
	r_n = \frac{R_n}{R_{n-1}} 
	\leq 
	2^{- \frac{N}{ N-( 2/(p-1)-\alpha ) } } < \frac{1}{2}
	\quad 
	\mbox{ for } n\geq 2. 
\]
In addition, $r_1=R_1/R_0=R_1<1/2$. 
Thus, there exists $0<\overline{r}<1/2$ such that 
\[
	r_n \leq \overline{r} \quad \mbox{ for } n\geq 1. 
\]
Note that this upper bound is valid for all $p\geq p_F$. 
On the other hand, 
we give a lower bound only for $p>p_F$. 
By $R_{n-1}>R_n$ and similar computations to the above, we have 
\[
\begin{aligned}
	2^{Nn} R_n^{N-\frac{2}{p-1}} \phi(R_n) 
	&= 
	2^{N(n-1)} R_{n-1}^{N-\frac{2}{p-1}-\beta} R_{n-1}^{\beta} \phi(R_{n-1})  \\
	&\geq 
	2^{N(n-1)} R_{n-1}^{N-\frac{2}{p-1}-\beta} R_n^{\beta} \phi(R_n) 
\end{aligned}
\]
for $n\geq 2$, and so 
\[
	r_n = \frac{R_n}{R_{n-1}} 
	\geq 
	2^{- \frac{N}{ N- 2/(p-1)-\beta  } } 
	\quad 
	\mbox{ for } n\geq 2. 
\]
Then there exists $0<\underline{r}<1/2$ such that 
\[
	r_n \geq \underline{r} \quad \mbox{ for } n\geq 1 
	\mbox{ and }p>p_F. 
\]

We define a Cantor type set $\{I_n\}_{n=0}^\infty\subset \R$ 
by $I_n:=\bigcup_{l=1}^{2^n}(a_{n,l},b_{n,l})$ for $n\geq0$, 
where  
\begin{equation}\label{eq:ajlibjli}
\begin{aligned}
	&a_{0,1}:=0, &&\quad  b_{0,1}:=1, \\
	&a_{n+1,2l-1}=a_{n,l}, &&\quad  b_{n+1,2l-1}:=a_{n,l}+R_{n+1}, \\
	&a_{n+1,2l}=b_{n,l}-R_{n+1}, &&\quad  b_{n+1,2l}:=b_{n,l}. 
\end{aligned}
\end{equation}
We note that $b_{n,l}-a_{n,l}=R_n$ and that 
\[
\begin{aligned}
	&I_0=(0,1), \\
	&I_1=(0,R_1) \cup (1-R_1,1), \\
	&I_2=(0,R_2) \cup (R_1-R_2,R_1) \cup 
	(1-R_1+R_2,1-R_1) \cup (1-R_2,1), 
\end{aligned}
\]
and so on. 
By $r_n=R_n/R_{n-1}\leq \overline{r}<1/2$, 
we also note that 
\[
\begin{aligned}
	&0=a_{n,1}<b_{n,1}<a_{n,2}<b_{n,2}<\cdots < 
	a_{n,2^n}<b_{n,2^n}=1, \\
	&I_n \supset I_{n+1} \supset I_{n+2} \supset \cdots, 
\end{aligned}
\]
for $n\geq1$. 
Thus, 
each $I_n$ is a disjoint union of $2^n$ small subintervals of $(0,1)$.

\subsection{Unboundedness of a fractional maximal operator}
In this subsection, we define some fractional maximal operator
and show its unboundedness. 
To define the operator, we set up notation. 
We define a Morrey type space by 
\begin{equation}\label{eq:YfYdefsu}
\begin{aligned}
	&Y:= \{ f; \mbox{ $f\in L^1_\loc(M)$ with } \|f\|_Y<\infty\}, \\
	&\|f\|_Y:=
	\sup_{x\in M} \sup_{0<\rho<\rho_{\infty}} 
	\left( \phi(\rho)^{-1} \rho^{-(N-\frac{2}{p-1})} 
	\int_{B(x,\rho)} |f(y)| dV_{g}(y) \right). 
\end{aligned}
\end{equation}
For $\xi=(\xi_1,\ldots,\xi_N)\in\R^N$ and $\rho>0$, we write 
\[
\begin{aligned}
	&Q(\xi,\rho):= 
	\{\eta=(\eta_1,\ldots,\eta_N)\in \R^N;  
	\xi_i-\rho < \eta_i < \xi_i+\rho \quad 
	\mbox{ for }1\leq i\leq N\}, \\
	&Q_\rho:=(0,\rho)\times \cdots  \times (0,\rho). 
\end{aligned}
\]
Set 
\[
\begin{aligned}
	&P=(\rho_1,\ldots,\rho_N)\in Q_a, \qquad 
	a:=\frac{1}{\sqrt{N}}\min\{\rho_{\infty},1\},
	\\
	&D_P(\xi):=\prod_{i=1}^N (\xi_i+(1-2\overline{r})\rho_i, \xi_i+\rho_i). 
\end{aligned}
\]
Let $z_0\in M$. Fix 
an orthonormal basis $(e_{1},\dots,e_{N})$ of $T_{z_0}M$. 
By using this basis, we identify $T_{z_0}M$ and $\R^N$. 
Recall that $\rho_{\infty}\leq\inj(M)$. 
Then the restriction of 
\[
	\R^N\ni (x_1,\ldots,x_N)\mapsto \exp_{z_0}(x_1 e_1+\dots+x_N e_{N})\in M
\]
to $B(O,\rho_{\infty})\subset \R^N$ is an injection, 
where $O$ is the origin of $\R^N$. 
We denote the inverse of this map by 
\[
	\varphi: B(z_0,\rho_{\infty})\ni x \mapsto 
	(x_1,\ldots,x_N) \in \R^N. 
\]
We also write $\varphi(x)=\xi$. 
Let $|\cdot|_{\R^N}$ be the Euclidean norm on $\R^N$. 
We write $|\cdot|=|\cdot|_{\R^N}$ when no confusion can arise. 

For $f\in Y$ and $x\in \varphi^{-1}(Q_a)$, 
define a fractional maximal operator by 
\[
	\cH[f](x):= \sup_{P=(\rho_1,\ldots,\rho_N)\in Q_a}
	|P|^{-N(1-\frac{2}{Np})} \int_{D_P(\varphi(x))}  f(\varphi^{-1}(\eta)) d\eta. 
\]
Our temporal goal is to show the unboundedness of 
$\cH$ from $Y$ to $L^p(\varphi^{-1}(Q_a))$. 
More explicitly, we define functions $f_n$ by 
\begin{equation}\label{eq:gndefchiIn}
	f_n(x):= \left\{ 
	\begin{aligned}
	&\prod_{i=1}^N \chi_{I_n}(a^{-1}\varphi_{i}(x)) 
	&& \mbox{ if } x\in \varphi^{-1}(Q_a),\\
	&0 && \mbox{ if } x\notin \varphi^{-1}(Q_a), 
	\end{aligned}
	\right.
\end{equation}
and we prove that 
\begin{equation}\label{eq:Mgnunbdd}
	\frac{\|\cH[f_n]\|_{L^p(\varphi^{-1}(Q_a))}}{\|f_n\|_Y} \to \infty
	\quad \mbox{ as }n\to\infty. 
\end{equation}
Here $I_n\subset (0,1)$ is given by the previous subsection, 
$\chi_{I_n}$ is the characteristic function on $I_n$ and 
$\varphi_i(x)$ is the $i$-th component of $\varphi(x)$. 

First, we prove that there exists $C>0$ independent of $n$ 
such that 
\begin{equation}\label{eq:gnYup}
	\|f_n\|_Y \leq C(2^n R_n)^N. 
\end{equation}
Since $\supp f_n\subset \{ x\in M; \varphi(x)\in Q_a\}$, we have 
\begin{equation}\label{dx0x}
	\sup_{x\in \supp f_n} d(z_0,x)
	\leq 
	d( z_0, \varphi^{-1}(a,\ldots,a) ) 
	=|(a,\ldots,a)| 
	= \sqrt{N}a\leq \rho_{\infty}.  
\end{equation}
Thus, for each $x\in M$, we have 
\[
\begin{aligned}
	&
	\sup_{0<\rho<\rho_{\infty}} 
	\left( \phi(\rho)^{-1} \rho^{-(N-\frac{2}{p-1})} 
	\int_{B(x,\rho)} |f_n(y)| dV_{g}(y) \right) \\
	&= 
	\sup_{0<\rho<\rho_{\infty}} 
	\left( \phi(\rho)^{-1} \rho^{-(N-\frac{2}{p-1})} 
	\int_{B(x,\rho)\cap B(z_0,\rho_{\infty})} |f_n(y)| dV_{g}(y) \right) \\
	&= 
	\sup_{0<\rho<\rho_{\infty}} 
	\left( \phi(\rho)^{-1} \rho^{-(N-\frac{2}{p-1})} 
	\int_{\varphi(B(x,\rho)\cap B(z_0,\rho_{\infty}))} 
	|f_n(\varphi^{-1}(\eta))| ((\varphi^{-1})^{*}dV_{g})(\eta) \right). 
\end{aligned}
\]
Remark that 
\begin{equation}\label{eq:invvol}
	((\varphi^{-1})^{*}dV_{g})(\eta)= \sqrt{\det(g_{ij}(\eta))} d\eta. 
\end{equation}
Thus, by the volume comparison theorem \eqref{prep2}, we have 
\[
\begin{aligned}
	\int_{\varphi(B(x,\rho)\cap B(z_0,\rho_{\infty}))} 
	|f_n(\varphi^{-1}(\eta))| ((\varphi^{-1})^{*}dV_{g})(\eta) 
	&\leq 
	2^{N-1}\int_{\varphi(B(x,\rho)\cap B(z_0,\rho_{\infty}))} 
	|f_n(\varphi^{-1}(\eta))| d\eta \\
	&= 2^{N-1}\int_{\varphi(B(x,\rho)\cap B(z_0,\rho_{\infty}))} 
	\prod_{i=1}^N \chi_{I_n}(a^{-1}\eta_{i}) d\eta. 
\end{aligned}
\]

We estimate the integral in the right-hand side from above. 
It suffices to consider the case where $B(x,\rho)\cap B(z_0,\rho_{\infty})\neq\emptyset$. 
Take $\bar{x}\in B(x,\rho)\cap B(z_0,\rho_{\infty})$ and 
set $\bar{\xi}:=\varphi(\bar{x})\in B(O,\rho_{\infty})$. 
Then, for any $x'\in B(x,\rho)\cap B(z_0,\rho_\infty)$, 
by the distance comparison theorem \eqref{prep7} 
with $\rho\leq \rho_{\infty}$, we have
$|\varphi_{i}(x')-\varphi_{i}(\bar{x})|\leq 2d(x',\bar{x})\leq  4\rho$. 
This together with $\xi_{i}=\varphi_{i}(x)$ implies 
$\varphi(B(x,\rho)\cap B(z_0,\rho_{\infty}))\subset Q(\bar{\xi},4\rho)$.

Thus, we have  
\[
\begin{aligned}
	\int_{\varphi(B(x,\rho)\cap B(z_0,\rho_{\infty}))} 
	\prod_{i=1}^N \chi_{I_n}(a^{-1}\eta_{i}) d\eta
	&\leq 
	\int_{Q(\bar{\xi},4\rho)} 
	 \prod_{i=1}^N \chi_{I_n}(a^{-1} \eta_i) d\eta \\
	&= 
	a^N\int_{Q(a^{-1}\bar{\xi},a^{-1}4\rho)} 
	\prod_{i=1}^N \chi_{I_n}(\zeta_i) d\zeta. 
\end{aligned}
\]
Thus, 
\[
\begin{aligned}
	\|f_n\|_Y  &\leq 
	C \sup_{\bar{\xi}\in B(O,\rho_{\infty})} \sup_{0<\rho<\rho_{\infty}} 
	\left( \phi(\rho)^{-1} \rho^{-(N-\frac{2}{p-1})} 
	\int_{Q(a^{-1}\bar{\xi},a^{-1}4\rho)} 
	\prod_{i=1}^N \chi_{I_n}(\zeta_i) d\zeta \right)  \\
	&\leq 
	C \sup_{\xi\in \R^N} \sup_{\rho>0} 
	\left( \phi(\rho)^{-1} \rho^{-(N-\frac{2}{p-1})} 
	\int_{Q(\xi,a^{-1}4\rho)} 
	\prod_{i=1}^N \chi_{I_n}(\zeta_i) d\zeta \right) \\
	&=
	C \sup_{\xi\in \R^N} \sup_{\rho>0} 
	\left( \phi\left( 4^{-1}a \rho \right)^{-1} 
	\left( 4^{-1}a \rho \right)^{-(N-\frac{2}{p-1})} 
	\int_{Q(\xi,\rho)} 
	\prod_{i=1}^N \chi_{I_n}(\zeta_i) d\zeta \right). 
\end{aligned}
\]
Recall that there exists $0<\alpha<2/(p-1)$ 
such that 
$\rho^{-\alpha}\phi(\rho)$ is strictly decreasing. 
This together with $4^{-1}a\rho<2\rho$ yields 
\[
\begin{aligned}
	\phi\left( 4^{-1}a \rho \right)^{-1} 
	\left( 4^{-1}a \rho \right)^{-(N-\frac{2}{p-1})} 
	&=
	\left( \left( 4^{-1}a \rho\right)^{-\alpha} 
	\phi\left( 4^{-1}a \rho \right)  \right)^{-1}
	\left(4^{-1}a \rho\right)^{-\alpha} 
	\left( 4^{-1}a \rho\right)^{-(N-\frac{2}{p-1})} \\
	&\leq 
	\left( (2\rho)^{-\alpha} \phi(2\rho)  \right)^{-1}
	\left( 4^{-1}a \rho\right)^{-\alpha} 
	\left( 4^{-1}a \rho\right)^{-(N-\frac{2}{p-1})} \\
	&= 
	C \phi(2\rho)^{-1} \rho^{-(N-\frac{2}{p-1})}. 
\end{aligned}
\]
From Fubini's theorem, it follows that 
\[
\begin{aligned}
	\|f_n\|_Y  &\leq 
	C \sup_{\xi\in \R^N} \sup_{\rho>0} 
	\left( \phi(2\rho)^{-1} \rho^{-(N-\frac{2}{p-1})} 
	\int_{Q(\xi,\rho)} 
	\prod_{i=1}^N \chi_{I_n}(\zeta_i) d\zeta \right)  \\
	&=
	C \sup_{\xi\in \R^N} \sup_{\rho>0} 
	\left( \phi(2\rho)^{-1} \rho^{-(N-\frac{2}{p-1})} 
	\prod_{i=1}^N \int_{\xi_i-\rho}^{\xi_i+\rho} \chi_{I_n}(\zeta_i) d\zeta_i \right) \\
	&=
	C \sup_{\xi\in \R^N} \sup_{\rho>0} \prod_{i=1}^N 
	\left( \phi(2\rho)^{-\frac{1}{N}} \rho^{-\frac{1}{N}(N-\frac{2}{p-1})} 
	\int_{\xi_i-\rho}^{\xi_i+\rho} \chi_{I_n}(\eta) d\eta \right) \\
	&\leq 
	C \prod_{i=1}^N 
	\sup_{\xi_i\in \R} \sup_{\rho_i>0} 
	\left( \phi(2\rho_i)^{-\frac{1}{N}} \rho_i^{-\frac{1}{N}(N-\frac{2}{p-1})} 
	\int_{\xi_i-\rho_i}^{\xi_i+\rho_i} \chi_{I_n}(\eta) d\eta \right). 
\end{aligned}
\]
Then by exactly the same computations as in 
\cite[Lemma 2]{Ta16} (see page 263, line $-6$), we obtain \eqref{eq:gnYup}.

Next, we show that there exists $C>0$ independent of $n$ such that 
\begin{equation}\label{eq:MgnLplow}
	\|\cH[f_n]\|_{L^p(\varphi^{-1}(Q_a))}^p \geq
	\left\{ 
	\begin{aligned}
	&\frac{1}{C} (2^nR_n)^{Np} \sum_{j=0}^{n-1} \phi(R_j)^{p-1} 
	&&\mbox{ if } p>p_F, \\
	&(2^nR_n)^{Np} \left\{ 
	\frac{1}{C} \sum_{j=0}^{n-1} 
	\left( \phi(R_j)^{p-1} \log \frac{R_j}{R_{j+1}} \right) -C \right\}
	&&\mbox{ if } p=p_F. 
	\end{aligned}
	\right. 
\end{equation}
Let $x\in \varphi^{-1}(Q_a)$. 
We write $\varphi(x)=\xi=(\xi_1,\ldots,\xi_N)$. Then, 
\[
	\cH[f_n](x)  
	= \sup_{P=(\rho_1,\ldots,\rho_N)\in Q_a}
	|P|^{-N(1-\frac{2}{Np})} 
	\int_{D_P(\xi)}  f_n(\varphi^{-1}(\eta)) d\eta.  
\]
Fubini's theorem and the change of variables yield 
\[
\begin{aligned}
	\int_{D_P(\xi)}  f_n(\varphi^{-1}(\eta)) d\eta
	&= 
	\int_{D_P(\xi)}  \prod_{i=1}^N 
	\chi_{I_n}(a^{-1}\eta) d\eta \\
	&=
	\prod_{i=1}^N  \int_{\xi_i+(1-2\overline{r})\rho_i}^{\xi_i+\rho_i}
	\chi_{I_n}(a^{-1}\eta_i) d\eta_i 
	=
	a^{-N}\prod_{i=1}^N  
	\int_{a^{-1}(\xi_i+(1-2\overline{r})\rho_i)}^{a^{-1}(\xi_i+\rho_i)}
	\chi_{I_n}(\zeta_i) d\zeta_i. 
\end{aligned}
\]
By writing 
$\tilde P=(\tilde \rho_1,\ldots,\tilde \rho_N)
:=(a^{-1} \rho_1,\ldots,a^{-1}\rho_N)$ and 
$\tilde \xi=(\tilde \xi_1,\ldots,\tilde \xi_N)
:=(a^{-1} \xi_1,\ldots,a^{-1}\xi_N)$, we see that 
\[
\begin{aligned}
	\cH[f_n](x)  
	&= 
	a^{-N}
	\sup_{P=(\rho_1,\ldots,\rho_N)\in Q_a}
	|P|^{-N(1-\frac{2}{Np})} \prod_{i=1}^N  
	\int_{a^{-1}(\xi_i+(1-2\overline{r})\rho_i)}^{a^{-1}(\xi_i+\rho_i)}
	\chi_{I_n}(\zeta_i) d\zeta_i \\
	&=
	a^{-N-N(1-\frac{2}{Np})} 
	\sup_{\tilde P=(\tilde \rho_1,\ldots,\tilde \rho_N)\in Q_1}
	|\tilde P|^{-N(1-\frac{2}{Np})} 
	\prod_{i=1}^N  
	\int_{\tilde \xi_i+(1-2\overline{r})\tilde \rho_i}^{\tilde \xi_i+\tilde \rho_i}
	\chi_{I_n}(\zeta_i) d\zeta_i.  
\end{aligned}
\]
Remark that $\tilde \xi\in Q_1$, since $x\in \varphi^{-1}(Q_a)$. 
The function defined by the supremum in the right-hand side 
has exactly the same form as the fractional maximal function 
defined in \cite[(13)]{Ta16}. 
Thus, by exactly the same argument as in 
the first part of the proof of \cite[Lemma 3]{Ta16} 
(see page 265, lines 3--15), 
we obtain 
\begin{equation}\label{eq:Mgnlowaxi}
	\cH[f_n](\varphi^{-1}(a\tilde \xi)) = \cH[f_n](x)  
	\geq 
	a^{-N-N(1-\frac{2}{Np})} (2^{n-j-1} R_n)^N 
	\left( \sum_{i=1}^N (b_{j,l_i}-\tilde \xi_i)^2 
	\right)^{-\frac{N}{2}(1-\frac{2}{Np})}
\end{equation}
for $x\in \varphi^{-1}(Q_a)$ 
with $\tilde \xi\in J_{j,l_1,\ldots,l_N}$ 
and $\tilde \xi=a^{-1}\xi=a^{-1}\varphi(x)$, 
where $J_{j,l_1,\ldots,l_N}$ is defined by 
\[
	J_{j,l_1,\ldots,l_N}:= J_{j,l_1}\times \cdots \times J_{j,l_N}, \quad 
	J_{j,l_i}:=
	(a_{j,l_i}+R_{j+1}, b_{j,l_i}-(2\overline{r})^{-1}R_{j+1}), 
\]
for $n\geq1$, $0\leq j\leq n-1$, $1\leq i\leq N$ and $1\leq l_i\leq 2^j$. 
Here $a_{j,l_i}$ and $b_{j,l_i}$ are given by \eqref{eq:ajlibjli}. 
Note that $J_{j,l_1,\ldots,l_N}$ satisfies 
\begin{equation}\label{eq:JjliNdis}
\begin{aligned}
	&\left( \bigcup_{1\leq l_1,\ldots,l_N\leq 2^j} 
	J_{j,l_1,\ldots,l_N} \right) \cap
	\left( \bigcup_{1\leq l_1',\ldots,l_N'\leq 2^{j'}} 
	J_{j',l_1',\ldots,l_N'} \right)=\emptyset
	\quad \mbox{ for }
	j\neq j', \\
	&\bigcup_{j=0}^{n-1}  \bigcup_{1\leq l_1,\ldots,l_N\leq 2^j} 
	J_{j,l_1,\ldots,l_N} \subset Q_1. 
\end{aligned}
\end{equation}

The change of variables gives 
\[
	\|\cH[f_n]\|_{L^p(\varphi^{-1}(Q_a))}^p
	= \int_{\varphi^{-1}(Q_a)} |\cH[f_n](x)|^p dV_{g}(x) 
	= 
	\int_{Q_a} \cH[f_n](\varphi^{-1}(\xi))^p ((\varphi^{-1})^{*}dV_{g})(\xi). 
\]
By \eqref{eq:invvol}, we have 
$((\varphi^{-1})^{*}dV_{g})(\xi)= \sqrt{\det(g_{ij}(\xi))} d\xi$. 
Since 
$Q_{a}\subset B(O,\sqrt{N}a)\subset B(O,\rho_{\infty})$, 
we also have $
	\sqrt{\det(g_{ij}(\xi))}\geq 2^{1-N}$ 
by the volume comparison theorem \eqref{prep4}. 
These together with \eqref{eq:JjliNdis} imply that 
\[
\begin{aligned}
	\|\cH[f_n]\|_{L^p(\varphi^{-1}(Q_a))}^p
	&\geq 
	2^{1-N}\int_{Q_a} \cH[f_n](\varphi^{-1}(\xi))^p d\xi \\
	&=
	2^{1-N} a^N 
	\int_{Q_1} \cH[f_n](\varphi^{-1}(a\tilde \xi))^p d\tilde \xi  \\
	&\geq 
	2^{1-N} a^N \sum_{j=0}^{n-1}  \sum_{1\leq l_1,\ldots,l_N\leq 2^j} 
	\int_{J_{j,l_1,\ldots,l_N}} \cH[f_n](\varphi^{-1}(a\tilde \xi))^p d\tilde \xi. 
\end{aligned}
\]
From \eqref{eq:Mgnlowaxi}, it follows that 
\[
	\|\cH[f_n]\|_{L^p(\varphi^{-1}(Q_a))}^p
	\geq 
	\frac{1}{C} \sum_{j=0}^{n-1} 
	(2^{n-j-1} R_n)^{Np} 
	\sum_{1\leq l_1,\ldots,l_N\leq 2^j} 
	\int_{J_{j,l_1,\ldots,l_N}} 
	\left( \sum_{i=1}^N (b_{j,l_i}-\tilde \xi_i)^2 
	\right)^{-\frac{Np}{2}(1-\frac{2}{Np})}
	d\tilde \xi, 
\]
where $C>0$ is a constant depending only on $N$, $p$ and $a$. 
The same expression in the right-hand side 
appears in the middle part of the proof of \cite[Lemma 3]{Ta16} 
(see page 265, line $-3$). 
Then by exactly the same argument as in \cite[Lemmas 3, 4 and 5]{Ta16}, 
we obtain \eqref{eq:MgnLplow}.

Finally, combining \eqref{eq:gnYup} and \eqref{eq:MgnLplow} gives 
\[
\begin{aligned}
	\frac{\|\cH[f_n]\|_{L^p(\varphi^{-1}(Q_a))}^p}{\|f_n\|_Y^p}
	\geq
	\left\{ 
	\begin{aligned}
	&\frac{1}{C} \sum_{j=0}^{n-1} \phi(R_j)^{p-1} 
	&&\mbox{ if } p>p_F, \\
	&
	\frac{1}{C} \sum_{j=0}^{n-1} 
	\left( \phi(R_j)^{p-1} \log \frac{R_j}{R_{j+1}} \right) -C 
	&&\mbox{ if } p=p_F, 
	\end{aligned}
	\right. 
\end{aligned}
\]
where $\phi$ is given by \eqref{eq:phisec5def}. 
This together with 
$\int_0^1 \eta^{-1} \phi(\eta)^{p-1}d\eta =\infty$ 
and \cite[Lemma 6]{Ta16} 
shows \eqref{eq:Mgnunbdd}, 
the unboundedness of $\cH$ from $Y$ to $L^p(\varphi^{-1}(Q_a))$.

\subsection{Existence of a specific initial data}
We prove the existence of a specific initial data in $Y$
based on the unboundedness of $\cH$ from $Y$ to $L^p(\varphi^{-1}(Q_a))$ 
and the closed graph theorem. 
The idea of using the theorem is due to Brezis and Cazenave 
\cite[Theorem 11]{BC96}.

For a function $f$ on $M$, define 
\[
	\cL[f](x,t):= \int_{M} K(x,y,t) f(y) dV_{g}(y). 
\]
We write 
$Q_\rho':=\varphi^{-1}(Q_\rho)\subset B(z_0,\inj(M))$ for 
$0<\rho<\inj(M)/\sqrt{N}$.

\begin{lemma}\label{lem:g_1Iunit}
There exists a nonnegative function $\tilde f\in Y$ 
such that 
$\tilde f=0$ a.e. in $M \setminus \overline{Q_a'}$ and 
$\|\cL[\tilde f]\|_{L^p(Q_a'\times I_0)}=\infty$, 
where $I_0:=(0,1)$. 
\end{lemma}

\begin{proof}
Define a closed subspace $\tilde Y$ of $Y$ by 
\[
	\tilde Y:=\{f\in Y; f=0 \mbox{ a.e. in }M\setminus \overline{Q_a'}\}. 
\]
To obtain a contradiction, suppose that 
\[
	\| \cL[f]\|_{L^p(Q_a'\times I_0)}<\infty
	\quad 
	\mbox{ for any }f\in \tilde Y \mbox{ with }f\geq0. 
\]
Under this assumption, 
by dividing the positive part and the negative part 
for sign-changing functions, we also have 
$\| \cL[f]\|_{L^p(Q_a'\times I_0)}<\infty$ for any $f\in \tilde Y$. 
Then, the following linear operator 
\[
	\cL:\tilde Y\ni f\mapsto \cL[f] \in L^p(Q_a'\times I_0)
\]
is well-defined. 
Since $\tilde Y$ is a closed subspace of $Y$, 
$\tilde Y$ is a Banach space with $\|\cdot\|_Y$. 
Therefore, the boundedness of $\cL:\tilde Y \to L^p(Q_a'\times I_0)$ 
will follow from the closed graph theorem once we prove closedness.

Let $\{f_n\}\subset \tilde Y$. 
We assume that there exists 
$f_\infty\in \tilde Y$ satisfying  
$f_n\to f_\infty$ in $Y$. 
We also assume the existence of 
$\tilde f_\infty\in L^p(Q_a'\times I_0)$ such that 
$\cL[f_n]\to \tilde f_\infty$ in $L^p(Q_a'\times I_0)$. 
To see the closedness of $\cL:\tilde Y \to L^p(Q_a'\times I_0)$, 
we check $\cL[f_\infty]=\tilde f_\infty$. 
From Fubini's theorem and $f_n, f_\infty \in \tilde Y$, it follows that 
\[
\begin{aligned}
	\| \cL[f_n]-\cL[f_\infty] \|_{L^1(Q_a'\times I_0)} 
	&\leq \int_{I_0} \int_{Q_a'} \int_{M} 
	K(x,y,t) |f_n(y)-f_\infty(y)| dV_{g}(y) dV_{g}(x) dt  \\
	&\leq 
	\int_{I_0} \int_{M} \int_{M} 
	K(x,y,t) |f_n(y)-f_\infty(y)| dV_{g}(y) dV_{g}(x) dt  \\
	&=
	\int_{M} |f_n(y)-f_\infty(y)| dV_{g}(y)
	=
	\int_{\overline{Q_a'}} |f_n(y)-f_\infty(y)| dV_{g}(y). 
\end{aligned}
\]
Since $\overline{Q_a'}$ is compact, 
there exist $k$ and $x_1',\ldots,x_k'\in M$ 
such that $\overline{Q_a'}\subset \bigcup_{i=1}^k B(x_i',1/4)$, 
where $k$ is independent of $n$. 
Thus, 
\[
\begin{aligned}
	\int_{\overline{Q_a'}} |f_n(y)-f_\infty(y)| dV_{g}(y)
	&\leq 
	\sum_{i=1}^k 
	\int_{B(x_i', 1/4)} |f_n(y)-f_\infty(y)| dV_{g}(y) \\
	&\leq 
	\sum_{i=1}^k 
	\phi(1/4) (1/4)^{N-\frac{2}{p-1}}
	\| f_n-f_\infty\|_Y   
	= 
	C\| f_n-f_\infty\|_Y, 
\end{aligned}
\]
and so 
$\| \cL[f_n]-\cL[f_\infty] \|_{L^1(Q_a'\times I_0)} 
\leq C\| f_n-f_\infty\|_Y \to 0$ 
as $n\to\infty$. 
On the other hand, 
by $\cL[f_n]\to \tilde f_\infty$ in $L^p(Q_a'\times I_0)$, 
we have 
$\| \cL[f_n]-\tilde f_\infty \|_{L^1(Q_a'\times I_0)} \to 0$
as $n\to\infty$. 
Hence $\cL[f_\infty]=\tilde f_\infty$, 
and so $\cL:\tilde Y \to L^p(Q_a'\times I_0)$ is closed. 
By the closed graph theorem, $\cL$ is also bounded. 
In particular, there exists a constant $C$ such that 
\begin{equation}\label{eq:Lgbddg}
	\|\cL[f]\|_{L^p(Q_a'\times I_0)}
	\leq C \|f\|_Y
	\quad \mbox{ for any }f\in \tilde Y. 
\end{equation}

Fix a nonnegative function $f\in \tilde Y$. 
We claim that 
\begin{equation}\label{eq:MgLgesker}
	\|\cL[f](x,\,\cdot\,)\|_{L^p(I_0)} \geq C \cH[f](x)
	\quad \mbox{ for any }x\in Q_a'. 
\end{equation}
Fix a point $P=(\rho_1,\ldots,\rho_N)\in Q_a$. 
By $a\leq 1/\sqrt{N}$, we have 
$|P|< \sqrt{N}a\leq 1$, and so 
$(|P|^2/2,|P|^2)\subset I_0$. 
Let $x\in Q_a'$ and $t\in (|P|^2/2,|P|^2)$. 
Then, 
\[
	\cL[f](x,t) =
	\int_{M} K(x,y,t) f(y) dV_{g}(y)
	= \int_{Q'_a} K(x,y,t) f(y) dV_{g}(y). 
\]
The lower bound of the heat kernel \eqref{prep10} gives 
\[
	K(x,y,t)\geq 
	C t^{-\frac{N}{2}}e^{-\frac{d(x,y)^2}{2t}}
	\quad\mbox{ for }x,y\in Q'_{a} \mbox{ and } t\in(0,1), 
\]
where $C>0$ is a constant depending only on 
$N$ and $\kappa$. 
This together with $Q_{a}'=\varphi^{-1}(Q_{a})$, 
$\xi=\varphi(x)$, \eqref{eq:invvol} and 
the volume comparison theorem \eqref{prep4} yields 
\[
\begin{aligned}
	\cL[f](x,t) &\geq 
	C  \int_{ Q_{a}' } 
	t^{-\frac{N}{2}} e^{-\frac{d(x,y)^2}{2t} } f(y) dV_{g}(y)  \\
	&= 
	C  \int_{ Q_{a} } 
	t^{-\frac{N}{2}} e^{-\frac{d(\varphi^{-1}(\xi),\varphi^{-1}(\eta))^2}{2t} } 
	f(\varphi^{-1}(\eta)) ((\varphi^{-1})^{*}dV_{g})(\eta) \\
	&\geq 
	C  \int_{ Q_{a} } 
	t^{-\frac{N}{2}} e^{-\frac{d(\varphi^{-1}(\xi),\varphi^{-1}(\eta))^2}{2t} } 
	f(\varphi^{-1}(\eta)) d\eta. 
\end{aligned}
\]
By \eqref{prep6}, we see that 
$d(\varphi^{-1}(\xi),\varphi^{-1}(\eta)) \leq 2 |\xi-\eta|$
for any $\xi,\eta\in Q_{a}$. 
Recall that $P=(\rho_1,\ldots,\rho_N)\in Q_a$ 
and $\xi=\varphi(x)\in Q_a$ for $x\in Q_a'$. 
Then $\xi_i + \rho_i <2a$ for $1\leq i\leq N$, and so 
\[
	Q_{2a} \supset D_P(\xi) \quad 
	\mbox{ for any }P\in Q_a 
	\mbox{ and }\xi\in Q_a. 
\]
Therefore, since $f=0$ a.e. in $M\setminus \overline{Q_a'}$, 
we see that 
\[
\begin{aligned}
	\cL[f](x,t) &\geq 
	C |P|^{-N}
	\int_{ Q_{a} }  e^{-\frac{2|\xi-\eta|^2}{t} }  
	f(\varphi^{-1}(\eta)) d\eta  \\
	&=
	C |P|^{-N}
	\int_{ Q_{2a} }  e^{-\frac{2|\xi-\eta|^2}{t} }  
	f(\varphi^{-1}(\eta)) d\eta  
	\geq 
	C |P|^{-N} 
	\int_{ D_P(\xi) }  e^{-\frac{2|\xi-\eta|^2}{t} }  
	f(\varphi^{-1}(\eta)) d\eta
\end{aligned}
\]
for $x\in Q_a'$ and $t\in (|P|^2/2,|P|^2)\subset I_0$. 
By 
\[
\begin{aligned}
	\frac{2|\xi-\eta|^2}{t}
	=
	\frac{2}{t}\left( 
	(\xi_1-\eta_1)^2 + \dots + (\xi_N-\eta_N)^2
	\right) \leq 
	\frac{4}{|P|^2}\left( \rho_1^2 + \dots + \rho_N^2 \right) 
	= 4 
\end{aligned}
\]
for $\eta\in D_P(\xi)$ and $t\in (|P|^2/2,|P|^2)$, we have 
\[
	\cL[f](x,t) 
	\geq 
	C |P|^{-N}  
	\int_{ D_P(\xi) }  f(\varphi^{-1}(\eta)) d\eta 
\]
for $x\in Q_a'$ and $t\in (|P|^2/2,|P|^2)$. 
Integrating over $t\in I_0$ and using 
$(|P|^2/2,|P|^2)\subset I_0$ yield 
\[
\begin{aligned}
	\|\cL[f](x,\,\cdot\,)\|_{L^p(I_0)} &\geq
	\left( \int_{|P|^2/2}^{|P|^2} \left( 
	C |P|^{-N} \int_{ D_P(\xi) }  f(\varphi^{-1}(\eta)) d\eta
	\right)^p dt \right)^\frac{1}{p}  \\
	&= 
	C |P|^{-N} \int_{ D_P(\xi) }  f(\varphi^{-1}(\eta)) d\eta
	\left( \int_{|P|^2/2}^{|P|^2} dt\right)^\frac{1}{p} \\
	&=
	2^{-\frac{1}{p}} C |P|^{-N(1-\frac{2}{Np})}  
	\int_{ D_P(\varphi(x)) }  f(\varphi^{-1}(\eta)) d\eta
\end{aligned}
\]
for $x\in Q_a'$. 
Then by $P=(\rho_1,\ldots,\rho_N)\in Q_a$, 
$x\in Q_a'$ and the definition of $\cH$, 
we obtain \eqref{eq:MgLgesker}.

Taking the $L^p(Q_a')$-norm in \eqref{eq:MgLgesker} gives 
\[
	\|\cH[f]\|_{L^p(Q_a')}
	\leq 
	C \left( \int_{Q_a'}  \|\cL[f](x,\,\cdot\,)\|_{L^p(I_0)}^p 
	dV_{g}(x) \right)^\frac{1}{p} 
	= C\|\cL[f]\|_{L^p(Q_a'\times I_0)}. 
\]
This together with \eqref{eq:Lgbddg} shows that 
$\|\cH[f]\|_{L^p(Q_a')} \leq C \|f\|_Y$. 
By $Q_a'=\varphi^{-1}(Q_a)$, we obtain 
\[
	\|\cH[f]\|_{L^p(\varphi^{-1}(Q_a))}
	\leq C \|f\|_Y
	\quad \mbox{ for any }
	f\in \tilde Y \mbox{ with }f\geq0. 
\]
Define $f_n$ by \eqref{eq:gndefchiIn}. 
By definition, we have $f_n\geq0$ and 
$\supp f_n \subset Q_a'$. 
The upper estimate \eqref{eq:gnYup} gives $f_n\in Y$. 
Thus, $f_n\in \tilde Y$  with $f_n\geq0$.
However, $f_n$ satisfies \eqref{eq:Mgnunbdd}, 
a contradiction. 
\end{proof}

By Lemma \ref{lem:g_1Iunit} and a scaling argument, 
we prove the following:

\begin{lemma}\label{lem:gTI0T}
Let $0<T<a^2$. 
Then there exists a nonnegative function $f_T\in Y$ 
such that $f_T=0$ a.e. in $M\setminus \overline{Q_{\sqrt{T}}'}$
and $\|\cL[f_T]\|_{L^p(Q_{\sqrt{T}}'\times (0,bT))}=\infty$, where 
$b:=9/a^2$.
\end{lemma}

\begin{proof}
We write $\lambda:=\sqrt{T}/a$ in this proof. 
By using $\tilde f$ in Lemma \ref{lem:g_1Iunit}, we set 
\[
	f_T(x):= 
	\left\{ 
	\begin{aligned}
	&\tilde f\left( \varphi^{-1} 
	\left( \lambda^{-1} \varphi(x) \right) \right) 
	&&\mbox{ for }x\in Q_{\sqrt{T}}', \\
	&0 
	&&\mbox{ for }x\in M\setminus Q_{\sqrt{T}}'. 
	\end{aligned}
	\right.
\]
Remark that $\varphi^{-1}( \lambda^{-1} \varphi(x) ) \in Q_a'$ 
for $x\in Q_{\sqrt{T}}'$. 
Then the properties of $\tilde f$ imply that 
$f_T\geq0$, $f_T\in Y$ and 
$f_T=0$ a.e. in $M\setminus \overline{Q_{\sqrt{T}}'}$. 
For simplicity of notation, 
we write $A_T:=\|\cL[f_T]\|_{L^p(Q_{\sqrt{T}}'\times (0,bT))}^p$ 
in this proof. 
In what follows, we compute 
\[
	A_T
	=
	\int_0^{bT} \int_{Q_{\sqrt{T}}'} 
	\left| \int_{M} K(x,y,t) f_T(y) dV_{g}(y) \right|^p
	dV_{g}(x) dt. 
\]

By the change of variables and \eqref{eq:invvol}, we have 
\[
\begin{aligned}
	\int_{M} K(x,y,t) f_T(y) dV_{g}(y) 
	&= 
	\int_{Q_{\sqrt{T}}'} K(x,y,t) 
	\tilde f\left( \varphi^{-1} \left( \lambda^{-1}\varphi(y)\right) \right) 
	dV_{g}(y)  \\
	&=
	\int_{Q_{\sqrt{T}}} K(x,\varphi^{-1}(\eta),t) 
	\tilde f\left( \varphi^{-1} \left( \lambda^{-1} \eta \right) \right) 
	\sqrt{\det(g_{ij}(\eta))} d\eta \\
	&=
	\int_{Q_a} 
	K\left( x,\varphi^{-1}\left( \lambda \tilde\eta \right),t \right) 
	\tilde f(\varphi^{-1}(\tilde\eta))
	\lambda^N  
	\sqrt{\det\left( g_{ij}\left( \lambda \tilde \eta \right) \right)}
	d\tilde \eta
\end{aligned}
\]
for $(x,t)\in Q_{\sqrt{T}}'\times (0,bT)$. 
Since $|\lambda\tilde\eta|\leq|\tilde\eta|< \sqrt{N}a\leq\rho_{\infty}$
for $\tilde\eta\in Q_a$, 
the volume comparison theorem \eqref{prep4} implies 
the existence of a constant $C>0$ independent of $T$ such that 
\[
	\int_{M} K(x,y,t) f_T(y) dV_{g}(y) 
	\geq  
	C \lambda^N 
	\int_{Q_{a}} K\left( x,\varphi^{-1}\left( \lambda\tilde\eta \right),t \right) 
	\tilde f(\varphi^{-1}(\tilde\eta)) d\tilde\eta
\]
for $(x,t)\in Q_{\sqrt{T}}'\times (0,bT)$. Thus, \eqref{eq:invvol} gives 
\[
\begin{aligned}
	A_T&\geq 
	C \lambda^{Np} 
	\int_0^{bT} \int_{Q_{\sqrt{T}}'} 
	\left| 
	\int_{Q_{a}} K\left( x,\varphi^{-1}\left( \lambda \tilde\eta \right),t \right) 
	\tilde f(\varphi^{-1}(\tilde\eta)) d\tilde\eta  \right|^p
	dV_{g}(x) dt   \\
	&=
	C\lambda^{Np} 
	\int_0^{bT} \int_{Q_{\sqrt{T}}} 
	\left| 
	\int_{Q_{a}} K\left( \varphi^{-1}(\xi), 
	\varphi^{-1}\left( \lambda \tilde\eta \right), t \right) 
	\tilde f(\varphi^{-1}(\tilde\eta)) d\tilde\eta  \right|^p
	\sqrt{\det(g_{ij}(\xi))} d\xi dt. 
\end{aligned}
\]
By the change of variables and the volume comparison theorem 
\eqref{prep4} again, 
we see that 
\[
\begin{aligned}
	A_T &\geq 
	C\lambda^{Np}
	\int_0^{bT} \int_{Q_a} 
	\left| 
	\int_{Q_{a}} K\left(
	\varphi^{-1}\left( \lambda \tilde\xi \right), 
	\varphi^{-1}\left( \lambda \tilde\eta \right), 
	t \right) 
	\tilde f(\varphi^{-1}(\tilde\eta)) d\tilde\eta  \right|^p 
	\lambda ^N
	\sqrt{\det(g_{ij}(\lambda \tilde \xi))} d\tilde \xi  dt \\
	&\geq
	C \lambda^{N(p+1)} T 
	\int_0^{b} \int_{Q_{a}} 
	\left| 
	\int_{Q_{a}} K\left(
	\varphi^{-1}\left( \lambda \xi \right), 
	\varphi^{-1}\left( \lambda \eta \right), 
	T t \right) 
	\tilde f(\varphi^{-1}(\eta))   d\eta \right|^p 
	d \xi  dt. 
\end{aligned}
\]

We estimate the heat kernel in the integrand. 
By $\lambda=\sqrt{T}/a<1$, we can apply 
\eqref{prep11} with $\alpha=\lambda$ and $\beta=T$, 
where we remark that $\beta/(9\alpha^2)=1/b$. 
Then there exists $C>0$ depending 
only on $N$, $\kappa$, $\inj(M)$, $a$ and $T$ such that 
\[
K\left(
	\varphi^{-1}\left( \lambda \xi \right), 
	\varphi^{-1}\left( \lambda \eta \right), 
	T t \right) \geq C K\left(\varphi^{-1}(\xi),\varphi^{-1}(\eta),b^{-1} T\right)
\]
for $0<\sqrt{t}<\min\{\inj(M),\pi/(4\sqrt{\kappa}),\sqrt{b T}\}$. 
This together with the volume comparison theorem \eqref{prep2} yields 
\[
\begin{aligned}
	A_T&\geq 
	C \int_0^b \int_{Q_a} 
	\left| 
	\int_{Q_a} 
	K(\varphi^{-1}(\xi),\varphi^{-1}(\eta),b^{-1} t)
	\tilde f(\varphi^{-1}(\eta))  d\eta  \right|^p  
	d \xi dt \\
	&=
	C\int_0^{1} \int_{Q_a} 
	\left| 
	\int_{Q_a} 
	K(\varphi^{-1}(\xi),\varphi^{-1}(\eta),t')
	\tilde f(\varphi^{-1}(\eta)) d\eta \right|^p  
	d \xi dt'\\
	&\geq
	C\int_0^{1} \int_{Q_a} 
	\left| 
	\int_{Q_a} 
	K(\varphi^{-1}(\xi),\varphi^{-1}(\eta),t')  
	\tilde f(\varphi^{-1}(\eta)) \sqrt{\det(g_{ij}(\eta))}d\eta \right|^p  
	\sqrt{\det(g_{ij}(\xi))}d \xi dt'\\
	&=
	C \int_0^{1} \int_{Q_a'} 
	\left| \int_{Q_a'} K(x,y,t) \tilde f(y) dV_{g}(y)  \right|^p  
	dV_{g}(x)  dt. 
\end{aligned}
\]
Recall that 
$\tilde f=0$ a.e. in $M \setminus \overline{Q_a'}$ 
and $\|\cL[\tilde f]\|_{L^p(Q_a'\times I_0)}=\infty$ 
by Lemma \ref{lem:g_1Iunit}. 
Then, 
\[
\begin{aligned}
	A_T
	&\geq 
	C\int_0^{1} \int_{Q_a'} 
	\left| \int_{M} K(x,y,t) \tilde f(y) dV_{g}(y) \right|^p  
	dV_{g}(x) dt \\
	&=
	C\|\cL[\tilde f]\|_{L^p(Q_a'\times I_0)}^p  =\infty, 
\end{aligned}
\]
and so $\|\cL[f_T]\|_{L^p(Q_{\sqrt{T}}'\times (0,bT))}=A_T^{1/p}=\infty$. 
\end{proof}

Lemma \ref{lem:gTI0T} shows the existence of a specific initial data $f_T$
depending on the choice of $T$. 
By using this lemma, we construct 
the desired initial data $f_0$ independent of $T$.

\begin{lemma}\label{lem:speinig0Y}
There exists a nonnegative function $f_0\in Y$ such that 
$\|f_0\|_Y\leq 1$, 
$f_0=0$ a.e. in $M \setminus \overline{Q_a'}$
and 
$\|\cL[f_0]\|_{L^p(Q_{\sqrt{T}}'\times (0,bT))}=\infty$ 
for all $0<T<a^2$. 
\end{lemma}

\begin{proof}
The proof is the same as in \cite[Theorem 2.4]{KT17} 
and \cite[Proposition 1]{Ta16}. 
We give an outline. 
Let $\{q_j\}_{j=1}^\infty$ be the all rational numbers in $(0,a^2)$. 
By Lemma \ref{lem:gTI0T}, for each $j$, there exists 
$0\leq f_{\sqrt{q_j}}\in Y$ 
such that 
$f_{\sqrt{q_j}}=0$ a.e. in $M \setminus \overline{Q_{\sqrt{q_j}}'}$ and 
$\|\cL[f_{\sqrt{q_j}}]\|_{L^p(Q_{\sqrt{q_j}}'\times (0,b q_j))}=\infty$. 
Set 
\[
	f_0(x):= \sum_{j=1}^\infty 2^{-j} 
	\frac{f_{\sqrt{q_j}}(x)}{\|f_{\sqrt{q_j}}\|_Y}. 
\]
We can see that $f_0\geq0$, $f_0\in Y$, $\|f_0\|_Y \leq 1$ and 
$f_0=0$ a.e. in $M \setminus \overline{Q_a'}$. 
For $0<T<a^2$, there exists $j'$ such that 
$Q_{\sqrt{q_{j'}}}'\times (0,bq_{j'})\subset Q_{\sqrt{T}}'\times (0,bT)$. 
Then $f_0(x)\geq  2^{-j'} \|f_{\sqrt{q_{j'}}}\|_Y^{-1} f_{\sqrt{q_{j'}}}(x)$ and 
\[
	\|\cL[f_0]\|_{L^p(Q_{\sqrt{T}}'\times (0,bT))} 
	\geq 
	2^{-j'} \|f_{\sqrt{q_{j'}}}\|_Y^{-1} 
	\|\cL[f_{\sqrt{q_{j'}}}]\|_{L^p(Q_{\sqrt{q_{j'}}}'\times (0,bq_{j'}))} =\infty. 
\]
Hence $f_0$ is the desired function. 
\end{proof}

\subsection{Completion of proof}

We are now in a position to prove Theorem \ref{th:nex}.

\begin{proof}[Proof of Theorem \ref{th:nex}]
Fix $c>0$. Set $u_0:= c f_0$, where 
$f_0$ is given by Lemma \ref{lem:speinig0Y}. 
We define a nonnegative Radon measure $\mu_0$ 
by $\mu_0(A):=\int_A u_0(x) dV_g(x)$. 
Then by $u_0\geq0$, $\|u_0\|_Y\leq c$ and the definition 
of $Y$ in \eqref{eq:YfYdefsu}, we have 
\[
\begin{aligned}
	\sup_{z\in M} \mu_0( B(z,\rho) ) 
	&=
	\sup_{z\in M} \int_{B(z,\rho)} u_0(y) dV_{g}(y)  \\
	&\leq c \rho^{N-\frac{2}{p-1}} 
	( \log(e+\rho^{-1}) )^{-\frac{1}{p-1}}
	\quad 
	\mbox{ for any }0<\rho< \rho_{\infty}. 
\end{aligned}
\]
Hence $\mu_0$ satisfies the growth condition in Theorem \ref{th:nex}. 
It suffices to prove that the problem \eqref{eq:main} with $u(\,\cdot\,,0)=\mu_0$ 
does not admit any solutions in $M\times [0,T)$ for all $T>0$. 

To obtain a contradiction, suppose that 
there exists $T>0$ such that 
\eqref{eq:main} with $u(\,\cdot\,,0)=\mu_0$ 
admits a solution $u$ in $M\times [0,T)$. 
Without loss of generality, 
we assume $0<T<a^2$. 
By Lemma \ref{lem:speinig0Y}, we have 
$\|\cL[u_0]\|_{L^p(Q_{\sqrt{T}}'\times (0,bT))}=\infty$. 
Thus, we will reach a contradiction 
once we obtain 
\begin{equation}\label{eq:contLpu0}
	\|\cL[u_0]\|_{L^p(Q_{\sqrt{T}}'\times (0,bT))}<\infty. 
\end{equation}

Let us prove \eqref{eq:contLpu0}. 
The same computations as \eqref{dx0x} give 
$d(z_0,x)\leq \sqrt{N}a$ for $x\in Q_a'$. 
Then by $Q_{\sqrt{T}}'\subset Q_a'$, we observe that 
$\supp u_0\subset Q_a'\subset B(z_0,\sqrt{N}a)$. 
Based on this observation, we choose 
$\psi\in C^\infty_0(M)$ such that 
$0\leq \psi\leq1$, $\psi=1$ in $B(z_0,\sqrt{N}a)$, 
$\psi=0$ in $M\setminus \overline{B(z_0,2\sqrt{N}a)}$. 
Let $0<\eps<bT$. 
Integrating the equation \eqref{eq:fujita} over $M\times(\eps,bT)$ 
and integrating by parts, we see that 
\[
\begin{aligned}
	\int_\eps^{bT} \int_{M} u^p \psi dV_{g} dt
	&= 
	\int_\eps^{bT} \int_{M} (u_t-\Delta u) \psi dV_{g} dt  \\
	&= 
	\int_{M} (u(x,bT)-u(x,\eps) ) \psi dV_{g} 
	- \int_\eps^{bT} \int_{M} u \Delta \psi dV_{g} dt  \\
	&\leq 
	\int_{M} u(x,bT) \psi dV_{g} 
	+ \int_\eps^{bT} \int_{M} u |\Delta \psi| dV_{g} dt. 
\end{aligned}
\]
By the same argument as in the proof of Theorem \ref{th:necsha} 
(proof of \eqref{eq:ini}), 
we see that 
\[
	\int_{M} u(x,t) \tilde \psi(x) dV_{g}(x)
	\to 
	\int_{M} u_0(x) \tilde \psi(x) dV_{g}(x) 
\]
for any $\tilde \psi\in C_0(M)$ as $t\to0$. 
Then there exists $0<t_0<bT$ such that 
\[
	\int_{M} |u(x,t)-u_0(x)| |\Delta \psi(x)| dV_{g}(x)
	\leq 1
	\quad \mbox{ for }0<t<t_0. 
\]
From this, it follows that 
\[
	\int_{M} u(x,t) |\Delta \psi(x)| dV_{g}(x)
	\leq 1 + \int_{M} u_0(x) |\Delta \psi(x)| dV_{g}(x) \leq C
\]
for $0<t<t_0$, where $C>0$ is a constant independent of $\eps$. 
Thus, 
\[
	\int_\eps^{bT} \int_{M} u |\Delta \psi| dV_{g} dt
	\leq 
	\int_0^{t_0} \int_{M} u |\Delta \psi| dV_{g} dt
	+ \int_{t_0}^{bT} \int_{M} u |\Delta \psi| dV_{g} dt 
	\leq C, 
\]
and so
\[
	\int_\eps^{bT} \int_{M} u^p \psi dV_{g} dt
	\leq 
	\int_{M} u(x,bT) \psi dV_{g} + C. 
\]
Since the right-hand side is independent of $\eps$ and $u^p \psi\geq 0$, 
we obtain 
\[
	\int_0^{bT} \int_{M} u^p \psi dV_{g} dt \leq 
	\int_{M} u(x,bT) \psi dV_{g} + C < \infty. 
\]
In particular, $u\in L^p(B(z_0,\sqrt{N}a)\times (0,bT))$.

Let $x\in M$, $0<t<bT$ and $0<\tau<bT-t$. 
Again by the same argument as in the proof of Theorem \ref{th:necsha} 
(proof of \eqref{eq:uint}), 
$u$ satisfies 
\[
	u(x,t+\tau) = 
	\int_{M} K(x,y,t) u(y,\tau) dV_{g}(y) 
	+ \int_\tau^{t+\tau} \int_{M} K(x,y,t+\tau-s) u(y,s)^p dV_{g}(y) ds. 
\]
Then by $K, u\geq 0$ and $\psi\leq 1$, we have 
\[
	u(x,t+\tau) \geq \int_{M} K(x,y,t)\psi(y) u(y,\tau) dV_{g}(y). 
\]
From $K(x,\,\cdot\,,t)\psi(\,\cdot\,) \in C_0(M)$ 
and letting $\tau\to0$, 
it follows that 
\[
	u(x,t) 
	\geq  \int_{M} K(x,y,t) \psi(y) d\mu_0(y)
	= \int_{M} K(x,y,t) \psi(y) u_0(y) dV_{g}(y). 
\]
Since $\psi(y)=1$ for $y\in \supp u_0\subset B(z_0,\sqrt{N}a)$, 
we obtain 
\[
	u(x,t) \geq \int_{M} K(x,y,t) u_0(y) dV_{g}(y)
	= \cL[u_0](x,t). 
\]
Integrating this over $Q_{\sqrt{T}}'\times (0,bT)$ yields 
\[
	\|\cL[u_0]\|_{L^p(Q_{\sqrt{T}}'\times (0,bT))}
	\leq \|u\|_{L^p(Q_{\sqrt{T}}'\times (0,bT))}. 
\]
Recall that $Q_{\sqrt{T}}' \subset B(z_0,\sqrt{N}a)$
and $u\in L^p(B(z_0,\sqrt{N}a)\times (0,bT))$. 
Therefore, 
\[
	\|\cL[u_0]\|_{L^p(Q_{\sqrt{T}}'\times (0,bT))}
	\leq 
	\|u\|_{L^p(B(z_0,\sqrt{N}a)\times (0,bT))} <\infty. 
\]
We obtain \eqref{eq:contLpu0}. The proof is complete. 
\end{proof}

\section{Proofs of corollaries}\label{sec:poc}
We give proofs of Corollaries \ref{cor:ex}, \ref{cor:singnex}, 
\ref{cor:nex} and \ref{cor:singex}. 
The existence part is a corollary of the proof of Theorem \ref{th:necshacri} and 
the nonexistence part is a corollary of Theorem \ref{th:nec}. 

\begin{proof}[Proof of Corollary \ref{cor:ex}]
We first assume (i). 
Let $1<\alpha<q$. 
For a nonnegative function $u_0\in L^q_{\uloc, \tilde \rho}(M)$ 
with $\tilde \rho\in(0,\rho_T]$, set 
\[
	\overline{u}(x,t):= 2U(x,t), \qquad 
	U(x,t):= \left( \int_M K(x,y,t) u_0(y)^\alpha dV_g(y) \right)^\frac{1}{\alpha}, 
\]
as in \eqref{eq:super}. 
Then by $\int_M K(x,y,t) u_0(y) dV_g(y) \leq U(x,t)$, 
\eqref{eq:ouesti} and \eqref{eq:ouestij}, we have 
\[
\begin{aligned}
	\Psi[\overline{u}] &\leq 
	U + 2^p U(x,t)
	\left\| U(\,\cdot\,,t)^{\alpha-1} \right\|_\infty 
	\int_0^t \left\| U(\,\cdot\,,s)^{p-\alpha} \right\|_\infty ds. 
\end{aligned}
\]
Lemma \ref{lem:linheat} and the H\"older inequality show that, for $(x,t)\in M\times (0,\tilde \rho^2)$,  
\[
\begin{aligned}
	U(x,t) \leq 
	\left( C t^{-\frac{N}{2}} \sup_{z\in M} \int_{B(z,t^\frac{1}{2})} u_0^\alpha dV_g
	\right)^\frac{1}{\alpha} \leq 
	C t^{-\frac{N}{2\alpha}} \sup_{z\in M} 
	\left( \int_{B(z,t^\frac{1}{2})} u_0^q dV_g\right)^\frac{1}{q} 
	\Vol(B(z,t^\frac{1}{2}))^{\frac{1}{\alpha}-\frac{1}{q}}, 
\end{aligned}
\]
where $C>0$ depends only on $N$ and $\alpha$. 
Note that the volume comparison theorem \eqref{prep3} yields 
$\Vol(B(z,t^{1/2}))
\leq 2^{N-1}\Area(\bS^{N-1})N^{-1} t^{N/2}$, 
where we used $t^{1/2}<\tilde \rho\leq \rho_\infty$. 
Then by $\alpha<q$, we have 
$U \leq C t^{-N/(2q)}
\|u_0\|_{L_{\uloc,\tilde \rho}^q(M)} 
\leq C  c t^{-N/(2q)}$, 
and so 
\[
	\Psi[\overline{u}] \leq 
	U + C c t^{1-\frac{N(p-1)}{2q}} U 
	\leq 
	( 1 + C c \tilde \rho^{2-\frac{N(p-1)}{q}}) U
\]
for $0<t<\tilde \rho^2$. 
Hence $\overline{u}$ is a supersolution of \eqref{eq:integ} 
in $M\times[0,\tilde \rho^2)$
if $C c \tilde \rho^{2-N(p-1)/q}\leq 1$. 
The rest is the same as that of Theorem \ref{th:necsha}.

We next assume (ii). 
For a nonnegative function $u_0\in L^q_{\uloc,\tilde \rho}(M)$ 
with $\tilde \rho\in(0,\rho_T]$, 
again by the volume comparison theorem \eqref{prep3}, 
we have 
\begin{equation}\label{eq:hvol}
	\sup_{z\in M}\int_{B(z,\rho)} u_0 dV_g 
	\leq 
	\|u_0\|_{L_{\uloc,\rho}^q(M)} 
	\Vol(B(z,\rho))^{1-\frac{1}{q}}
	\leq 
	C \|u_0\|_{L_{\uloc,\tilde \rho}^q(M)} 
	\rho^{N(1-\frac{1}{q})}
\end{equation}
for any $0<\rho<\tilde \rho$, where 
$C>0$ is a constant depending only on $N$ and $q$. 
From $X^{N(1-1/q)}\leq C(\log(e+X^{-1}))^{-N/2-1}$ for $0<X<1$, 
it follows that 
\[
	\sup_{z\in M}\int_{B(z,\rho)} u_0 dV_g 
	\leq 
	C c \tilde \rho^{N(1-\frac{1}{q})} 
	(\log(e+\tilde \rho \rho^{-1}))^{-\frac{N}{2}-1}. 
\]
Define $\mu$ by $\mu(A):=\int_A u_0 dV_g$. 
Then $\mu$ satisfies \eqref{eq:suf} with $\eps=1$ 
provided that $C c \tilde \rho^{N(1-1/q)}$ is small enough.  
Therefore we can apply Theorem \ref{th:suf}.

Finally, we assume (iii). 
Define $\mu$ as above. 
By \eqref{eq:hvol}, 
$\sup_{z\in M}\mu(B(z,\tilde \rho)) \leq 
C c \tilde \rho^{N(1-1/q)}$, and so $\mu$ satisfies \eqref{eq:sha1} 
if $C c \tilde \rho^{N(1-1/q)}$ is small. 
Hence we can apply Theorem \ref{th:necsha}. 
\end{proof}

\begin{proof}[Proof of Corollary \ref{cor:singnex}]
Let $z_0\in M$ and 
let $u_0\in L^\infty_\loc(M\setminus\{z_0\})$ satisfy \eqref{eq:u0la}. 
Set $\mu(A):=\int_A u_0 dV_g$. 
By using Riemannian normal coordinates 
$(x_1,\ldots,x_N)$ centered at $z_0$, the volume comparison theorem \eqref{prep4} and 
$1\geq (e|x|+1)^{-1}$, we have 
\[
\begin{aligned}
	\mu(B(z_0,\rho)) &\geq 
	\left\{ 
	\begin{aligned}
	&c \int_{B(0,\rho)} |x|^{-\frac{2}{p-1}} \sqrt{\det(g_{ij}(x))} dx 
	&&(p>p_F)\\
	&c \int_{B(0,\rho)} |x|^{-N} (e|x|+1)^{-1}
	(\log(e+|x|^{-1}))^{-\frac{N}{2}-1} \sqrt{\det(g_{ij}(x))} dx 
	&&(p=p_F)
	\end{aligned}
	\right.  \\
	&\geq 
	\left\{ 
	\begin{aligned}
	&\tilde c c \rho^{N-\frac{2}{p-1}}  
	&&(p>p_F)\\
	&\tilde c c (\log(e+\rho^{-1}))^{-\frac{N}{2}}
	&&(p=p_F)
	\end{aligned}
	\right.
\end{aligned}
\]
for any $0<\rho<\rho_\infty$, 
where $\tilde c>0$ is a constant depending only on $N$ and $p$. 
We denote by $\tilde C$ the constant guaranteed by Theorem \ref{th:nec}, 
where $\tilde C$ depends only on $N$ and $p$. 

Let $c>\tilde C/\tilde c$. 
To obtain a contradiction, we suppose that 
there exists $T>0$ such that \eqref{eq:fun} has a solution on $M\times[0,T)$. 
Then by Theorem \ref{th:nec}, we have 
\[
\begin{aligned}
	\mu(B(z_0,\rho)) &\leq 
	\left\{ 
	\begin{aligned}
	&\tilde C \rho^{N-\frac{2}{p-1}}  &&\mbox{ if }p>p_F, \\
	&\tilde C (\log(e+\rho_T \rho^{-1}))^{-\frac{N}{2}} &&\mbox{ if }p=p_F, 
	\end{aligned}
	\right. 
\end{aligned}
\]
for any $0<\rho<\rho_T$. 
If $p>p_F$, then $c \leq \tilde C/\tilde c$. 
If $p=p_F$, since $\rho_T\leq \rho_\infty$, we have 
\[
	c 
	\leq (\tilde C/\tilde c)  \left( 
	\frac{\log(e+\rho^{-1})}{\log(e+\rho_T \rho^{-1})} \right)^{\frac{N}{2}}
	\quad \mbox{ for any }0<\rho<\rho_T. 
\]
Letting $\rho\to0$ gives $c\leq \tilde C/\tilde c$. 
Hence, we obtain $c\leq \tilde C/\tilde c$ for $p\geq p_{F}$, contrary to $c>\tilde C/\tilde c$. 
\end{proof}

\begin{proof}[Proof of Corollary \ref{cor:nex}]
Let $\tilde c$ and $\tilde C$ be as in the proof of Corollary \ref{cor:singnex}. 
Fix $c>\tilde C/\tilde c$ and $z_0\in M$. We set 
$u_{0}(x):=\tilde{u}_{0}(r(x))$ with $r(x):=d(z_0,z)$ and 
\[
	\tilde{u}_0(r) := \left\{ 
	\begin{aligned}
	& c r^{-\frac{2}{p-1}} 
	&&\mbox{ if }p>p_F, \\
	& c r^{-N} 
	( \log (e + r^{-1} )^{-\frac{N}{2}-1}
	&&\mbox{ if }p=p_F, 
	\end{aligned}
	\right.
\]
for $x\in B(z_0,\rho_\infty)$ and $u_0(x):=0$ 
for $x\in M\setminus B(z_0,\rho_\infty)$. 
Let $1\leq q<N(p-1)/2$ if $p>p_F$ and $q=1$ if $p=p_F$. 
Since $\tilde u_0$ is decreasing, 
we see that \eqref{ballcheck} with $f$ replaced by $\tilde u_0^{q}$ holds. 
Thus, 
\[
	\|u_0\|_{L_{\uloc,\rho_\infty}^q(M)}
	= \sup_{z\in M} \left( \int_{B(z,\rho_\infty)} u_0^q dV_g \right)^\frac{1}{q} 
	\leq 
	C(N)\left( \int_{0}^{\rho_\infty} r^{N-1}\tilde{u}_0^q(r) dr
	\right)^\frac{1}{q} 
	<\infty, 
\]
so that $u_0\in L_{\uloc,\rho_\infty}^q(M)$. 
On the other hand, by $c>\tilde C/\tilde c$, 
the proof of Corollary \ref{cor:singnex} shows that 
\eqref{eq:fun} does not admit any local-in-time solutions. 
Hence $u_0$ is the desired function. 
\end{proof}

\begin{proof}[Proof of Corollary \ref{cor:singex}]
Let $\eta\in C^\infty(\R)$ be as in Subsection \ref{subsec:cri}. 
Since $u_0\in L^\infty(M\setminus B(z_0,\rho_\infty))$ 
if $M\setminus B(z_0,\rho_\infty)\neq \emptyset$, 
we can assume that $u_0(x) \leq c f(x) + \tilde C$ for $x\in M$, 
where $\tilde C\geq 0$ 
is determined by $u_0$, and $f$ is given by 
\[
	f(x):= 
	\left\{ 
	\begin{aligned}
	& d(z_0,x)^{-\frac{2}{p-1}}\eta( \rho_\infty^{-1} d(z_0,x) ) 
	&&\mbox{ if }p>p_F, \\
	& d(z_0,x)^{-N} ( \log (e + d(z_0,x)^{-1} )^{-\frac{N}{2}-1}
	\eta( \rho_\infty^{-1} d(z_0,x)) 
	&&\mbox{ if }p=p_F. 
	\end{aligned}
	\right.
\]
Define $h$ by \eqref{eq:hdefff}. Set 
\[
	\overline{u}(x,t):= 2c U(x,t) + 2\tilde C, \qquad 
	U(x,t) := h^{-1} \left( \int_M K(x,y,t) h( f(y) )  dV_g(y) \right). 
\]
From $\int_M K(x,y,t) dV_g(y)=1$, it follows that 
\[
\begin{aligned}
	\Psi[\overline{u}] &\leq c U + \tilde C
	+ C c^p \int_0^t \int_M K(x,y,t-s) U(y,s)^p dV_g(y) ds 
	+ C \tilde C^p t \\
	&\leq 
	(c+C c^p J(t)) U + \tilde C + C \tilde C^p t, 
\end{aligned}
\]
where $J$ is given by \eqref{eq:ouestij}. 
We estimate $J$ as follows. 
By computations similar to the derivation of \eqref{eq:hfygcom} 
(with $\rho_T$ replaced by $1$), we have 
\[
\begin{aligned}
	&h(U(x,t)) 
	&\leq 
	\left\{ 
	\begin{aligned}
	&C t^{-\frac{\alpha}{p-1}}  
	&&\mbox{ if }p>p_F,  \\
	&C ( 1+t^\frac{1}{2} ) t^{-\frac{N}{2}} 
	(\log(e+ r^{-1}))^{-\frac{N}{2}+\beta} 
	&&\mbox{ if }p=p_F. 
	\end{aligned}
	\right.
\end{aligned}
\]
Recall \eqref{eq:hinves}. Then we obtain 
\[
\begin{aligned}
	U(x,t) &\leq 
	\left\{
	\begin{aligned}
	& C t^{-\frac{1}{p-1}} && \mbox{ if }p>p_F,   \\
	& C( 1+t^\frac{1}{2} ) t^{-\frac{N}{2}} 
	(\log(e+t^{-\frac{1}{2}}))^{-\frac{N}{2}} 
	&& \mbox{ if }p=p_F,  \\
	\end{aligned}
	\right.
\end{aligned}
\]
for $(x,t)\in M\times (0,\rho_\infty^2)$, 
where $C>0$ depends only on $N$ and $p$. 
Similar computations to \eqref{eq:Jtesti} show that 
\[
\begin{aligned}
	J(t) &\leq 
	\left\{ 
	\begin{aligned}
	& C t^{-\frac{\alpha-1}{p-1}} 
	\int_0^t s^{-\frac{p-\alpha}{p-1}} ds 
	&& (p>p_F) \\
	& 
	C (\log(e + t^{-\frac{1}{2}} ))^\beta (1+t^\frac{1}{2})^{p-1}
	\int_0^t  s^{-1} 
	(\log(e+ s^{-\frac{1}{2}}))^{-1-\beta} ds 
	&& (p=p_F) 
	\end{aligned}
	\right. \\
	&\leq 
	\left\{ 
	\begin{aligned}
	& C  && (p>p_F) \\
	& C (1+t^\frac{1}{2})^p && (p=p_F)
	\end{aligned}
	\right. 
\end{aligned}
\]
for $0<t<\rho_\infty^2$, 
where $C>0$ depends only on $N$ and $p$.

Let $0<T<\rho_\infty^2$. Then, 
\[
\begin{aligned}
	\Psi[\overline{u}] 
	&\leq 
	\left\{ 
	\begin{aligned}
	&(c+C c^p ) U + \tilde C + C \tilde C^p T
	&&\mbox{ if }p>p_F, \\
	&(c+C c^p (1+T^\frac{1}{2})^p ) U + \tilde C + C \tilde C^p T 
	&&\mbox{ if }p=p_F, \\
	\end{aligned}
	\right.
\end{aligned}
\]
for $0<t<T$. 
By choosing $T\leq \min\{ 1, C^{-1} \tilde C^{-(p-1)} \}$ 
and $c\leq 2^{-p/(p-1)} C^{-1/(p-1)}$, 
we see that $\overline{u}$ is a supersolution 
of \eqref{eq:integ} in $M\times [0,T)$. 
We remark that the choice of $c$ depends only on $N$ and $p$. 
The rest of the proof is the same as that of Theorem \ref{th:necsha}. 
\end{proof}

\appendix
\section{Comparison theorems}\label{sec:app-comp}
In this appendix, we prepare estimates which are needed in this paper and strongly related to the so-called comparison theorems in Riemannian geometry. 
Those include bounds of the Laplacian of the distance function, Jacobi fields, the volume of balls and the heat kernel. 
In this section, we assume that $(M,g)$ is a connected $N$-dimensional complete Riemannian manifold with $0<\inj(M)\leq\infty$ and $|\sec(M)| \leq \kappa$ 
for some $0\leq \kappa<\infty$.

\begin{remark}
We remark that the bound of $\sec(M)$ is not imposed for $N=1$, 
since the sectional curvature is not defined. 
It is clear that a 1-dimensional connected manifold 
is diffeomorphic to $\R^{1}$ or $\bS^{1}$. 
Moreover, by a suitable coordinate change, 
we can say that any 1-dimensional connected complete Riemannian manifold is 
isometric to $\R^{1}$ with the standard metric $g_{\R^{1}}$ 
or $\bS^{1}$ with the standard metric $g_{\bS^{1}}$. 
This means that it suffices to consider 
$(\R^{1},g_{\R^{1}})$ 
or $(\bS^{1},g_{\bS^{1}})$ when $N=1$ 
in the category of connected complete Riemannian manifolds. 
Thus, in the following, we implicitly assume that 
$g$ is $g_{\R^{1}}$ or $g_{\bS^{1}}$ in the case $N=1$. 
\end{remark}

\subsection{Laplacian comparison}
We start with the Laplacian comparison. 
Fix a point $z_0\in M$ and put $r(x):=d(z_0,x)$ for $x\in M$. 
Then, $r$ is a smooth function on $B(z_0,\inj(M))\setminus \{z_0\}$ with $|\nabla r|=1$. 
Moreover, since $\Ric\geq (N-1)(-\kappa)$ in the case $N\geq 2$, from the Laplacian comparison theorem (see \cite[Theorem 11.15]{Lebook} for instance), 
it follows that 
\[
	\Delta r (x)\leq 
	\left\{ 
	\begin{aligned}
	&(N-1)r^{-1}(x)  &&\mbox{ if }\kappa=0,  \\
	&(N-1)\sqrt{\kappa} \coth (\sqrt{\kappa} r(x))  
	&& \mbox{ if }\kappa>0, 
	\end{aligned}
	\right. 
\]
for $x\in B(z_0,\inj(M))$ with $0<r(x)<\pi/\sqrt{\kappa}$. 
Since $\sqrt{\kappa} \coth (\sqrt{\kappa} r)\leq 2r^{-1}$ for $r\in(0,1/\sqrt{\kappa})$ if $\kappa>0$, 
we have 
\begin{equation}\label{prep1}
	\Delta r (x)\leq 2(N-1)r^{-1}(x)
\end{equation}
for $x\in B(z_0,\inj(M))$ with $0<r(x)<\pi/(4\sqrt{\kappa})$. 
This also holds for $x\in B(z_0,\inj(M))$ in the case $N=1$, since $\Delta r(x)=0$.

\subsection{Volume comparison}
We explain the upper and lower bound of the volume form. 
Fix a point $z_0\in M$ and put $r(x):=d(z_0,x)$ for $x\in M$. 
Let $(e_{1},\dots,e_{N})$ be an orthonormal basis of $T_{z_0}M$. 
Then, the map which sends $(x_{1},\dots,x_{N})\in \R^N$ to $\exp_{z_0}(x_{1}e_{1}+\dots+x_{N}e_{N})\in M$ gives 
Riemannian normal coordinates centered at $z_0$ on $B(z_0,\inj(M))$, where $\exp_{z_0}:T_{z_0}M\to M$ is the exponential map at $z_0$. 
Denote by $g_{ij}$ the coefficients of $g$ with respect to this normal coordinates. 
For a while, assume $N\geq 2$. 
Then, by the proof of \cite[Theorem 11.16]{Lebook}, we have 
\[
	\sqrt{\det(g_{ij}(x))}
	\leq
	\left\{ 
	\begin{aligned}
	&1 &&\mbox{ if }\kappa=0,  \\
	&(\sqrt{\kappa} r(x))^{1-N} \sinh^{N-1}(\sqrt{\kappa}r(x)) 
	&&\mbox{ if }\kappa>0, 
	\end{aligned}
	\right.
\]
for $x\in B(z_0,\inj(M))$, where we used $\Ric\geq (N-1)(-\kappa)$. 
Since $(\sqrt{\kappa} r)^{1-N} \sinh^{N-1}(\sqrt{\kappa}r)\leq 2^{N-1}$ for $r\in[0,1/\sqrt{\kappa})$ if $\kappa>0$, 
we have 
\begin{equation}\label{prep2}
	\sqrt{\det(g_{ij}(x))}\leq 2^{N-1}
\end{equation}
for $x\in B(z_0,\inj(M))$ with $0\leq r(x)<\pi/(4\sqrt{\kappa})$. 
This also holds for $x\in B(z_0,\inj(M))$ in the case $N=1$, since $g_{11}=1$. 
Especially, we have 
\begin{equation}\label{prep3}
	\Vol(B(x,r))\leq 2^{N-1}N^{-1}\Area(\bS^{N-1})r^{N}
\end{equation}
for all $x\in M$ and $0\leq r < \min\{\inj(M),\pi/(4\sqrt{\kappa})\}$ when $N\geq 2$, and for all $x\in M$ and $0\leq r < \inj(M)$ when $N=1$. 
Here, $\Area(\bS^{N-1})$ is the area of the unit sphere in $\R^{N}$. 

We can also derive the lower bound of the volume form, see for instance the proof of \cite[Theorem 11.14]{Lebook}. 
From $\sec(M)\leq \kappa$ in the case $N\geq 2$, it follows that 
\[
	\sqrt{\det(g_{ij}(x))}
	\geq 
	\left\{ 
	\begin{aligned}
	&1 &&\mbox{ if }\kappa=0,  \\
	&(\sqrt{\kappa} r(x))^{1-N} \sin^{N-1}(\sqrt{\kappa} r(x)) &&
	\mbox{ if }\kappa>0, 
	\end{aligned}
	\right. 
\]
for $x\in B(z_0,\inj(M))$ with $0\leq r(x)<\pi/\sqrt{\kappa}$. 
Since $(\sqrt{\kappa}r)^{1-N} \sin^{N-1}(\sqrt{\kappa}r)\geq (1/2)^{N-1}$ for $r\in[0,\pi/(2\sqrt{\kappa}))$ if $\kappa>0$, 
we have 
\begin{equation}\label{prep4}
	\sqrt{\det(g_{ij}(x))}\geq 2^{1-N}
\end{equation}
for $x\in B(z_0,\inj(M))$ with $0\leq r(x)<\pi/(4\sqrt{\kappa})$, 
and this also holds for $x\in B(z_0,\inj(M))$ in the case $N=1$, since $g_{11}=1$. 
Especially, we have 
\begin{equation}\label{prep5}
	\Vol(B(x,r))\geq 2^{1-N}N^{-1}\Area(\bS^{N-1})r^{N}
\end{equation}
for all $x\in M$ and $0\leq r < \min\{\inj(M),\pi/(4\sqrt{\kappa})\}$
when $N\geq 2$, and for all $x\in M$ and $0\leq r < \inj(M)$ when $N=1$ 

\subsection{Jacobi field comparison}
Based on the estimates of Jacobi fields, 
we derive the upper and lower bound of the derivative of the exponential map and deduce some applications. 
Fix a point $z_0\in M$ and put $r(x):=d(z_0,x)$ for $x\in M$. 
Let $(e_{1},\dots,e_{N})$ be an orthonormal basis of $T_{z_0}M$. 
Denote by $\varphi:B(z_0,\inj(M))\to B(O,\inj(M))\subset \R^{N}$ 
the inverse of the map which sends $(x_{1},\dots,x_{N})\in \R^N$ to $\exp_{z_0}(x_{1}e_{1}+\dots+x_{N}e_{N})\in M$. 
Then, in the case $N\geq2$, the Rauch comparison theorem, see \cite[Corollary 5.6.1]{Jostbook} for instance, implies that 
\[
	|D_{\eta}\exp_{z_0}(w)|^2\leq
	\left\{
	\begin{aligned}
	&|w_{1}|^2+|w_{2}|^2 &&\mbox{ if }\kappa=0, \\
	&|w_{1}|^2+|w_{2}|^2(\sqrt{\kappa}|\eta|)^{-2} \sinh^{2}(\sqrt{\kappa}|\eta|) &&\mbox{ if }\kappa>0, 
	\end{aligned}
	 \right.
\]
for $\eta\in T_{z_0}M$, 
where $w_{1}$ and $w_{2}$ are vectors in $T_{z_0}M$ which are uniquely determined by $w=w_{1}+w_{2}$
with $w_{1}\in \R \eta$ and $w_{2}\in (\R \eta)^{\bot}$. 
Then, similarly as \eqref{prep2}, we can say that 
\[
|D_{\eta}\exp_{z_0}(w)|\leq 2 |w|
\]
for $\eta\in T_{z_0}M$ with $|\eta|\leq\pi/(4\sqrt{\kappa})$. 
As a corollary of this estimate, we can compare the distance on $M$ and $\R^{N}$. 
Let $a,b\in B(z_0,\rho_{\infty})$, where $\rho_{\infty}$ 
is given by \eqref{eq:rhopr0def}. 
Put $a':=\exp_{z_0}^{-1}(a)$ and $b':=\exp_{z_0}^{-1}(b)$. 
Then, $c(t):=\exp_{z_0}((1-t)a'+tb')$ is a curve in $B(z_0,\rho_{\infty})$ 
joining $a$ and $b$. 
Thus, we have 
\begin{equation}\label{prep6}
d(a,b)\leq \int_{0}^{1}|\dot{c}(t)|dt=\int_{0}^{1}|D_{c(t)}\exp_{z_0}(b'-a')|dt\leq 2|b'-a'|. 
\end{equation}
We remark that this also holds for $N=1$. 

Similarly, we can estimate the distance from below. 
The upper bound of the sectional curvature with the Rauch comparison theorem (see \cite[Corollary 5.6.1]{Jostbook}) 
implies that 
\[
	|D_{\eta}\exp_{z_0}(w)|^2\geq
	\left\{
	\begin{aligned}
	&|w_{1}|^2+|w_{2}|^2 &&\mbox{ if }\kappa=0, \\
	&|w_{1}|^2+|w_{2}|^2(\sqrt{\kappa}|\eta|)^{-2} \sin^{2}(\sqrt{\kappa}|\eta|)&&\mbox{ if }\kappa>0, 
	\end{aligned}
	 \right.
\]
for $\eta\in T_{z_0}M$ with $|\eta|\leq\pi/\sqrt{\kappa}$ in the case $N\geq 2$, 
where $w_{1}$ and $w_{2}$ are as above. 
Then, similarly as \eqref{prep4}, we can say that 
\[
|D_{\eta}\exp_{z_0}(w)|\geq 2^{-1} |w|
\]
for $\eta\in T_{z_0}M$ with $|\eta|\leq7\pi/(12\sqrt{\kappa})$. 
By this estimate, we can also bound $d(a,b)$ from below. 
Fix $a,b\in B(z_0,\rho_{\infty})$ and let $c(t)$ 
($t\in[0,1]$) be the geodesic from $a$ to $b$. 
Then, we see that $c(t)\in B(z_0,7\rho_{\infty}/3)$ for all $t\in[0,1]$. 
Actually, if not, there exists $t'\in(0,1)$ such that 
$c(t')\notin B(z_0,7\rho_{\infty}/3)$. 
Then, $d(a,c(t'))>4\rho_{\infty}/3$ and $d(b,c(t'))>4\rho_{\infty}/3$, 
and this implies that the length of $c(t)$ 
is bigger than $8\rho_{\infty}/3$. 
This contradicts the fact that the length of $c(t)$ is now $d(a,b)$ 
and it is smaller than $2\rho_{\infty}$. 
We remark that $7\rho_{\infty}/3<\inj(M)$. 
Put $\xi(t):=\exp_{z_0}^{-1}(c(t))\in B(O,7\rho_{\infty}/3)$. 
Then, we have $c(t)=\exp_{z_0}(\xi(t))$ and 
\begin{equation}\label{prep7}
d(a,b)=\int_{0}^{1}|\dot{c}(t)|dt=\int_{0}^{1}|D_{\xi(t)}\exp_{z_0}(\dot{\xi}(t))|dt
\geq 2^{-1}\int_{0}^{1}|\dot{\xi}(t)|dt\geq 2^{-1}|b'-a'|, 
\end{equation}
where $a':=\exp_{z_0}^{-1}(a)$ and $b':=\exp_{z_0}^{-1}(b)$. 
We remark that this also holds  for $N=1$. 

\subsection{Estimates of the heat kernel}
Let $K(x,y,t)$ be the heat kernel on $(M,g)$. 
Thanks to many established works on the heat kernel on a Riemannian manifold (for instance Cheeger and Yau \cite{CY81} and Li and Yau \cite{LY86}), 
we can assume that $K(x,y,t)$ is similar to the heat kernel on the Euclidean space especially in the case where $d(x,y)$ and $t$ are very small. 
We explain those estimates necessary in this paper. 

First, we intorduce a direct corollary form \cite[THEOREM 2.2]{LY86}. 
For a while, assume $N\geq 2$. 
Since $\sec(M)\geq -\kappa$ implies $\Ric\geq (N-1)(-\kappa)$, we can use it. 
Set $q=\gamma=\gamma_{0}=\theta=0$, $t_1=t$, $t_2=2t$ 
and $\alpha=3/2$ in that theorem. 
Then, we can say that there exists a constant $C>0$ depending only on $N$ such that 
\[
	K(x,y,t) \leq 
	\left\{ 
	\begin{aligned}
	& 2^{3N/4}\exp\left(\frac{3 d(y,z)^2}{8t}\right) 
	K(x,z,2t) &&\mbox{ if }\kappa =0, \\
	& 2^{3N/4}\exp\left(C(N-1)\kappa t+\frac{3 d(y,z)^2}{8t}\right)
	K(x,z,2t) &&\mbox{ if }\kappa>0, \\
	\end{aligned}
	\right. 
\]
for all $x,y,z\in M$ and $t\in(0,\infty)$. 
Thus, if $d(y,z)\leq \sqrt{t}$ and $t<\pi^2/\kappa$, 
we can say that there exists a constant $C>0$ depending only on $N$ such that 
\begin{equation}\label{prep8}
K(x,y,t) \leq C K(x,z,2t).
\end{equation}
This also holds for $d(y,z)\leq \sqrt{t}$ when $N=1$ by the explicit formula of the heat kernel on $(\R^{1},g_{\R^{1}})$ or $(\bS^{1},g_{\bS^{1}})$. 

Next, we introduce an upper bound of $K(x,y,t)$ by $C_{1}t^{-N/2}e^{-d(x,y)^2/(C_{2}t)}$ as an application of \cite[COROLLARY 3.1]{LY86}.
They proved that, when $N\geq2$, if the Ricci curvature is bounded from below by $-A$, for some $A\geq 0$, 
then for $1<\alpha<2$ and $0<\varepsilon<1$, the heat kernel satisfies 
\[K(x,y,t)\leq \left(\frac{C(\varepsilon)^{2\alpha}}{\Vol(B(x,\sqrt{t}))\Vol(B(y,\sqrt{t}))}\right)^{\frac{1}{2}}\exp\left(\frac{C(N)\varepsilon At}{\alpha-1}-\frac{d(x,y)^2}{(4+\varepsilon)t}\right). \]
Thus, we can apply the above upper bound with $A=(N-1)\kappa$. 
Then, letting $\alpha=3/2$ and $\varepsilon=1/2$ and applying the volume comparison 
theorem \eqref{prep5}, 
we can say that the heat kernel satisfies 
\begin{equation}\label{prep9}
K(x,y,t)\leq Ct^{-N/2}\exp\left(-\frac{d(x,y)^2}{(4+(1/2))t}\right)
\end{equation}
for $0<\sqrt{t}<\min\{\inj(M),\pi/(4\sqrt{\kappa})\}$, 
where $C>0$ depends only on $N$, $\kappa$ and $\inj(M)$. 
Of course, we can check that this holds for $0<\sqrt{t}<\inj(M)$ when $N=1$ by the explicit formula of the heat kernel. 

A lower bound of $K(x,y,t)$ can be induced from a result of Cheeger and Yau \cite{CY81}. 
Put 
\[
	\lambda:=\left\{ 
	\begin{aligned}
	&-1 && \mbox{ if } \kappa=0, \\
	&-\kappa && \mbox{ if } \kappa> 0, 
	\end{aligned}
	\right.
\]
for $N\geq 2$. 
Then, $\lambda< 0$ and $\Ric\geq (N-1)\lambda$. 
By \cite[Theorem 3.1]{CY81} (see also THEOREM 7 of Section 4 of Chapter VIII in \cite{Chbook}), 
we can say that the heat kernel satisfies 
\[K(x,y,t)\geq K_{\lambda}(d(x,y),t)\]
for all $x,y\in M$ and $t\in(0,\infty)$, where $K_{\lambda}:[0,\infty)\times(0,\infty)\to\mathbb{R}$ 
be a smooth function for $(r,t)$ so that $K_{\lambda}(d_{\lambda}(\tilde{x},\tilde{y}),t)$ 
becomes the heat kernel on the $N$-dimensional simply connected space form of constant sectional curvature $\lambda$ with 
its induced distance function $d_{\lambda}$. 
A lower bound of $K_{\lambda}$ is obtained by Davies and Mandouvalos \cite[Theorem 3.1]{DM88}. 
Even though their estimate is for the case $\lambda=-1$, by scaling appropriately for $\lambda<0$, we can see that 
\[
\begin{aligned}
	K_{\lambda}(r,t) 
	&\geq Ct^{-\frac{N}{2}}
	\exp\left(-\frac{r^2}{4t}+\frac{(N-1)^2\lambda t}{4}
	-\frac{(N-1)\sqrt{-\lambda}r}{2}\right)\\
	&\quad \times \left( 1+ \sqrt{-\lambda}r-\lambda t
	\right)^{\frac{N-1}{2}-1}\left(1+\sqrt{-\lambda}r\right), 
\end{aligned}
\]
where $C>0$ depends only on $N$ and $\kappa$. 
By Young's inequality, we have  
\[
\frac{(N-1)\sqrt{-\lambda}r}{2}=2\frac{r}{2\sqrt{t}}\frac{(N-1)\sqrt{-\lambda t}}{2}\leq \frac{r^2}{4t}+\frac{(N-1)^2(-\lambda)t}{4}. 
\]
Thus, if $t$ is restricted to $[0,T)$ with $0<T<\infty$, we have 
\[\exp\left(-\frac{r^2}{4t}+\frac{(N-1)^2\lambda t}{4}-\frac{(N-1)\sqrt{-\lambda}r}{2}\right)\geq C \exp\left(-\frac{r^2}{2t}\right), \]
where $C>0$ depends only on $N$, $T$ and $\kappa$. 
Also, one can see that, for $t \in [0,T)$ with $0<T<\infty$, 
\[
\left( 1+ \sqrt{-\lambda}r-\lambda t\right)^{\frac{N-1}{2}-1}\left(1+\sqrt{-\lambda}r\right)\geq C
\]
for some $C>0$ depends only on $N$, $T$ and $\kappa$. 
Thus, we see that there exists $C>0$ depending only on $N$, $T$ and $\kappa$ such that for $t\in (0,T)$ and $x,y \in M$ 
the heat kernel satisfies 
\begin{equation}\label{prep10}
K(x,y,t)\geq Ct^{-N/2}\exp\left(-\frac{d(x,y)^2}{2t}\right). 
\end{equation}
Of course, this holds for $N=1$. 

Combining \eqref{prep9} and \eqref{prep10}, we can deduce 
the following scaling property of the heat kernel. 
Assume $N\geq 2$ for a while. 
Fix a base point $z_0\in M$ and an orthonormal basis $(e_{1},\dots,e_{N})$ of $T_{z_0}M$. 
Denote by $\varphi:B(z_0,\inj(M))\to B(O,\inj(M))\subset \R^{N}$ 
the inverse of the map which sends $(x_{1},\dots,x_{N})\in \R^N$ to $\exp_{z_0}(x_{1}e_{1}+\dots+x_{N}e_{N})\in M$. 
Take $\xi,\eta \in B(O,\rho_{\infty})\subset \R^{N}$, $0<\alpha<1$ and $0<\beta<\infty$. Put, for simplicity, 
\[
	\alpha\xi':=\varphi^{-1}(\alpha \xi),
	\qquad 
	\alpha\eta':=\varphi^{-1}(\alpha \eta),
	\qquad \xi':=\varphi^{-1}(\xi),
	\qquad \eta':=\varphi^{-1}(\eta). 
\]
Fix $0<T<\infty$. Then, by \eqref{prep10} and \eqref{prep6}, 
there exists $C>0$ depending only on $N$, $\beta T$ and $\kappa$ such that 
\[K(\alpha\xi',\alpha\eta',\beta t)\geq C(\beta t)^{-N/2}\exp\left(-\frac{\alpha^2|\xi-\eta|^2}{\beta t}\right)\]
for all $t\in(0,\beta T)$. 
On the other hand, by \eqref{prep9} and \eqref{prep7}, there exists $C'>0$ depending only on $N$, $\kappa$ and $\inj(M)$ such that 
\[
	K(\xi',\eta',t)\leq C't^{-N/2}
	\exp\left(-\frac{|\xi-\eta|^2}{2(4+(1/2))t}\right)
\]
for $0<\sqrt{t}<\min\{\inj(M),\pi/(4\sqrt{ \kappa})\}$. 
Comparing the above two inequalities, we can see that 
there exists $C>0$ depending 
only on $N$, $\kappa$, $\inj(M)$, $\alpha$, $\beta$ and $T$ such that 
\begin{equation}\label{prep11}
K(\alpha\xi',\alpha\eta',\beta t)\geq C K\left(\xi',\eta',\frac{\beta}{9\alpha^2}t\right)
\end{equation}
for $0<\sqrt{t}<\min\{\inj(M),\pi/(4\sqrt{\kappa}),\sqrt{\beta T}\}$. 
Without the above argument, we can directly prove \eqref{prep11} 
in the case $N=1$ by the explicit formula of the heat kernel. 

\section{Covering theorems}\label{sec:appa}
Int this appendix, 
we prove two facts on Riemannian manifolds. 
One is a kind of the Besicovitch covering theorem and 
the other is an upper bound of the number which 
we need to cover a ball with half balls. 
Recall that $(M,g)$ is a connected $N$-dimensional 
complete Riemannian manifold 
with $\inj(M)>0$. 
We also assume $|\sec(M)| \leq \kappa$ for some $0\leq \kappa < \infty$ when $N\geq 2$. 

Roughly speaking, 
by \cite[Theorem 2.8.14]{Febook}, 
the Besicovitch covering theorem 
holds if a metric space $(X,d)$ is $(\xi,\eta,\zeta)$-limited 
in the sense of \cite[Definition 2.8.9]{Febook}. 
First, we introduce the definition of a $(\xi,\eta,\zeta)$-limited space. 
Let $(X,d)$ be a metric space. 
Fix $A\subset X$, $\xi>0$ and $0<\eta\leq 1/3$. 
For $a\in A$ and $B\subset B(a,\xi)\setminus \{a\}$ with $\#B\geq2$, we define 
\begin{equation}\label{width-func}
	\Width_{a,B}(b,c):=
	\inf\left\{\frac{d(x,c)}{d(a,c)}; 
	\begin{aligned}
	&x\in X, d(a,x)=d(a,c), \\
	&d(a,x)+d(x,b)=d(a,b)
	\end{aligned}\right\}
\end{equation}
for all $b,c\in B$ with $b\neq c$ and $d(a,b)\geq d(a,c)$. 

\begin{definition}
For $\zeta\in\mathbb{N}$, 
$(X,d)$ is directionally $(\xi,\eta,\zeta)$-limited at $A$ if 
\[
	\sup\{\#B; a \mbox{ and } B \mbox{ satisfy } \Width_{a,B}\geq \eta\}
	\leq \zeta. \]
\end{definition}

\begin{theorem}[{\cite[Theorem 2.8.14]{Febook}}]\label{Besicovitch}
Assume that $(X,d)$ is directionally $(\xi,\eta,\zeta)$-limited at $A$. 
If $\mathcal{F}$ is a family of closed balls 
with radii less than $\mu$ with $\mu<\xi/2$ and 
each point of $A$ is the center of some member of $\mathcal{F}$, then 
$A$ is contained in the union of $2\zeta+1$ disjointed subfamilies of $\mathcal{F}$.
\end{theorem}

For our use, it suffices to prove that $(M,g)$ is directionally $(\xi,\eta,\zeta)$-limited at $M$ for some $(\xi,\eta,\zeta)$. 
Take $\xi\in (0,\inj(M))$. 
Fix $a\in M$ and $B\subset B(a,\xi)\setminus \{a\}$ with $\#B\geq2$. 
We denote by $UT_{a}M:=\{v\in T_{a}M; |v|_{g}=1\}$ the unit sphere in $T_{a}M$ 
and define a map $\theta:B(a,\inj(M))\setminus \{a\}\to UT_{a}M$ by $\theta(b):=\exp_{a}^{-1}(b)/|\exp_{a}^{-1}(b)|$. 

\begin{lemma}\label{theta-inj}
If $\Width_{a,B}>0$, then $\theta$ restricted to $B$ is injective. 
\end{lemma}
\begin{proof}
If not, there exist $b,c\in B$ with $b\neq c$ such that $\theta(b)=\theta(c)$. 
Assume $d(a,b)\geq d(b,c)$ if necessary. 
Then $x:=c\in M$ satisfies $d(a,x)=d(a,c)$ and $d(a,x)+d(x,b)=d(a,b)$, since 
$b$ and $c$ are on the same unit speed shortest geodesic $\gamma$ emerging from $a$ with $\gamma'(0)=\theta(b)(=\theta(c))$. 
Then, the right-hand side of \eqref{width-func} should be $0$ and this contradicts $\Width_{a,B}>0$. 
\end{proof}

Considering $\theta(B)$ as a subset of the metric space $T_{a}M$ with metric $g_{a}$, 
we can also define $\Width_{0,\theta(B)}$ for $\theta(B)$. 
Put a positive constant 
\begin{equation}\label{def_of_A}
	A:= 
	\left\{ 
	\begin{aligned}
	&2 && \mbox{ if }  \kappa= 0 \mbox{ or } N=1,\\
	&2(\sqrt{\kappa}\xi)^{-1}\sinh(\sqrt{\kappa}\xi) && \mbox{ if } \kappa> 0. 
	\end{aligned}
	\right.
\end{equation}
We remark that $A\to \infty$ as $\sqrt{\kappa}\xi\to \infty$ if $\kappa>0$. 

\begin{lemma}\label{width-width}
If $\Width_{a,B}\geq \eta$ for some $0<\eta<\infty$, then 
$\Width_{0,\theta(B)}\geq A^{-1}\eta$.
\end{lemma}
\begin{proof}
Take $\tilde{b},\tilde{c}\in \theta(B)$ with $\tilde{b}\neq \tilde{c}$. 
Note that $|\tilde{b}|=|\tilde{c}|=1$. 
There exist $b,c\in B$ with $b\neq c$ such that $\tilde{b}=\theta(b)$ and $\tilde{c}=\theta(c)$. 
We remark that $\tilde{c}=\theta(c)$ means $c=\exp_{a}(d(a,c)\tilde{c})$, since $|\exp_{a}^{-1}c|=d(a,c)$. 
We suppose $d(a,b)\geq d(a,c)$ if necessary. 
Since 
\[
	\{\tilde{x}\in T_{a}M; |\tilde{x}|=1, 
	|\tilde{x}|+|\tilde{x}-\tilde{b}|=|\tilde{b}|\}=\{ \tilde{b}\}, 
\]
we have 
\[\Width_{0,\theta(B)}(\tilde{b},\tilde{c})=\frac{|\tilde{b}-\tilde{c}|}{|\tilde{c}|}=|\tilde{b}-\tilde{c}|. \]
On the other hand, $x:=\exp_{a}(d(a,c)\tilde{b})\in M$ satisfies $d(a,x)=d(a,c)$ and $d(a,x)+d(x,b)=d(a,b)$. 
Thus, $d(x,c)/d(a,c)\geq \Width_{a,B}\geq \eta$. 
It suffices to estimate $d(x,c)/d(a,c)$ from above by $A|\tilde{b}-\tilde{c}|$. 
Then, by the distance comparison theorem \eqref{prep6} with replacing $2$ by $A$, we have 
\[d(x,c)\leq A|d(a,c)\tilde{b}-d(a,c)\tilde{c}|\leq A|\tilde{b}-\tilde{c}|d(a,c). \]
Thus, $\Width_{0,\theta(B)}(\tilde{b},\tilde{c})\geq A^{-1}\eta$ 
and the proof is complete. 
\end{proof}

\begin{definition}
Let $\bS^{N-1}$ be a unit sphere in $\R^{N}$. 
For $w\in (0,\infty)$, define $\Dis(N,w)$ 
as the positive integer such that the following property holds: 
There exists $P\subset \bS^{N-1}$ with $\#P \leq \Dis(N,w)$ such that 
$|p_{1}-p_{2}|\geq w$ for all $p_{1},p_{2}\in P$ ($p_{1}\neq p_{2}$), 
and there does not exist $P\subset \bS^{N-1}$ with $\#P\geq \Dis(N,w)+1$ 
such that 
$|p_{1}-p_{2}|\geq w$ for all $p_{1},p_{2}\in P$ ($p_{1}\neq p_{2}$). 
We remark that $\Dis(1,w)=2$ for $w\in(0,2]$. 
\end{definition}

Roughly speaking, $\Dis(N,w)$ is the maximum number of points 
distributed on the unit sphere keeping distance over $w$. 
One can see that $\Dis(N,w)\to \infty$ as $w\to 0$. 

\begin{proposition}
A connected $N$-dimensional complete Riemannian manifold $(M,g)$ with $\inj(M)>0$ satisfying $|\sec(M)|\leq\kappa$ for some $0\leq \kappa<\infty$ when $N\geq 2$
is directionally $(\xi,\eta,\zeta)$-limited at $M$ for $\xi\in(0,\inj(M))$, $\eta\in(0,1/3]$ and $\zeta\geq\Dis(N,A^{-1}\eta)$, 
where $A$ is defined by \eqref{def_of_A}. 
\end{proposition}
\begin{proof}
Fix $a\in M$ and $B\subset B(a,\xi)\setminus\{a\}$ with $\#B\geq2$. 
Assume that $\Width_{a,B}\geq \eta$. 
Then, by Lemma \ref{width-width}, we have $\Width_{0,\theta(B)}\geq A^{-1}\eta$. 
Thus, $\tilde{b},\tilde{c}\in\theta(B)$ with $\tilde{b}\neq \tilde{c}$ satisfy $|\tilde{b}-\tilde{c}|\geq A^{-1}\eta$, 
and so $\#\theta(B)\leq \Dis(N,A^{-1}\eta)$. 
Since $\#\theta(B)=\#B$ by Lemma \ref{theta-inj}, the proof is complete. 
\end{proof}

This proposition together with Theorem \ref{Besicovitch} 
shows the following:  

\begin{corollary}\label{besi}
Let $(M,g)$ be a connected $N$-dimensional complete Riemannian manifold 
with $\inj(M)>0$ satisfying $|\sec(M)|\leq \kappa$ for some $0\leq \kappa<\infty$ when $N\geq2$. 
For $\xi\in(0,\inj(M)/2)$, 
there exists a positive integer $k_{0}$ such that the following holds:  
If $\mathcal{F}$ is a family of closed balls with radii less 
than $\rho$ with $\rho<\xi$ and 
each point of $M$ is the center of some member of $\mathcal{F}$, then 
$M$ is covered by the union of $k_{0}$ disjointed subfamilies of $\mathcal{F}$. 
Here, the constant $k_{0}$ depends only on $N$ if $\kappa= 0$ and on $N$ and $\sqrt{\kappa}\xi$ if $\kappa>0$. 
Moreover, $k_{0}\to \infty$ as $\sqrt{\kappa}\xi\to \infty$. 
\end{corollary}

We next give a bound on how many balls with radii $\rho/2$ 
are necessary to cover a ball with radius $\rho$ for sufficiently small $\rho$. 
Put 
\begin{equation}\label{sp_rad}
	R':=\left\{
	\begin{aligned}
	&\frac{4}{5}\inj(M)&&  \mbox{ if }N=1,\\
	&\frac{4}{5}\min\left\{\inj(M),\frac{\pi}{2\sqrt{\kappa}}\right\}
	&& \mbox{ if }N\geq 2. 
	\end{aligned}
	\right.
\end{equation}
For $a\in M$ and $\rho\in (0,R']$, we set the following property 
(P) for a subset $B\subset B(a,\rho)$ with $\#B\geq2$. 
\[
	\mbox{(P): } B(b_{1},\rho/4)\cap B(b_{2},\rho/4)=\emptyset \quad 
	\mbox{ for all }b_{1},b_{2}\in B \mbox{ with } b_{1}\neq b_{2}. 
\]
Then, we define 
\[\Pac(a,\rho):=\sup\{\#B;  B\subset B(a,\rho)\mbox{ satisfies (P)}\}. \]

\begin{lemma}\label{Pac}
Assume $\Pac(a,\rho)<\infty$. Take $B$ such that $\# B=\Pac(a,\rho)$. 
Then, 
\[B(a,\rho)\subset \bigcup_{b\in B}B(b,\rho/2). \]
\end{lemma}
\begin{proof}
If not, there exists $c\in B(a,\rho)\setminus(\cup_{b\in B}B(b,\rho/2))$. 
Then, one can easily see that $B':=B\cup \{c\}$ also satisfies (P). 
This contradicts the maximality of $\# B$. 
\end{proof}

Thus, it suffices to bound $\Pac(a,\rho)$ from above. 
Fix $B\subset B(a,\rho)$ with $\#B\geq2$ satisfying (P). 
We remark that 
\begin{equation}\label{balls}
	\bigcup_{b\in B}B(b,\rho/4)\subset B(a,5\rho/4). 
\end{equation}
We estimate the volume of $B(b,\rho/4)$ and $B(a,5\rho/4)$ 
from below and above, respectively. 
By the volume comparison theorem from above \eqref{prep3}, we have 
\[
\Vol(B(a,5\rho/4))\leq 2^{N-1}N^{-1}\Area(\bS^{N-1})(5\rho/4)^{N}, 
\]
where we used $5\rho/4\leq \pi/(2\sqrt{\kappa})$ when $N\geq 2$ and $\kappa>0$. 
On the other hand, by the volume comparison theorem from below \eqref{prep5}, we have 
\[
\Vol(B(b,\rho/4))\geq 2^{1-N}N^{-1}\Area(\bS^{N-1})(\rho/4)^{N}, 
\]
where we used $\rho/4\leq \pi/(2\sqrt{\kappa})$ when $N\geq 2$ and $\kappa>0$. 
Since the left-hand side of \eqref{balls} is the disjoint union, 
by taking the volume of the both-hand sides, 
we see that $\# B$ is bounded from above by the ratio of the right-hand sides of the above two inequalities.  
Then, by Lemma \ref{Pac}, we obtain the following: 

\begin{theorem}\label{half_ball}
Let $(M,g)$ be a connected $N$-dimensional complete Riemannian manifold 
with $\inj(M)>0$ satisfying $|\sec(M)|\leq \kappa$ for some $0\leq \kappa<\infty$ when $N\geq2$. 
Then, there exists a number $k_{1}$ depending only on $N$ such that any ball with radius $\rho< R'$ can be covered 
by at most $k_{1}$ balls with radii $\rho/2$, 
where $R'$ is given by \eqref{sp_rad}. 
\end{theorem}


\begin{thebibliography}{00}

\bibitem{AQ04} 
H. Amann, P. Quittner, 
Semilinear parabolic equations involving measures and low regularity data. 
Trans. Amer. Math. Soc. 356 (2004), no. 3, 1045--1119. 


\bibitem{An97}
D. Andreucci, 
Degenerate parabolic equations with initial data measures. 
Trans. Amer. Math. Soc. 349 (1997), no. 10, 3911--3923. 


\bibitem{AD91}
D. Andreucci, E. DiBenedetto, 
On the Cauchy problem and initial traces for a class of evolution equations 
with strongly nonlinear sources.
Ann. Scuola Norm. Sup. Pisa Cl. Sci. 18 (1991), no. 3, 363--441. 

\bibitem{BPT11}
C. Bandle, M. A. Pozio, A. Tesei, 
The Fujita exponent for the Cauchy problem in the hyperbolic space. 
J. Differential Equations 251 (2011), no. 8, 2143--2163. 

\bibitem{BP85} 
P. Baras, M. Pierre, 
Crit\`ere d'existence de solutions positives pour des \'{e}quations
semi-lin\'{e}aires non monotones. 
Ann. Inst. H. Poincar\'{e} Anal. Non Lin\'{e}aire 2 (1985), no. 3, 185--212. 

\bibitem{BC96} 
H. Brezis, T. Cazenave, 
A nonlinear heat equation with singular initial data.
J. Anal. Math. 68 (1996), 277--304. 


\bibitem{Chbook}
I. Chavel, 
Eigenvalues in Riemannian geometry.
Including a chapter by Burton Randol. 
With an appendix by Jozef Dodziuk. 
Pure and Applied Mathematics, 115. Academic Press, Inc., Orlando, FL, 1984.

\bibitem{CY81}
J. Cheeger, S.-T. Yau, 
A lower bound for the heat kernel. 
Comm. Pure Appl. Math. 34 (1981), no. 4, 465--480.

\bibitem{DM88}
E. B. Davies, N. Mandouvalos, 
Heat kernel bounds on hyperbolic space and Kleinian groups. 
Proc. London Math. Soc. 57 (1988), no. 1, 182--208.

\bibitem{Febook}
H. Federer, 
Geometric measure theory.
Die Grundlehren der mathematischen Wissenschaften, 
Band 153 Springer-Verlag New York Inc., New York, 1969. 

\bibitem{FI18}
Y. Fujishima, N. Ioku, 
Existence and nonexistence of solutions for the heat equation 
with a superlinear source term. 
J. Math. Pures Appl. 118 (2018), 128--158. 

\bibitem{GV97}
J. Ginibre, G. Velo, 
The Cauchy problem in local spaces for the complex Ginzburg-Landau equation. 
II. Contraction methods. 
Comm. Math. Phys. 187 (1997), no. 1, 45--79. 

\bibitem{GM20pre}
T. Giraudon, Y. Miyamoto, 
Fractional semilinear heat equations with singular and nondecaying initial data. 
Rev. Mat. Complut. 35 (2022), no. 2, 415--445.


\bibitem{Grbook}
A. Grigor'yan, 
Heat kernel and analysis on manifolds.
AMS/IP Studies in Advanced Mathematics, 47. 
American Mathematical Society, Providence, RI; International Press, Boston, MA, 2009. 

\bibitem{GGPpre}
G. Grillo, M. Giulia, F. Punzo, 
Blow-up versus global existence of solutions for reaction-diffusion 
equations on classes of Riemannian manifolds, 
preprint (arXiv:2112.00125). 



\bibitem{GSX20}
Q. Gu, Y. Sun, J. Xiao, F. Xu, 
Global positive solution to a semi-linear parabolic equation with potential 
on Riemannian manifold. 
Calc. Var. Partial Differential Equations 59 (2020), no. 5, Paper No. 170, 24 pp. 

\bibitem{HI18}
K. Hisa, K. Ishige, 
Existence of solutions for a fractional semilinear parabolic equation 
with singular initial data. 
Nonlinear Anal. 175 (2018), 108--132. 

\bibitem{IKO20}
K. Ishige, T. Kawakami, S. Okabe, 
Existence of solutions for a higher-order semilinear parabolic equation 
with singular initial data. 
Ann. Inst. H. Poincar\'{e} Anal. Non Lin\'{e}aire 37 (2020), no. 5, 1185--1209. 


\bibitem{Jostbook}
J. Jost, 
Riemannian geometry and geometric analysis.
Sixth edition. Universitext. Springer, Heidelberg, 2011. 


\bibitem{KT17}
T. Kan, J. Takahashi, 
Time-dependent singularities in semilinear parabolic equations: 
existence of solutions. 
J. Differential Equations 263 (2017), no. 10, 6384--6426. 



\bibitem{Ka75}
T. Kato, 
The Cauchy problem for quasi-linear symmetric hyperbolic systems.
Arch. Rational Mech. Anal. 58 (1975), no. 3, 181--205. 


\bibitem{Kk89}
K. Kobayasi, 
Semilinear parabolic equations with nonmonotone nonlinearity. 
Mem. Sagami Inst. Tech. 23 (1989), no. 2, 83--99. 

\bibitem{KY94}
H. Kozono, M. Yamazaki, 
Semilinear heat equations and the Navier-Stokes equation 
with distributions in new function spaces as initial data.
Comm. Partial Differential Equations 19 (1994), no. 5-6, 959--1014. 


\bibitem{LRSV16}
R. Laister, J. C. Robinson, M. Sier\.{z}\polhk ega, A. Vidal-L\'{o}pez, 
A complete characterisation of local existence 
for semilinear heat equations in Lebesgue spaces. 
Ann. Inst. H. Poincar\'{e} Anal. Non Lin\'{e}aire 33 (2016), no. 6, 1519--1538. 

\bibitem{Lebook}
J. M. Lee, 
Introduction to Riemannian manifolds.
Graduate Texts in Mathematics, 176. Springer, Cham, 2018.

\bibitem{LY86}
P. Li, S.-T. Yau, 
On the parabolic kernel of the {S}chr\"{o}dinger operator.
Acta Math. 156 (1986), no. 3-4, 153--201. 


\bibitem{MT06}
Y. Maekawa, Y. Terasawa, 
The Navier-Atokes equations with initial data in uniformly local {$L^p$} spaces. 
Differential Integral Equations 19 (2006), no. 4, 369--400. 


\bibitem{MMP17}
P. Mastrolia, D. D. Monticelli, F. Punzo, 
Nonexistence of solutions to parabolic differential inequalities 
with a potential on Riemannian manifolds. 
Math. Ann. 367 (2017), no. 3-4, 929--963. 

\bibitem{Mi20}
Y. Miyamoto, 
A doubly critical semilinear heat equation in the {$L^1$} space. 
J. Evol. Equ. 21 (2021), no. 1, 151--166. 


\bibitem{Niw86}
Y. Niwa, 
Semi-linear heat equations with measures as initial data. 
Doctoral thesis, The University of Tokyo, 1986. 

\bibitem{PS19}
A. L. A. Poh, M. Shimojo, 
Blow-up of radially symmetric solutions 
for a semilinear heat equation on hyperbolic space. 
Rev. Mat. Complut. 32 (2019), no. 3, 655--680.

\bibitem{Pu11}
F. Punzo, 
On well-posedness of semilinear parabolic and elliptic problems 
in the hyperbolic space. 
J. Differential Equations 251 (2011), no. 7, 1972--1989. 

\bibitem{Pu12b}
F. Punzo, 
Blow-up of solutions to semilinear parabolic equations 
on Riemannian manifolds with negative sectional curvature. 
J. Math. Anal. Appl. 387 (2012), no. 2, 815--827. 

\bibitem{Pu12o}
F. Punzo, 
On well-posedness of the semilinear heat equation on the sphere. 
J. Evol. Equ. 12 (2012), no. 3, 571--592. 

\bibitem{Pu14}
F. Punzo, 
Global existence of solutions to the semilinear heat equation 
on Riemannian manifolds with negative sectional curvature. 
Riv. Math. Univ. Parma (N.S.) 5 (2014), no. 1, 113--138. 


\bibitem{Pu19}
F. Punzo, 
Global solutions of semilinear parabolic equations 
on negatively curved Riemannian manifolds. 
J. Geom. Anal. 31 (2021), no. 1, 543--559. 

\bibitem{QS19b}
P. Quittner, Ph. Souplet, 
Superlinear parabolic problems.
Blow-up, global existence and steady states. Second edition. 
Birkh\"{a}user Advanced Texts: Basler Lehrb\"{u}cher.
Birkh\"{a}user/Springer, Cham, 2019.

\bibitem{RS13}
J. C. Robinson, M. Sier\.{z}\polhk ega, 
Supersolutions for a class of semilinear heat equations. 
Rev. Mat. Complut. 26 (2013), no. 2, 341--360. 


\bibitem{SL10}
H. Shang, F. Li, 
On the {C}auchy problem for the evolution {$p$}-{L}aplacian 
equations with gradient term and source and measures as initial data. 
Nonlinear Anal. 72 (2010), no. 7-8, 3396--3411. 


\bibitem{Sibook}
L. Simon, 
Lectures on geometric measure theory.
Proceedings of the Centre for Mathematical Analysis, 
Australian National University, 3. 
Australian National University, Centre for Mathematical Analysis, Canberra, 1983. 


\bibitem{SW03}
Ph. Souplet, F. B. Weissler, 
Regular self-similar solutions of the nonlinear heat equation 
with initial data above the singular steady state. 
Ann. Inst. H. Poincar\'{e} Anal. Non Lin\'{e}aire 20 (2003), no. 2, 213--235. 

\bibitem{Ta16}
J. Takahashi, 
Solvability of a semilinear parabolic equation with measures as initial data. 
Geometric properties for parabolic and elliptic PDE's, 257--276,
Springer Proc. Math. Stat., 176, Springer, Cham, 2016. 

\bibitem{Um21}
H. Umakoshi, 
A semilinear heat equation with initial data 
in negative Sobolev spaces. 
Discrete Contin. Dyn. Syst. Ser. S 14 (2021), no. 2, 745--767. 


\bibitem{We79}
F. B. Weissler, Semilinear evolution equations in Banach spaces. 
J. Functional Analysis 32 (1979), no. 3, 277--296.

\bibitem{We80}
F. B. Weissler, 
Local existence and nonexistence for semilinear parabolic equations in {$L^p$}. 
Indiana Univ. Math. J. 29 (1980), no. 1, 79--102.


\bibitem{Zh99}
Qi S. Zhang, 
Blow-up results for nonlinear parabolic equations on manifolds.
Duke Math. J. 97 (1999), no. 3, 515--539. 

\end{thebibliography}
\end{document}